\documentclass[11pt,onesided,reqno]{amsart}
\usepackage[foot]{amsaddr}
\usepackage{amsmath}
\usepackage{amssymb}
\usepackage{amsthm}
\usepackage{bbm}
\usepackage{bm}
\usepackage{graphicx}
\usepackage{thmtools}
\usepackage{thm-restate}
\usepackage[margin=1in]{geometry}
\usepackage{hyperref}

\newtheorem{theorem}{Theorem}[section]
\newtheorem{proposition}[theorem]{Proposition}
\newtheorem{lemma}[theorem]{Lemma}
\newtheorem{corollary}[theorem]{Corollary}
\theoremstyle{definition}
\newtheorem{definition}[theorem]{Definition}

\newtheorem{remark}[theorem]{Remark}

\newcommand{\R}{\mathbb{R}}
\newcommand{\C}{\mathbb{C}}
\newcommand{\E}{\mathbb{E}}
\newcommand{\cE}{\mathcal{E}}
\renewcommand{\P}{\mathbb{P}}
\newcommand{\eps}{\varepsilon}

\newcommand{\I}{\mathbbm{1}}
\newcommand{\Tr}{\operatorname{Tr}}
\newcommand{\diag}{\operatorname{diag}}
\newcommand{\supp}{\operatorname{supp}}
\newcommand{\N}{\mathcal{N}}

\newcommand{\Id}{\operatorname{Id}}

\renewcommand{\Re}{\operatorname{Re}}
\renewcommand{\Im}{\operatorname{Im}}
\newcommand{\G}{\mathcal{G}}
\newcommand{\Var}{\operatorname{Var}}
\newcommand{\sign}{\operatorname{sign}}
\newcommand{\semi}{\mathrm{sc}}
\newcommand{\MP}{\mathrm{MP}}
\newcommand{\1}{\mathbf{1}}

\allowdisplaybreaks

\title{The Spectral Norm of Random Inner-Product Kernel Matrices}
\author{Zhou Fan$^1$}
\author{Andrea Montanari$^{1,2}$}
\address{$^1$Department of Statistics, Stanford University}
\address{$^2$Department of Electrical Engineering, Stanford University}
\email{zhoufan@stanford.edu, montanari@stanford.edu}
\thanks{ZF is supported by a Hertz Foundation Fellowship and an NDSEG
Fellowship (DoD, Air Force Office of Scientific Research, 32 CFR
168a).
AM is partially supported by NSF grants CCF-1319979 and DMS-1106627 and the AFOSR grant
FA9550-13-1-0036}

\begin{document}
\maketitle
\begin{abstract}
We study an ``inner-product kernel'' random matrix model, whose
empirical spectral distribution was shown by Xiuyuan Cheng and Amit Singer to
converge to a deterministic measure in the large $n$ and $p$ limit. We provide
an interpretation of this limit measure as the additive free convolution of a
semicircle law and a Marcenko-Pastur law. By comparing the tracial moments of
this random matrix to those of a deformed GUE matrix with the same limiting
spectrum, we establish that for odd kernel functions, the spectral norm of this
matrix convergences almost surely to the edge of the limiting spectrum. Our
study is motivated by the analysis of a
covariance thresholding procedure for the statistical detection and estimation
of sparse principal components, and our results characterize the
limit of the largest eigenvalue of the thresholded sample covariance matrix in
the null setting.
\end{abstract}

\section{Introduction}
Let $X \in \R^{p \times n}$ be a random matrix with independent entries of mean
0 and variance 1, and let $\hat{\Sigma}=n^{-1}XX^T$ be the sample covariance.
Define a matrix $K(X) \in \R^{p \times p}$ entrywise as
\begin{equation}\label{eq:K}
K(X)_{ii'}=\begin{cases} \frac{1}{\sqrt{n}}k(\sqrt{n}\hat{\Sigma}_{ii'}) &
i \neq i' \\ 0 & i=i' \end{cases}
\end{equation}
where $k:\R \to \R$ is a (nonlinear) ``kernel'' function. In this paper, we
study the spectral norm $\|K(X)\|$ in the asymptotic regime $n,p \to \infty$
such that $p/n \to \gamma \in (0,\infty)$, when $k$ is a fixed function
independent of $n$ and $p$.

Our study of this model is motivated by the analysis of a covariance
thresholding procedure proposed in \cite{krauthgameretal} and subsequently
analyzed in \cite{deshpandemontanari} for the sparse PCA problem in statistics.
In the simplest setting, this problem may be formulated as follows:

\subsection{Sparse PCA}\label{subsec:sparsePCA}
Consider a data matrix $X \in \R^{p \times n}$ with independent columns
distributed as $\N(0,\Sigma)$, where $\Sigma$ is a $p \times p$ covariance
matrix of the ``spiked model'' form
\begin{equation}\label{eq:spikedmodel}
\Sigma=\Id+\lambda vv^T
\end{equation}
with $\lambda>0$ a constant and $v \in \R^p$ a vector of unit Euclidean norm.
Assume further that $\|v\|_0 \ll p$ where $\|v\|_0$ denotes the number of
nonzero entries of $v$, and (for simplicity of discussion)
that each such nonzero entry equals
$\pm 1/\sqrt{\|v\|_0}$. Based on observing $X$, we would like to detect the
spike (i.e.\ distinguish this from the null model $\Sigma=\Id$)
and to recover the support of $v$ \cite{aminiwainwright,berthetrigollet1}.

As $n,p \to \infty$ with $p/n \to \gamma \in (0,\infty)$, in the
``supercritical'' regime $\lambda>\lambda^*$ where $\lambda^*:=\sqrt{\gamma}$,
the largest eigenvalue $\lambda_{\max}(\hat{\Sigma})$ separates from the bulk,
and the corresponding eigenvector $\hat{v}$ partially aligns with $v$.
Consequently, consistent spike detection and support recovery may be performed
using $\lambda_{\max}(\hat{\Sigma})$ and $\hat{v}$ \cite{krauthgameretal}.
However, in the ``subcritical'' regime $\lambda<\lambda^*$,
$|\hat{v}^Tv| \to 0$ almost surely, 
$\lambda_{\max}(\hat{\Sigma})$ cannot distinguish the null and spiked models,
and furthermore no test using only the eigenvalues of $\hat{\Sigma}$ can
distinguish the models with probability approaching one
\cite{johnstonelu2,baiketal,baiksilverstein,paul,nadler,onatskietal}.
In this regime, \cite{krauthgameretal} proposed to exploit
the sparsity of $v$ by applying a thresholding operation
$x \mapsto k_\tau(\sqrt{n}x)/\sqrt{n}$ entrywise to $\hat{\Sigma}$ to yield a
matrix $M_\tau(X)$, and then performing a spectral decomposition of
$M_\tau(X)$.
Here, $\tau>0$ is a constant and $k_\tau:\R \to \R$ is a threshold
function satisfying $k_\tau(x)/x \to 1$ as $x \to \pm \infty$ and
$k_\tau(x)=0$ for $|x| \leq \tau$, so that entries
of $\hat{\Sigma}$ of magnitude less than $\tau/\sqrt{n}$ are set to 0 while
large entries are essentially preserved. (The matrix $K(X)$ in (\ref{eq:K}) when
$k:=k_\tau$ is precisely $M_\tau(X)$ with diagonal set to 0.)

The choice of threshold level $\tau/\sqrt{n}$ is motivated by the following
consideration: For $\lambda<\lambda^*$, it is in fact conjectured that
no polynomial-time algorithm can consistently detect the spike or recover the
support of $v$ if $\|v\|_0 \gtrsim n^{1/2+\eps}$, for any $\eps>0$
\cite{berthetrigollet2,krauthgameretal}. Hence it is believed that the most
difficult setting which permits a computationally tractable solution to these
problems is when $\|v\|_0 \asymp \sqrt{n}$. In this setting,
both the non-zero off-diagonal entries of $\Sigma$ (the ``signal'')
and the fluctuations of the entries of $\hat{\Sigma}$ (the ``noise'') are
of order $1/\sqrt{n}$, so the threshold must also be of order $1/\sqrt{n}$
to preserve the signal while reducing the noise.

\begin{figure}
\includegraphics[width=0.8\textwidth]{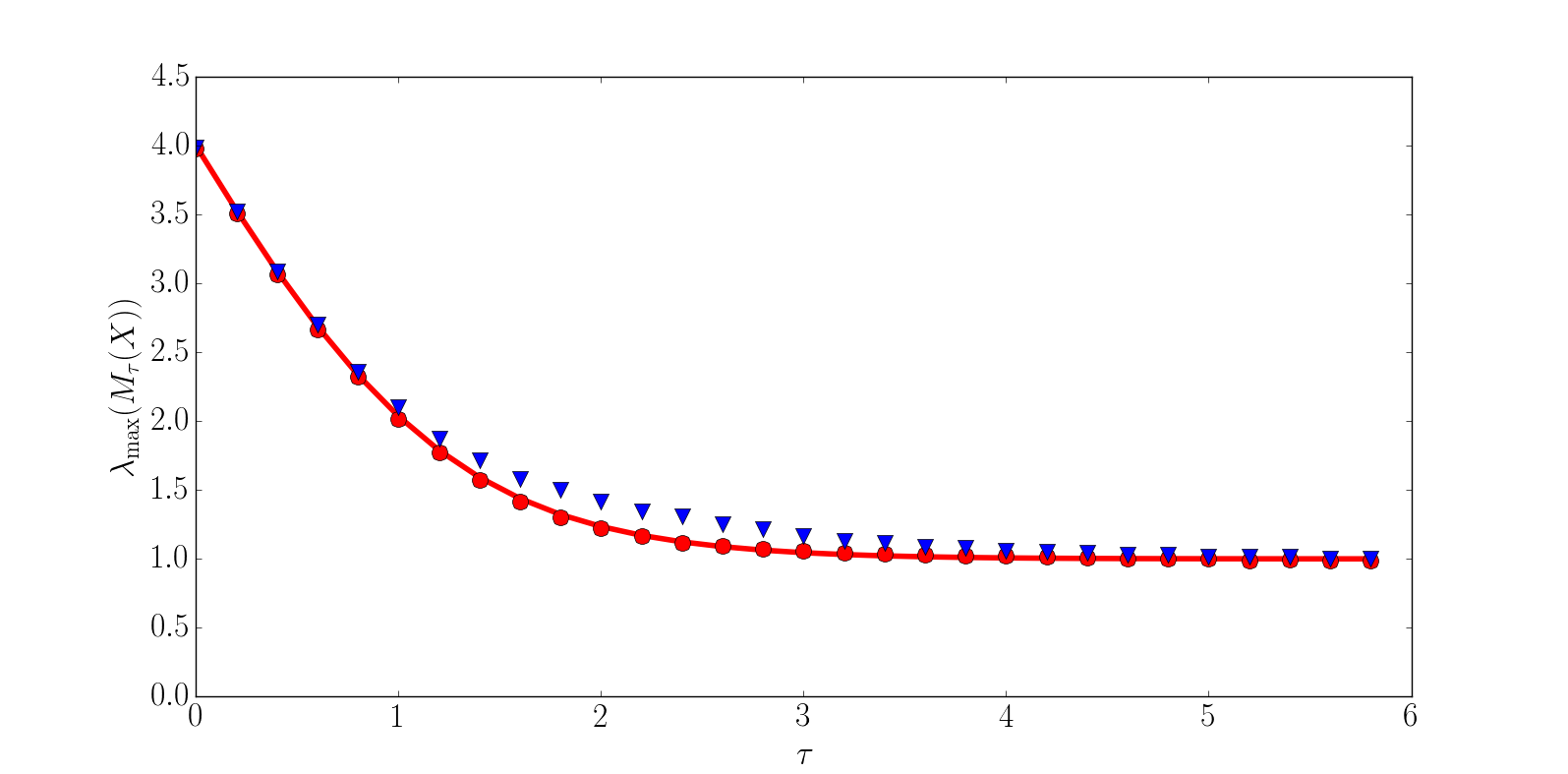}
\caption[SparsePCA Caption]
{Largest eigenvalue of the thresholded covariance matrix $M_\tau(X)$ for
a threshold function $k_\tau$ under $\Sigma=\Id$ (red circles)
and $\Sigma=\Id+\lambda vv^T$ (blue triangles), for $n=p=2000$, $\lambda=0.9$,
and $\|v\|_0=0.3\sqrt{n}$. The asymptotic prediction $\|\mu_{a,\nu,\gamma}\|+1$
is shown as the red curve, where $a:=\E[\xi k_\tau(\xi)]$,
$\nu:=\E[k_\tau(\xi)^2]$, and $\gamma:=p/n=1$.
The threshold function $k_\tau$ here is a smoothed
soft-threhold, defined by $k_\tau(x)=0$ for $|x| \leq 0.8\tau$,
$k_\tau(x)=\sign(x)(|x|-\tau)_+$ for $|x| \geq 1.2\tau$, and quadratic
interpolation in between.\footnotemark}\label{fig:sparsePCA}
\end{figure}

In \cite{deshpandemontanari}, it was shown that for any $\lambda>0$,
spike detection and support recovery based on $\lambda_{\max}(M_\tau(X))$ and
the corresponding eigenvector can
succeed with probability approaching 1 when $\|v\|_0 \leq c\sqrt{n}$, for
some constants $c:=c(\lambda)>0$ and $\tau:=\tau(c,\lambda)>0$. This phenomenon
is illustrated in Figure \ref{fig:sparsePCA},\footnotetext{The proof of our
main result requires a technical condition that $k(x)$ is continuously
differentiable. Oftentimes threshold functions used in practice are not smooth
in this sense,
but the same qualitative phenomena regarding detection and support recovery
should hold for both smooth and non-smooth thresholds.}
which shows that for a range of
thresholds $\tau$, there is a difference between the values of
$\lambda_{\max}(M_\tau(X))$ under the null model $\Sigma=\Id$ and under a
spiked alternative with $\lambda<\lambda^*$ and sparsity
$\|v\|_0 \asymp \sqrt{n}$. The main result of this paper strengthens the
non-asymptotic analysis in \cite{deshpandemontanari} under the null model
$\Sigma=\Id$ to establish an exact asymptotic value for
$\lambda_{\max}(M_\tau(X))$ in terms of $k_\tau$. Procedurally, this indicates
the point above which this method should reject the null
model in favor of a spiked alternative. We are not aware of a similar analytic
characterization of the value of $\lambda_{\max}(M_\tau(X))$ under the
alternative model; such a characterization may yield insight on the
exact critical sparsity level $c^*(\lambda)$ and optimal choices of $\tau$ and
$k_\tau$ for spike detection using this method to succeed.

In the null model of this example, since all diagonal entries of
$\hat{\Sigma}$ concentrate around 1 and thresholding essentially preserves the
diagonal, the thresholded sample covariance
$M_\tau(X)$ satisfies $\|M_{\tau}(X)-(K(X)+\Id)\| \to 0$, where $K(X)$ is as in
(\ref{eq:K}) for $k:=k_\tau$. Hence the largest eigenvalue limit of $M_\tau(X)$
is simply that of $K(X)$ translated by 1. For odd and increasing threshold
functions, the condition of Corollary \ref{cor:largesteig} below is satisfied,
so the largest eigenvalue of $K(X)$ equals its spectral norm.

\subsection{Properties of the limit measure}
For the model (\ref{eq:K}), the weak limit of the empirical spectral measure
$p^{-1}\sum_i \delta_{\lambda_i(K(X))}$ of $K(X)$ was characterized by
Cheng and Singer \cite[Theorem 3.4 and Remark 3.2]{chengsinger}. We restate this
result in the following form:
\begin{theorem}[Cheng, Singer]\label{thm:ESD}
Let $X \in \R^{p \times n}$ have entries $x_{ij} \overset{IID}{\sim} \N(0,1)$.
For $\xi \sim \N(0,1)$, suppose
$\E[k(\xi)]=0$, $\E[k(\xi)^2]<\infty$, and
$\int k(x)^2 |q_n(x)-q(x)|dx \to 0$ as $n \to \infty$
where $q$ and $q_n$ are the density
functions of the laws of $\xi$ and $\sqrt{n}\hat{\Sigma}_{12}$.
Then, denoting $a:=\E[\xi k(\xi)]$ and $\nu:=\E[k(\xi)^2]$,
as $n,p \to \infty$ with $p/n \to \gamma \in
(0,\infty)$,
\[\frac{1}{p}\sum_{i=1}^p \delta_{\lambda_i(K(X))} \Rightarrow
\mu_{a,\nu,\gamma}\]
weakly almost surely, where $\mu_{a,\nu,\gamma}$ is a deterministic measure
whose Stieltjes transform $m:\C^+ \to \C^+$ is the unique solution (in $\C^+$,
for any $z \in \C^+$) to the equation
\begin{equation}\label{eq:stieltjes}
-\frac{1}{m(z)}=z+a\left(1-\frac{1}{1+a\gamma
m(z)}\right)+\gamma(\nu-a^2)m(z).
\end{equation}
\end{theorem}
\noindent This result was generalized by Do and Vu to the setting
of non-Gaussian entries $x_{ij}$ in \cite{dovu}.

Before stating our main results,
let us discuss some basic properties of this limit measure:
For a linear kernel function $k(x)=ax$, $\mu_{a,\nu,\gamma}$
is a translation and rescaling of the Marcenko-Pastur law.
Interestingly, it was observed in \cite{chengsinger} that for kernel functions
for which $a=0$, $\mu_{a,\nu,\gamma}$ is a Wigner semicircle law. In fact, the
measure $\mu_{a,\nu,\gamma}$ in general is the additive free convolution
(in the sense of Voiculescu \cite{voiculescuconvolution}) of these two laws.
\begin{proposition}\label{prop:freeconvolution}
Let $\mu_{\semi}$ be the semicircle law supported on $[-2,2]$ and let
$\sqrt{\gamma(\nu-a^2)}\mu_{\semi}$ denote the law of $\sqrt{\gamma(\nu-a^2)}y$
for $y \sim \mu_\semi$. Let $\mu_{\MP,\gamma}$ be the standard
Marcenko-Pastur law that is the limiting spectral measure of $n^{-1}XX^T$ when
$X \in \R^{p \times n}$ and $p/n \to \gamma$, and let
$a(\mu_{\MP,\gamma}-1)$ denote the law of $a(y-1)$ for
$y \sim \mu_{\MP,\gamma}$. Then
\[\mu_{a,\nu,\gamma}=a(\mu_{\MP,\gamma}-1) \boxplus \sqrt{\gamma(\nu-a^2)}
\mu_{\semi}.\]
\end{proposition}
\begin{proof}
By (\ref{eq:stieltjes}), the measure $\mu_{a,\nu,\gamma}$ has
$\mathcal{R}$-transform
\begin{equation}\label{eq:Rtransform}
\mathcal{R}(z)=-a\left(1-\frac{1}{1-a\gamma z}\right)+\gamma(\nu-a^2)z.
\end{equation}
It is easily verified that $-a(1-1/(1-a\gamma z))$ is the
$\mathcal{R}$-transform of $a(\mu_{\MP,\gamma}-1)$ and $\gamma(\nu-a^2)z$
is the $\mathcal{R}$-transform of $\sqrt{\gamma(\nu-a^2)}\mu_{\semi}$, and the
result follows from additivity of $\mathcal{R}$-transforms under additive free
convolution \cite{voiculescuconvolution}.
\end{proof}

Recalling that the additive free convolution of semicircle laws is itself a
semicircle law, Proposition \ref{prop:freeconvolution} implies the following
further decomposition of $\mu_{a,\nu,\gamma}$: Let
$\{h_d\}_{d=0}^\infty$ denote any orthonormal basis of functions $f:\R \to \R$
with respect to
the inner product $\langle f,g \rangle_\xi:=\E[f(\xi)g(\xi)]$ when $\xi \sim
\N(0,1)$, where $h_0(x)=1$ and $h_1(x)=x$. Consider the corresponding
orthogonal decomposition of the kernel function
\[k(x)=\sum_{d=1}^\infty a_dh_d(x)\]
(where $a_0=0$ because $\E[k(\xi)]=0$), and the decomposition
\[K(X)=\sum_{d=1}^\infty K_d(X)\]
where $K_d(X)$ is the matrix (\ref{eq:K}) with kernel function $a_dh_d(x)$.
Letting $\mu_d$ denote the limiting spectral measure of $K_d(X)$, which is
$a(\mu_{\MP,\gamma}-1)$ for $d=1$ and $|a_d|\gamma^{1/2}\mu_\semi$ for $d \geq
2$, the limiting spectral measure of $K(X)$ is given by
\begin{equation}\label{eq:mudecomposition}
\mu_{a,\nu,\gamma}=\mu_1 \boxplus \mu_2 \boxplus \mu_3 \boxplus \ldots.
\end{equation}
In the proof of our main result, we will apply such a decomposition of the
kernel matrix when each $h_d$ is the degree-$d$ Hermite polynomial.

Proposition \ref{prop:freeconvolution} implies, via the general
analysis of \cite{biane}, that $\mu_{a,\nu,\gamma}$ is compactly supported,
has one interval of support when $\gamma \leq 1$ and at most two intervals of
support when $\gamma>1$, and (except for the singularity at 0 in the
Marcenko-Pastur case $\nu=a^2$ and
$\gamma>1$) admits a density on all of $\R$ that is analytic in the interior
of the support. The following may also be deduced from the
$\mathcal{R}$-transform:
\begin{proposition}\label{prop:largesteig}
Let $\supp(\mu_{a,\nu,\gamma})$ denote the support of $\mu_{a,\nu,\gamma}$.
If $a \geq 0$, then
\[\max\{x:x \in \supp(\mu_{a,\nu,\gamma})\}
\geq -\min\{x:x \in \supp(\mu_{a,\nu,\gamma})\},\]
and if $a \leq 0$, then
\[\max\{x:x \in \supp(\mu_{a,\nu,\gamma})\}
\leq -\min\{x:x \in \supp(\mu_{a,\nu,\gamma})\}.\]
\end{proposition}
\begin{proof}
Replacing $k(x)$ by $-k(x)$, it suffices to consider $a \geq 0$. The
$\mathcal{R}$-transform (\ref{eq:Rtransform}) admits the series expansion
\[\mathcal{R}(z)=\gamma \nu z+\sum_{l \geq 2} a^{l+1}\gamma^l z^l\]
around $z=0$, implying that the free cumulants of $\mu_{a,\nu,\gamma}$ are
given by $\kappa_1=0$, $\kappa_2=\gamma \nu$, and $\kappa_l=a^l\gamma^{l-1}$
for $l \geq 3$ \cite{speicher}. The moments of $\mu_{a,\nu,\gamma}$ are then
\[\int x^l \mu_{a,\nu,\gamma}(dx)=\sum_{\pi \in \mathrm{NC}_l}
\prod_{S \in \pi} \kappa_{|S|},\]
where $\mathrm{NC}_l$ denotes the set of all non-crossing partitions of
$\{1,\ldots,l\}$ \cite{speicher}. In particular, when $a \geq 0$, all moments of
$\mu_{a,\nu,\gamma}$ are non-negative, whereas if
$\max\{x:x \in \supp(\mu_{a,\nu,\gamma})\}<
-\min\{x:x \in \supp(\mu_{a,\nu,\gamma})\}$,
then the $l^\text{th}$
moment must be negative for a sufficiently large odd integer $l$.
\end{proof}

The support of $\mu_{a,\nu,\gamma}$ is easily numerically computed, as
(\ref{eq:stieltjes}) is a cubic equation in $m(z)$,
and $\supp(\mu_{a,\nu,\gamma})$ is
the set of $z \in \R$ for which this cubic equation has an imaginary root.
The explicit form for the density function of $\mu_{a,\nu,\gamma}$ was provided
in \cite[Appendix A]{chengsinger}.

\subsection{Main results}
Denoting $\|\mu_{a,\nu,\gamma}\|=\max\{|x|:x \in \supp(\mu_{a,\nu,\gamma})\}$,
the following is the main result of this paper:
\begin{theorem}\label{thm:main}
Suppose $k:\R \to \R$ is odd (i.e.\ $k(-x)=-k(x)$) and continuously
differentiable, with $|k'(x)| \leq Ae^{\beta |x|}$ for some constants 
$A,\beta>0$ and all $x \in \R$. Let $X \in \R^{p \times n}$ have entries
$x_{ij} \overset{IID}{\sim} \N(0,1)$. Then with $\mu_{a,\nu,\gamma}$ as defined
in Theorem \ref{thm:ESD}, almost surely as $n,p \to \infty$
with $p/n \to \gamma \in (0,\infty)$,
\[\|K(X)\| \to \|\mu_{a,\nu,\gamma}\|.\]
\end{theorem}
\noindent Proposition \ref{prop:largesteig} yields the following corollary:
\begin{corollary}\label{cor:largesteig}
Under the conditions of Theorem \ref{thm:main}, if $a:=\E[\xi k(\xi)] \geq
0$, then almost surely
\[\lambda_{\max}(K(X)) \to \max\{x:x \in \supp(\mu_{a,\nu,\gamma})\}.\]
\end{corollary}
\noindent (It may be verified, cf.\ our proof of the above corollary in Section
\ref{sec:proofoverview}, that any kernel function $k$ satisfying the conditions
of Theorem \ref{thm:main} also satisfies the conditions of Theorem
\ref{thm:ESD}.)

We will prove Theorem \ref{thm:main} via the following two auxiliary results,
the first giving a non-asymptotic concentration bound on
$\|K(X)\|$ that is of constant order when $n \asymp p$,
and the second providing an asymptotically tight bound
in the case where $k(x)$ is a polynomial function:

\begin{theorem}\label{thm:concentration}
Suppose $k:\R \to \R$ is odd, continuous, and differentiable almost everywhere
with $|k'(x)| \leq Ae^{\beta |x|}$ for some $A,\beta>0$ and all $x \in \R$. Let
$X \in \R^{p \times n}$ have entries $x_{ij} \overset{IID}{\sim} \N(0,1)$.
Then, for
any $\alpha>0$, there exist constants $C,C'>0$ depending only on $A$, $\beta$, 
and $\alpha$ such that
\[\P\left[\|K(X)\|>C\max\left(\frac{p}{n},\sqrt{\frac{p}{n}}\right)
\right] \leq C'(p^{-\alpha}+pe^{-\alpha n}).\]
\end{theorem}

\begin{theorem}\label{thm:polynomial}
Let $k$ be a polynomial function such that $\E[k(\xi)]=0$ when $\xi \sim
\N(0,1)$, and let $a_2:=\tfrac{1}{\sqrt{2}}\E[k(\xi)(\xi^2-1)]$. Let $X \in
\R^{p \times n}$ have IID entries that are symmetric in law
($x_{ij}\overset{L}{=} -x_{ij}$) and satisfy $\E[x_{ij}^2]=1$ and
\begin{equation}\label{eq:momentassump}
\E[|x_{ij}|^k] \leq k^{\alpha k}
\end{equation}
for all $k \geq 2$ and some $\alpha>0$. Then
\[K(X)=\tilde{K}(X)+\tilde{R}(X)\]
where $\tilde{K}(X)$ and $\tilde{R}(X)$ are such that,
as $n,p \to \infty$ with $p/n \to \gamma \in (0,\infty)$, 
\begin{enumerate}
\item $\|\tilde{K}(X)\| \to \|\mu_{a,\nu,\gamma}\|$ almost surely, and
\item $\tilde{R}(X)=0$ if $a_2=0$, and otherwise $\tilde{R}(X)$ is of rank at
most two, with non-zero eigenvalues converging to $\pm a_2\gamma
\sqrt{(\E[x_{ij}^4]-1)/2}$.
\end{enumerate}
\end{theorem}
\noindent The precise form of the rank-two matrix $\tilde{R}(X)$ is given by
(\ref{eq:tildeR}) in Section \ref{sec:proofoverview}. We make the trivial
observation that $a_2=0$ and $\tilde{R}(X)=0$ if the polynomial $k$ is an odd
function.

\begin{figure}
\includegraphics[width=0.7\textwidth]{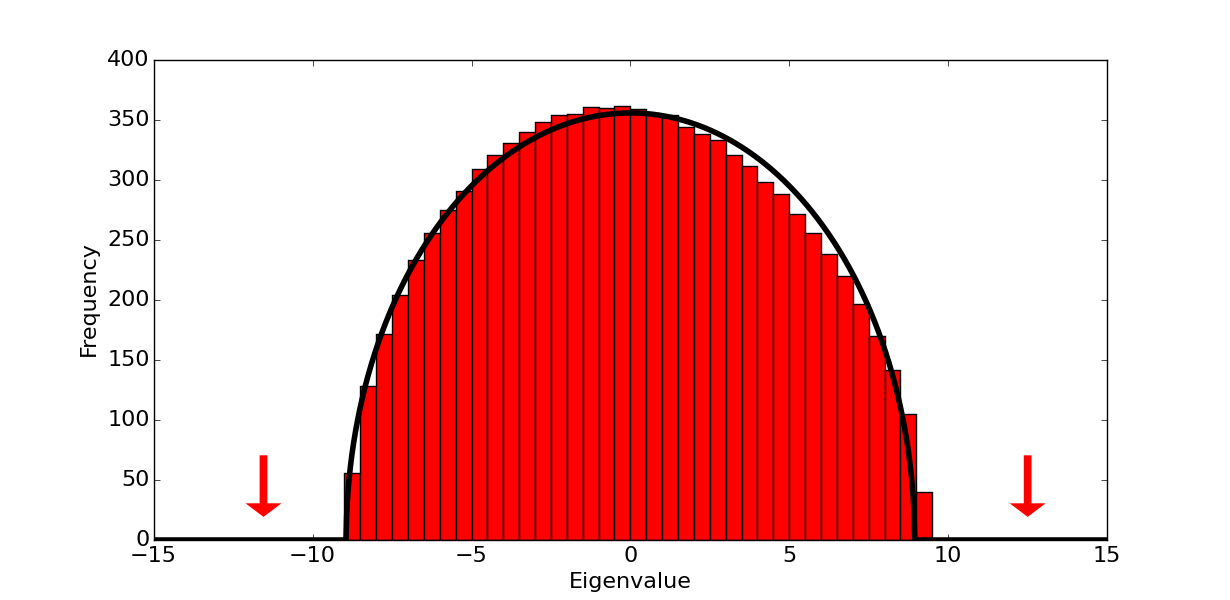}
\caption{Simulated spectrum of $K(X)$ when
$k(x):=h_2(x)+h_3(x)=\frac{1}{\sqrt{2}}(x^2-1)+\frac{1}{\sqrt{6}}(x^3-3x)$,
$n=1000$, and $p=10000$.
The semicircle limit for the spectral distribution is superimposed in black,
and the locations of
two observed outlier eigenvalues of $K(X)$ are indicated with red arrows.}
\label{fig:spikes}
\end{figure}

Theorem \ref{thm:main} follows from Theorems \ref{thm:concentration} and
\ref{thm:polynomial} via a polynomial approximation argument, which we present
in Section \ref{sec:proofoverview}. The assumption that $k(x)$ is odd, or more
specifically that $a_2=0$, is important: Figure \ref{fig:spikes} displays
the simulated spectrum of $K(X)$ for a kernel function where
$a_2 \neq 0$, in which we see that $\tilde{R}(X)$ contributes two spike
eigenvalues to $K(X)$ that fall outside of $\supp(\mu_{a,\nu,\gamma})$. In the
covariance thresholding application of Section \ref{subsec:sparsePCA}, 
commonly-used threshold functions are indeed odd. We recommend caution if using 
a non-odd threshold function, as the
possible presence of these spurious spike eigenvalues may lead to the incorrect
inference that $\Sigma$ has non-trivial spike eigenvectors, even in this null
setting where $\Sigma=\Id$.

\subsection{Further related literature}
The off-diagonal entries of $K(X)$ are the evaluations of a symmetric kernel
$f(u,v):=k(u^Tv/\sqrt{n})/\sqrt{n}$ on pairs of rows of $X$. Such matrices for
general kernels $f(u,v)$ are used in ``kernel methods'' in statistics and
machine learning, such as SVM classifiers 
\cite{boseretal} and kernel PCA \cite{scholkopfetal}. Koltchinskii and Gin\'e
\cite{koltchinskiigine} studied the spectra of kernel matrices in a
regime where each row of $X$ is sampled from a probability distribution over a
fixed space (for example $\R^n$ for fixed $n$), showing that under suitable
conditions, as $p \to \infty$, the spectrum converges to that of a limiting
infinite-dimensional operator. El Karoui \cite{elkarouikernel}
studied kernel matrices in the regime $n,p \to \infty$ with $p/n \to \gamma \in
(0,\infty)$ under the alternative scaling $f(u,v):=k(u^Tv/n)$, showing
that under mild conditions, the matrix is asymptotically equivalent to a linear
combination of $XX^T$, the all-1's matrix, and the identity, and hence the
limiting spectrum is Marcenko-Pastur. The scaling in (\ref{eq:K}) is different
from the regime considered in \cite{elkarouikernel}: Each off-diagonal entry of
$\hat{\Sigma}$ has typical size $1/\sqrt{n}$, and hence (\ref{eq:K}) applies the
nonlinearity $k$ to values of size $O(1)$ rather than $O(1/\sqrt{n})$.
This and more general scalings were studied
probabilistically in \cite{chengsinger}, and the results were further
generalized in \cite{dovu}. Let us
remark that \cite{koltchinskiigine,elkarouikernel} considered
distributions for the rows of $X$ where the entries are not necessarily IID,
but that the extension of our result to more general covariances
$\Sigma \in \R^{p \times p}$ for the application of
Section \ref{subsec:sparsePCA} will require the study of a model in
which the columns (rather than the rows) of $X$ are independent with this
covariance.

Sparse PCA has been widely studied in statistics for both the ``single spike''
model (\ref{eq:spikedmodel}) as well as multi-spike
models. Computationally-efficient procedures for estimating
sparse principal components include diagonal thresholding
\cite{johnstonelu1,johnstonelu2,birnbaumetal},
$\ell_1$- and model-selection-penalization approaches
\cite{jolliffeetal,zouetal,daspremontetal,shenhuang,wittenetal,caietal1},
iterative thresholding via the QR method \cite{ma}, approximate message passing
\cite{deshpandemontanari}, and covariance thresholding as discussed
in Section \ref{subsec:sparsePCA} \cite{krauthgameretal,deshpandemontanari}.
From the theoretical perspective, both exact
\cite{aminiwainwright,krauthgameretal} and approximate
\cite{johnstonelu2,birnbaumetal,ma,vulei,caietal1,caietal2} sparsity models have
been considered, and a major focus has been on rate-optimal recovery of the
sparse eigenvectors, their spanned subspace, and/or the sparse covariance
\cite{birnbaumetal,ma,caietal1,vulei,caietal2}. Support 
recovery and spike detection in the specific model (\ref{eq:spikedmodel}) were
considered in
\cite{aminiwainwright,berthetrigollet1,berthetrigollet2,krauthgameretal,deshpandemontanari}.
In this setting, non-polynomial-time algorithms can detect the spike and recover
the support even when $v$ has sparsity near-linear in $n$
\cite{aminiwainwright,berthetrigollet1,caietal2}, but it is conjectured that
polynomial-time methods require the higher sparsity levels $\|v\|_0 \lesssim
\sqrt{n}$. This problem is closely related to the planted clique
problem in computer science \cite{berthetrigollet1,berthetrigollet2} upon which
this conjecture is based, with a vector $v$ of sparsity $\|v\|_0 \asymp
\sqrt{n}$ corresponding to a planted clique of size $k \asymp \sqrt{n}$ in a
graph of $n$ vertices.

Consistency of elementwise hard-thresholding for estimating sparse covariance
matrices was studied in \cite{bickellevina,elkarouicovariance}.
Optimal rates of convergence under various matrix norms and sparsity models
were established in \cite{caizhou,caizhou2}, and generalizations to other
thresholding functions and to ``adaptive'' entry-specific thresholds were
studied respectively in \cite{rothmanetal} and \cite{cailiu}. Many analyses
assume that each row of $\Sigma$ contains $\ll \sqrt{n}$
non-zero elements and perform thresholding at the level $\sqrt{(\log
p)/n}$ or higher, which does not apply to the regime of interest discussed in
Section \ref{subsec:sparsePCA} where $v$ has sparsity $\|v\|_0 \asymp \sqrt{n}$
and non-zero elements of size $n^{-1/4}$. Thresholding at more general levels,
including $1/\sqrt{n}$, was studied in \cite{deshpandemontanari}, which
established a special case of Theorem \ref{thm:concentration} for the
soft-thresholding kernel function $k$.
The proof of \cite{deshpandemontanari} may be extended to globally Lipschitz
functions $k$, but we require the application of such a bound when $k$ is
the difference of the (possibly Lipschitz) kernel function of interest and a
polynomial approximation to this function.
This difference may increase at any polynomial rate as
$|x| \to \infty$, hence requiring new ideas in the proof of Theorem
\ref{thm:concentration} to extend beyond Lipschitz kernels.
Other spectral norm bounds for polynomial kernels were derived in
\cite{chengsinger} and \cite{kasiviswanathanrudelson}, but they do not
yield the desired bound of constant order when restricted to our setting.

In the context of random matrix theory, convergence of the
extremal eigenvalues of $K(X)$ was posed as an open question in
\cite{chengsinger}. For linear kernels $k$, $K(X)$ is equivalent to a
translation and rescaling of the sample covariance $\hat{\Sigma}$, and
almost-sure convergence of the extremal eigenvalues follows from
\cite{geman,yinetallargest,baiyinsmallest}. Proposition 
\ref{prop:freeconvolution} implies that in the general case, $K(X)$ has the same
limiting spectrum as a deformed Wigner matrix $W+V$ where $W$ is Wigner and
$V$ is deterministic with spectral measure converging to
$a(\mu_{\MP,\gamma}-1)$ \cite{voiculescumatrix}.
When $W$ is GUE and $V$ has no spike eigenvalues, the
results of \cite{capitaineetal,male} imply that the eigenvalues of $W+V$ stick
to the limiting support, and the fluctuations of the eigenvalues at the edges of
the support are also understood in various settings
\cite{shcherbina,capitainepeche,leeschnelli}. The proof of our main result
leverages the connection between these models.

Our proof uses the moment method and is different from the resolvent analysis
of \cite{chengsinger}, although the decomposition of
$k(x)$ in the Hermite polynomial basis plays an important role in both analyses.
While the resolvent method has been successful in establishing many properties
of Wigner and covariance matrices (see e.g.\
\cite{silversteinbai,gotzetikhomirov,taovu,erdosetal,pillaiyin}
as well as the recent work of \cite{srivastavavershynin,chafaitikhomirov}
in a non-independent setting), the model (\ref{eq:K}) for nonlinear kernels
does not have the same independence structure as these
models, and it is also not a sum of rank-one updates. These difficulties were
overcome in \cite{chengsinger} via Gaussian conditioning arguments, but
strengthening the bounds of \cite{chengsinger} to yield finer control of the
Stieltjes transform $m(z)$ near the real axis does not seem (in our viewpoint)
more straightforward than our moment-based approach. We believe that our
combinatorial estimates and moment-comparison argument, in the simpler setting
of a fixed moment $l$ not varying with $n$,
are sufficient to yield an alternative proof of
Theorem \ref{thm:ESD} and also to establish asymptotic freeness of the
matrices $K_1(X),K_2(X),\ldots$ leading to the decomposition
(\ref{eq:mudecomposition}). For brevity, we will not discuss this in the
current paper.

\subsection{Notation}
$\|v\|=(\sum_i v_i^2)^{1/2}$ denotes the Euclidean norm for vectors.
$\|X\|=\max_{\|v\|=1} \|Xv\|$ denotes the spectral norm (i.e.\ $\ell_2$-operator
norm) for matrices. $X_i$ denotes the $i^\text{th}$ row of $X$. 
If $X \in \R^{p \times p}$ is symmetric,
$\lambda_{\max}(X):=\lambda_1(X) \geq \ldots \geq \lambda_p(X)$ denote the
ordered eigenvalues of $X$.
$\supp(\mu)$ denotes the support of a measure $\mu$, and $\|\mu\|$ denotes
$\max\{|x|:x \in \supp(\mu)\}$.

In an asymptotic setting, for positive ($n,p$-dependent) quantities $a$ and $b$,
$a \asymp b$ means $ca \leq b \leq Ca$ for constants
$C,c>0$, $a \sim b$ means $a/b \to 1$, $a \ll b$ means $a/b \to 0$,
and $a \lesssim b$ means $a \leq Cb$
for a constant $C>0$.

We will use $i,i',i_1,i_2,\ldots$ for indices in $\{1,\ldots,p\}$,
and $j,j',j_1,j_2,\ldots$ for indices in $\{1,\ldots,n\}$.

\section{Overview of proof}\label{sec:proofoverview}
In this section, we summarize the high-level proof ideas for
Theorems \ref{thm:concentration} and \ref{thm:polynomial}, and we establish
Theorem \ref{thm:main} and Corollary \ref{cor:largesteig} using these results.

The proof of Theorem \ref{thm:concentration} uses a covering net argument:
\[\|K(X)\| \leq C\sup_{y \in D_2^p} y^TK(X)y,\]
for a constant $C>0$ and a finite covering net $D_2^p$ of the unit ball
$\{y \in \R^p:\|y\| \leq 1\}$.
We use a particular construction of a covering net due to Latala \cite{latala}:
\begin{definition}\label{def:D2}
For $m=\lceil \log_2 p \rceil$, let
\[D_2^p=\left\{y \in \R^p:\|y\| \leq 1,\;y_i^2 \in
\{0,1,2^{-1},2^{-2},\ldots,2^{-(m+3)}\}
\text{ for all } i \right\}.\]
For each $l=0,1,\ldots,m+3$, let $\pi_l:D_2^p \to D_2^p$ be defined by
$(\pi_l(y))_i=y_i\I\{y_i^2 \geq 2^{-l}\}$, and let $\pi_{l \setminus l-1}:
D_2^p \to D_2^p$ be defined by $(\pi_{l \setminus l-1}(y))_i
=y_i\I\{y_i^2=2^{-l}\}$.
\end{definition}

Corresponding to the identity $y=\sum_{l=0}^{m+3}
\pi_{l \setminus l-1}(y)$ for any $y \in D_2^p$,
$y^TK(X)y$ may be decomposed as
\begin{equation}\label{eq:yKy}
y^TK(X)y=\sum_{l=0}^{m+3} \pi_l(y)^TK(X)\pi_{l \setminus l-1}(y)
+\sum_{l=1}^{m+3} \pi_{l \setminus l-1}(y)^TK(X)\pi_{l-1}(y).
\end{equation}
Each of the terms
\[\sup_{y \in D_2^p} \pi_l(y)^TK(X)\pi_{l \setminus l-1}(y),\;\;\;\;
\sup_{y \in D_2^p} \pi_{l \setminus l-1}(y)^TK(X)\pi_{l-1}(y)\]
may be bounded via a standard union bound, with quantities of the form
$F_{y,z}(X):=y^TK(X)z$ controlled by bounding the gradient
$\|\nabla_X F_{y,z}(X)\|$ and applying Gaussian concentration of
measure for Lipschitz functions. The key idea of the construction of $D_2^p$ and
the decomposition (\ref{eq:yKy}) is that for each $l$, the union bound may be
applied over $y \in \pi_l(D_2^p)$, which has smaller cardinality for
smaller $l$. For larger $l$, the entries of $\pi_{l \setminus l-1}(y)$ are
smaller, which we will show implies stronger control of the gradient
$\|\nabla_X F_{y,z}(X)\|$ over a high-probability set $X \in \mathcal{G}$. The
moment generating function of $F_{y,z}(X)$ may be controlled using
the integration argument of Maurey and Pisier, by extending this
high-probability set to pairs of matrices $(X,X')$ in such a way that we remain
in this set along the entire integration path. The cardinality of 
$\pi_l(D_2^p)$ balances the moment generating function bound thus obtained for
each $l$, yielding Theorem \ref{thm:concentration}.
Details of this argument are given in Section \ref{sec:concentration}.

The proof of Theorem \ref{thm:polynomial} uses the moment method and a moment
comparison with a deformed GUE matrix. We first define the orthonormal
Hermite polynomials, which play a central role in our proof (as well as in the
proof in \cite{chengsinger} for Theorem \ref{thm:ESD}):
\begin{definition}\label{def:hermite}
Let $\{h_d\}_{d=0}^\infty$ denote the orthonormal Hermite polynomials
with respect to the inner product $\langle f,g \rangle_\xi=\E[f(\xi)g(\xi)]$
when $\xi \sim \N(0,1)$, i.e.\ $h_d$ is of degree $d$ and
$\langle h_d,h_{d'} \rangle_\xi=\I\{d=d'\}$.
\end{definition}
\noindent The first few such polynomials are given by $h_0(x)=1$, $h_1(x)=x$,
$h_2(x)=\frac{1}{\sqrt{2}}(x^2-1)$, and $h_3(x)=\frac{1}{\sqrt{6}}(x^3-3x)$.

Our proof of Theorem \ref{thm:polynomial} follows three high-level steps:
\begin{enumerate}
\item For IID random variables $z_1,\ldots,z_n$ with $\E[z_i]=\E[z_i^3]=0$
and $\E[z_i^2]=1$, we show that
\begin{equation}\label{eqapproxhermiteidentity}
\sqrt{d!}h_d\left(\frac{\sum_{i=1}^n z_i}{\sqrt{n}}\right)
\approx \sqrt{\frac{1}{n^d}}
\mathop{\sum_{j_1,\ldots,j_d=1}^n}_{j_1 \neq j_2 \neq \ldots \neq j_d}
\prod_{i=1}^d z_{j_i}.
\end{equation}
(The summation on the right side is over all tuples of
distinct indices $j_1,\ldots,j_d \in \{1,\ldots,n\}$.)
Each $h_d$ has leading coefficient $1/\sqrt{d!}$, so
$\sqrt{d!}h_d(x)=x^d+\text{lower degree terms}$.
Replacing $\sqrt{d!}h_d(x)$ with
$x^d$ on the left side would yield the right side of
(\ref{eqapproxhermiteidentity}) without the restriction that the indices of
summation
$j_1,\ldots,j_d$ are distinct; (\ref{eqapproxhermiteidentity}) states that
the terms of this summation in which the indices
$j_1,\ldots,j_d$ are not distinct are essentially cancelled out by the lower
degree terms of $\sqrt{d!}h_d(x)$. We prove this approximation in Section
\ref{sec:hermiteapprox} by induction on $d$, using the three-term recurrence
for Hermite polynomials.
The right side of (\ref{eqapproxhermiteidentity}) is of typical size $O(1)$,
and we also quantify the error of the approximation by computing
a second-order term, which is of typical size $O(n^{-1/2})$, and showing that
the third and higher-order terms in this approximation are of typical size
$O(n^{-1})$.
\item Since $k$ is a polynomial such that $\E[k(\xi)]=0$, we may write
\[k(x)=\sum_{d=1}^D a_dh_d(x)\]
where $D<\infty$ is the degree of $k$. Applying the approximation in step (1)
above to each $h_d$, we obtain a decomposition
\[K(X)=Q(X)+R(X)+S(X),\]
where $Q$, $R$, and $S$ correspond to the first-order, second-order, and
third-and-higher-order terms of these approximations (each summed over all
$d=1,\ldots,D$). We establish for the first-order matrix $Q(X)$ that
\[\limsup_{n,p \to \infty} \|Q(X)\| \leq \|\mu_{a,\nu,\gamma}\|\]
almost surely, via a moment comparison argument:
For an even integer $l \asymp \log n$, we apply the standard moment method bound
$\|Q(X)\|^l \leq \Tr Q(X)^l$ \cite{furedikomlos,geman}. By
(\ref{eqapproxhermiteidentity}), the non-diagonal entries of $Q(X)$ are given by
\[Q(X)_{ii'}=\sum_{d=1}^D a_dn^{-d/2} \mathop{\sum_{j_1,\ldots,j_d=1}^n}_{j_1
\neq j_2 \neq \ldots \neq j_d} \prod_{s=1}^d x_{ij_s}x_{i'j_s}.\]
We expand the trace $\Tr Q(X)^l$ and interpret the terms of the resulting sum as
labelings of a certain graph. We then consider
a deformed GUE matrix $M=W+V$ having the same limiting spectrum as $K(X)$,
and employ a combinatorial argument
to upper-bound $\E[\Tr Q(X)^l]$ using $\E[\Tr M^l]$. We conclude the proof by
using the known convergence result $\|M\| \to \|\mu_{a,\nu,\gamma}\|$ from
\cite{capitaineetal} and a concentration of measure argument to bound
$\E[\Tr M^l]$. We present the main ideas of this
step in Section \ref{sec:Q}, with details deferred to Appendices
\ref{appendixcombinatorics} and \ref{appendixdeformedGUE}.
\item Finally, we analyze the remainder matrices $R(X)$ and
$S(X)$ from the decomposition in step (2) above. It is easily shown that
$\|S(X)\| \to 0$. For $R(X)$, we may write
\[R(X)=\sum_{d=2}^D R_d(X),\]
where $R_d(X)$ is the contribution from the Hermite polynomial $h_d$. (The
linear polynomial $h_1$ does not have such a remainder term in the
decomposition.) We show $\|R_d(X)\| \to 0$ for each
$d \geq 3$, and $\|R_2(X)-\tilde{R}(X)\| \to 0$ where
\begin{equation}\label{eq:tildeR}
\tilde{R}(X)=\frac{a_2}{n\sqrt{2}}(v(X)\1^T+\1 v(X)^T),
\end{equation}
$\1=(1,\ldots,1) \in \R^p$, and $v(X) \in \R^p$ has entries
$(v(X))_i=\sum_{j=1}^n (x_{ij}^2-1)/\sqrt{n}$. Noting that $\tilde{R}(X)$ is a
rank-two matrix, this yields Theorem \ref{thm:polynomial} upon setting
$\tilde{K}(X)=K(X)-\tilde{R}(X)$. This argument and the conclusion of the proof
of Theorem \ref{thm:polynomial} are presented in Section \ref{sec:RS}.
\end{enumerate}

Let us now prove Theorem \ref{thm:main} and Corollary \ref{cor:largesteig}
using Theorems \ref{thm:concentration} and \ref{thm:polynomial}.
We approximate the derivative of the kernel function
by a polynomial using the following result:
\begin{theorem}[Carleson \cite{carleson}]\label{theorempolyapprox}
Suppose $w(x)$ is an even, lower semi-continuous function on $\R$ with
$1 \leq w(x)<\infty$, such that $\log w(x)$ is a convex function of
$\log x$. Let $C_w$ be the class of continuous functions on
$\R$ such that $\lim_{|x| \to \infty} f(x)/w(x)=0$ for all $f \in C_w$,
and suppose $C_w$ contains all polynomial functions. If $\int_1^\infty
(\log w(x))/x^2\,dx=\infty$, then for any $f \in C_w$ and $\eps>0$, there
exists a polynomial $P$ such that $|f(x)-P(x)|<\eps w(x)$ for all $x \in \R$.
\end{theorem}
\begin{proof}[Proof of Theorem \ref{thm:main}]
By the given conditions, there exists $\beta>0$ such that
$\lim_{|x| \to \infty} |k'(x)|/e^{\beta|x|}=0$. Applying Theorem
\ref{theorempolyapprox} with $w(x)=e^{\beta|x|}$, for any $\eps>0$, there exists
a polynomial $\dot{q}$ such that $|k'(x)-\dot{q}(x)|<\eps e^{\beta|x|}$ for all
$x \in \R$. As $k$ is an odd function, $k'$ is even, so we may take $\dot{q}$
to be an even polynomial function. (Otherwise, take the polynomial to be
$\frac{1}{2}(\dot{q}(x)+\dot{q}(-x))$.) Let $q(x)=\int_0^x \dot{q}(x)dx$ for all
$x \in \R$, and let $r(x)=k(x)-q(x)$. Then $q$ is an odd polynomial function,
$r$ is hence also an odd function, and $|r'(x)|<\eps e^{\beta |x|}$ by
construction. Let $Q(X)$ be the kernel matrix (\ref{eq:K}) with kernel
function $q(x)$, and let $R(X)$ be the kernel matrix (\ref{eq:K})
with kernel function $r(x)$, so that $K(X)=Q(X)+R(X)$. (These matrices $Q(X)$
and $R(X)$ are not related to the matrices $Q$, $R$, and $S$ in the above proof
outline.)

Applying Theorem \ref{thm:concentration} with $\alpha=2$ to $R(X)$,
$\limsup_{n,p \to \infty} \|R(X)\|<\eps C_{\beta,\gamma}$ almost surely
for some constant $C_{\beta,\gamma}>0$. On
the other hand, if $q(x)=a_{0,\eps}+a_{1,\eps}h_1(x)+\ldots+a_{D,\eps}h_D(x)$
where $h_1,\ldots,h_D$ are the orthonormal Hermite polynomials of Definition
\ref{def:hermite}, then $a_{j,\eps}=0$ for all even $j$ (since $q$ is an odd
function), and Theorem \ref{thm:polynomial} implies $\|Q(X)\| \to
\|\mu_{a_\eps,\nu_\eps,\gamma}\|$ where $a_\eps=a_{1,\eps}$ and
$\nu_\eps=\sum_{d=1}^D a_{d,\eps}^2$. Hence, almost surely,
\[\|\mu_{a_\eps,\nu_\eps,\gamma}\|-\eps C_{\beta,\gamma}
<\liminf_{n,p \to \infty} \|K(X)\| \leq \limsup_{n,p \to \infty}
\|K(X)\|<\|\mu_{a_\eps,\nu_\eps,\gamma}\|+\eps C_{\beta,\gamma}\]
for any $\eps>0$.
Note that $|k(x)-q(x)| \leq \frac{\eps}{\beta}e^{\beta |x|}$ for all $x \in \R$,
so by dominated convergence $\lim_{\eps \to 0}
\E[(k(\xi)-q(\xi))^2]=0$ for $\xi \sim \N(0,1)$. Then $a_\eps \to a$ and
$\nu_\eps \to \nu$ as $\eps \to 0$, where $a:=\E[\xi k(\xi)]$ and
$\nu:=\E[k(\xi)^2]$. As $\|\mu_{a,\nu,\gamma}\|$ is continuous in $a$, $\nu$,
and $\gamma$, $\lim_{\eps \to 0}
\|\mu_{a_\eps,\nu_\eps,\gamma}\| \to \|\mu_{a,\nu,\gamma}\|$, and hence
taking $\eps \to 0$ yields
$\lim_{n,p \to \infty} \|K(X)\|=\|\mu_{a,\nu,\gamma}\|$ almost surely.
\end{proof}

\begin{proof}[Proof of Corollary \ref{cor:largesteig}]
We verify the conditions of Theorem \ref{thm:ESD}: The kernel function is odd
and bounded as $|k(x)| \leq Ce^{\beta |x|}$ for a constant $C:=C_{A,\beta}>0$,
so $\E[k(\xi)]=0$ and $\E[k(\xi)^2]<\infty$.
Writing $y:=\sqrt{n}\hat{\Sigma}_{12}$, for any $R>0$
\[\E[k(y)^2\I\{|y| \geq R\}] \leq C^2\E\left[e^{4\beta |y|}\right]^{1/2}
\P[|y| \geq R]^{1/2}.\]
Note that $\E[e^{4\beta y}]=\E[e^{-4\beta y}]=(1-16\beta^2/n)^{-n/2}$ for all
$n>16\beta^2$, so $\E[e^{4\beta |y|}]$ is bounded by a constant for all large
$n$. By Lemma C.4 of \cite{chengsinger}, $\P[|y| \geq R]^{1/2} \to 0$ as
$R \to \infty$ uniformly in $n$. Then Lemma C.5 of \cite{chengsinger} implies
that the remaining technical condition of Theorem \ref{thm:ESD} holds.
Theorems \ref{thm:ESD} and \ref{thm:main} then together imply
\[\max\{x:x \in \supp(\mu_{a,\nu,\gamma})\}
\leq \liminf_{n,p \to \infty} \lambda_{\max}(K(X))
\leq \limsup_{n,p \to \infty} \lambda_{\max}(K(X))
\leq \|\mu_{a,\nu,\gamma}\|,\]
and the result follows as the left and right sides coincide by
Proposition \ref{prop:largesteig}.
\end{proof}

\section{Proof of concentration inequality}\label{sec:concentration}
In this section, we prove Theorem \ref{thm:concentration} following the outline
sketched in Section \ref{sec:proofoverview}. By rescaling $k(x)$, we may
assume without loss of generality $A=1$. We denote by $X_i$ the $i^\text{th}$
row of $X$.

\begin{lemma}\label{lemma:goodset}
For any $\alpha,\beta>0$, there exist constants $C,C'>0$ depending only
on $\alpha$ and $\beta$ such that the following holds: Define
$\G(\alpha,\beta) \subset \R^{p \times n}
\times \R^{p \times n}$ as the set of pairs $(X,X')$ of $p \times n$ matrices
such that $\|X\| \leq \sqrt{p}+(1+\sqrt{2\alpha})\sqrt{n}$,
$\|X'\| \leq \sqrt{p}+(1+\sqrt{2\alpha}) \sqrt{n}$, and for each $l=1,\ldots,p$
\[\frac{1}{p}\mathop{\sum_{i=1}^p}_{i \neq l}
\exp\left(\frac{16\beta}{\sqrt{n}}|X_i^TX_l|\right) \leq C,
\hspace{0.2in}
\frac{1}{p}\mathop{\sum_{i=1}^p}_{i \neq l}
\exp\left(\frac{16\beta}{\sqrt{n}}|{X_i'}^TX_l'|\right) \leq C,\]
\[\frac{1}{p}\mathop{\sum_{i=1}^p}_{i \neq l}
\exp\left(\frac{16\beta}{\sqrt{n}}|{X_i'}^TX_l|\right)\leq C,
\hspace{0.2in}
\frac{1}{p}\mathop{\sum_{i=1}^p}_{i \neq l}
\exp\left(\frac{16\beta}{\sqrt{n}}|X_i^TX_l'|\right) \leq C.\]
If $X,X' \in \R^{p \times n}$ are random and independent with
$x_{ij},x_{ij}' \overset{IID}{\sim} \N(0,1)$, then
\[\P[(X,X') \notin \G(\alpha,\beta)]
\leq C'\left(p^{-\alpha}+pe^{-\alpha n}\right).\]
\end{lemma}
\begin{proof}
By Corollary 5.35 of \cite{vershynin},
$\P\left[\|X\|>\sqrt{p}+(1+\sqrt{2\alpha})\sqrt{n}\right]
\leq 2e^{-\alpha n}$, and similarly for $X'$.

For $\xi \sim \N(0,1)$ and any $u>0$, $\E[e^{u|\xi|}] \leq
\E[e^{u\xi}]+\E[e^{-u\xi}]=2e^{\frac{u^2}{2}}$ and
$\Var[e^{u|\xi|}] \leq \E[e^{2u|\xi|}] \leq 2e^{2u^2}$.
Let $C(\alpha,u)$ and $c(\alpha,u)$ denote large and small constants that may
change from instance to instance.
Defining $f(\xi)=e^{u|\xi|}-\E[e^{u|\xi|}]$,
$\E[|f(\xi)|^{\alpha+2}] \leq C(\alpha,u)$.
Then for $\xi_1,\ldots,\xi_p \overset{IID}{\sim} \N(0,1)$,
applying Corollary 4 of \cite{fuknagaev} with $t=\alpha+2$,
\begin{align*}
&\P\left[\frac{1}{p}\sum_{i=1}^p e^{u|\xi_i|}
>3e^{\frac{u^2}{2}}\right]
\leq \P\left[\sum_{i=1}^p f(\xi_i)>pe^{\frac{u^2}{2}}\right]
\leq C(\alpha,u)p^{-\alpha-1}.
\end{align*}
For any $i \neq l$,
$(X_i^TX_l,X_l) \overset{L}{=} (\|X_l\|\xi_i,X_l)$ where $\xi_i \sim \N(0,1)$
is independent of $X_l$. Hence
\begin{align*}
\P\left[\frac{1}{p}\mathop{\sum_{i=1}^p}_{i \neq l}
\exp\left(\frac{16\beta}{\sqrt{n}}|X_i^TX_l|\right)>
3e^{\frac{128\beta^2\|X_l\|^2}{n}} \Bigg| X_l\right]
&\leq C\left(\alpha,\frac{\beta\|X_l\|}{\sqrt{n}}\right)p^{-\alpha-1},
\end{align*}
and
\[\P\left[\frac{1}{p}\mathop{\sum_{i=1}^p}_{i \neq l}
\exp\left(\frac{16\beta}{\sqrt{n}}|X_i^TX_l|\right)>
C(\alpha,\beta) \Bigg| \|X_l\|^2 \leq
(1+2\alpha+2\sqrt{\alpha})n\right] \leq C'(\alpha,\beta)p^{-\alpha-1}\]
for some constants $C(\alpha,\beta)$ and $C'(\alpha,\beta)$.
Lemma 1 of \cite{laurentmassart} implies the chi-squared tail bound
$\P[\|X_l\|^2>(1+2\alpha+2\sqrt{\alpha})n] \leq e^{-\alpha n}$.
The same argument holds for the analogous sums with ${X_i'}^TX_l$, $X_i^TX_l'$,
and ${X_i'}^TX_l'$ in place of $X_i^TX_l$, and the result follows by a union
bound over $l$.
\end{proof}
\begin{lemma}\label{lemma:MGF}
Let $y,z \in \R^p$ satisfy $\|y\| \leq 1$ and $\|z\| \leq 1$. Under the setup of
Theorem \ref{thm:concentration}, let $F(X)=z^TK_{n,p}(X)y$ and define
$\G(\alpha,\beta)$ as in Lemma \ref{lemma:goodset}. Then for a constant
$C:=C(\alpha,\beta)>0$ and any $t>0$,
\[\E\left[e^{t(F(X)-F(X'))}\I\{(X,X') \in \G(\alpha,\beta)\}
\right] \leq 2\exp\left(\frac{C\|y\|_{\infty}t^2p^{1/2}(n+p)}{n^2}
\right).\]
\end{lemma}
\begin{proof}
Consider $F$ as a function from $\R^{pn}$ to $\R$. The gradient with respect to
column $l$ of $X$ is
\begin{align*}
\nabla_{X_l} F(X)&=\nabla_{X_l} \left(\sum_{i=1}^p
\mathop{\sum_{i'=1}^p}_{i' \neq i}
\frac{1}{\sqrt{n}}k\left(\frac{X_i^TX_{i'}}{\sqrt{n}}\right)z_iy_{i'}\right)\\
&=\mathop{\sum_{i=1}^p}_{i \neq l} \frac{1}{n}k'\left(\frac{X_i^TX_l}{\sqrt{n}}
\right)(z_iy_l+y_iz_l)X_i^T=\frac{y_l}{n}v_z^TX+\frac{z_l}{n}v_y^TX,
\end{align*}
for $v_y,v_z \in \R^p$ with
$(v_y)_i=k'(X_i^TX_l/\sqrt{n})y_i\I\{i \neq l\}$ and
$(v_z)_i=k'(X_i^TX_l/\sqrt{n})z_i\I\{i \neq l\}$. This yields the gradient
bound
\begin{align*}
\|\nabla F(X)\|^2 &=\sum_{l=1}^p \|\nabla_{X_l} F(X)\|^2
\leq \sum_{l=1}^p
\frac{2y_l^2}{n^2}\|X\|^2\|v_z\|^2+\frac{2z_l^2}{n^2}\|X\|^2\|v_y\|^2\\
&=\frac{4\|X\|^2}{n^2}\sum_{i=1}^p\mathop{\sum_{l=1}^p}_{l \neq i}
k'\left(\frac{X_i^TX_l}{\sqrt{n}}\right)^2z_i^2y_l^2
\leq \frac{4\|X\|^2}{n^2}\max_{i=1}^p 
\mathop{\sum_{l=1}^p}_{l \neq i}
k'\left(\frac{X_i^TX_l}{\sqrt{n}}\right)^2y_l^2,
\end{align*}
where the last inequality applies $v^TMw \leq \|v\|_1\|Mw\|_\infty$. Applying
Cauchy-Schwarz and the bound $\|y\|_4^2 \leq \|y\|_2\|y\|_\infty \leq
\|y\|_\infty$,
\begin{align}
\|\nabla F(X)\|^2 &\leq \frac{4\|X\|^2}{n^2}\max_{i=1}^p 
\left(\mathop{\sum_{l=1}^p}_{l \neq i}
k'\left(\frac{X_i^TX_l}{\sqrt{n}}\right)^4\right)^{1/2}
\left(\mathop{\sum_{l=1}^p}_{l \neq i} y_l^4\right)^{1/2}\nonumber\\
&\leq \frac{4\|X\|^2\|y\|_\infty}{n^2}\max_{i=1}^p 
\left(\mathop{\sum_{l=1}^p}_{l \neq i}
k'\left(\frac{X_i^TX_l}{\sqrt{n}}\right)^4\right)^{1/2}.
\label{eq:Lipschitzbound}
\end{align}

We apply the integration argument of Maurey and Pisier:
For each $\theta \in \left[0,\frac{\pi}{2}\right]$, let
$X_\theta=X'\cos\theta+X\sin\theta$
and $\tilde{X}_\theta=-X'\sin\theta+X\cos\theta$. Then
\begin{align*}
&\phantom{=}\E\left[e^{t(F(X)-F(X'))}\I\{(X,X') \in \G(\alpha,\beta)\}\right]\\
&=\E\left[\exp\left(\frac{2}{\pi}\int_0^{\frac{\pi}{2}} \frac{\pi t}{2}
\frac{d}{d\theta}F(X_\theta)d\theta\right)\I\{(X,X') \in \G(\alpha,\beta)\}
\right]\\
&\leq \E\left[\frac{2}{\pi}\int_0^{\frac{\pi}{2}}\exp\left(
\frac{\pi t}{2}\frac{d}{d\theta}F(X_\theta)\right)d\theta\,
\I\{(X,X') \in \G(\alpha,\beta)\}\right]\\
&=\frac{2}{\pi}\int_0^{\frac{\pi}{2}}\E\left[
\exp\left(\frac{\pi t}{2}\nabla F(X_\theta)^T\tilde{X}_\theta\right)
\I\{(X,X') \in \G(\alpha,\beta)\}\right]d\theta,
\end{align*}
where $\nabla F(X_\theta)^T\tilde{X}_\theta$ represents the vector
inner-product in $\R^{pn}$. Noting that $X_\theta$ and $\tilde{X}_\theta$ are
independent and both equal in law to $X$, we may first condition on $X_\theta$
and use the Cauchy-Schwarz inequality and the bound
$\E[e^{c|(\tilde{X}_{\theta})_{ij}|}]
\leq \E[e^{c(\tilde{X}_{\theta})_{ij}}]+
\E[e^{-c(\tilde{X}_{\theta})_{ij}}]\leq 2e^{\frac{c^2}{2}}$ to obtain
\begin{align*}
&\phantom{=}\E\left[e^{t(F(X)-F(X'))}\I\{(X,X') \in \G(\alpha,\beta)\}\right]\\
&\leq \frac{2}{\pi}\int_0^{\frac{\pi}{2}} \E\left[\E\left[
\exp\left(\pi t
\nabla F(X_\theta)^T\tilde{X}_\theta\right)\Bigg| X_\theta\right]^{\frac{1}{2}}
\E\left[\I\{(X,X') \in \G(\alpha,\beta)\}
\Bigg| X_\theta \right]^{\frac{1}{2}}\right]d\theta\\
&\leq \frac{4}{\pi}\int_0^{\frac{\pi}{2}}
\E\left[\exp\left(\frac{\pi^2t^2\|\nabla F(X_\theta)\|^2}{4}\right)
\E\left[\I\{(X,X') \in \G(\alpha,\beta)\}\bigg| X_\theta\right]^{\frac{1}{2}}
\right]d\theta\\
&\leq \frac{4}{\pi}\int_0^{\frac{\pi}{2}}
\E\left[\exp\left(\frac{\pi^2t^2\|\nabla F(X_\theta)\|^2}{2}\right)
\I\{(X,X') \in \G(\alpha,\beta)\}\right]^{\frac{1}{2}}d\theta.
\end{align*}

The definition of $\G(\alpha,\beta)$ implies that $\|\nabla F(X_\theta)\|^2$ is
controlled over the entire integration path $\theta \in [0,\frac{\pi}{2}]$:
We have
\[\|X_\theta\|^2 \leq 2\|X'\|^2(\cos\theta)^2+2\|X\|^2(\sin\theta)^2
\leq 2\max(\|X\|^2,\|X'\|^2),\]
and also
\begin{align*}
\mathop{\sum_{l=1}^p}_{l \neq i}
k'\left(\frac{(X_\theta)_i^T(X_\theta)_l}{\sqrt{n}}\right)^4
&\leq \mathop{\sum_{l=1}^p}_{l \neq i}
\exp\left(\frac{4\beta|(X_\theta)_i^T(X_\theta)_l|}{\sqrt{n}}\right)\\
&\leq \mathop{\sum_{l=1}^p}_{l \neq i}
\exp\left(\frac{4\beta(|X_i^TX_l|+|{X_i'}^TX_l|+|X_i^TX_l'|+|{X_i'}^TX_l'|)}
{\sqrt{n}}\right)\\
&\leq \left(\mathop{\sum_{l=1}^p}_{l \neq i} \exp\left(\frac{16\beta|X_i^TX_l|}
{\sqrt{n}}\right)\right)^{1/4}
\left(\mathop{\sum_{l=1}^p}_{l \neq i} \exp\left(\frac{16\beta|{X_i'}^TX_l|}
{\sqrt{n}}\right)\right)^{1/4}\\
&\hspace{0.5in}
\left(\mathop{\sum_{l=1}^p}_{l \neq i} \exp\left(\frac{16\beta|X_i^TX_l'|}
{\sqrt{n}}\right)\right)^{1/4}
\left(\mathop{\sum_{l=1}^p}_{l \neq i} \exp\left(\frac{16\beta|{X_i'}^TX_l'|}
{\sqrt{n}}\right)\right)^{1/4}
\end{align*}
by H\"older's inequality. Then for any
$(X,X') \in \G(\alpha,\beta)$ and $\theta \in [0,\frac{\pi}{2}]$,
(\ref{eq:Lipschitzbound}) implies
\[\|\nabla F(X_\theta)\|^2 \leq
\frac{C(\alpha,\beta)\|y\|_\infty p^{1/2}(n+p)}{n^2},\]
and the result follows.
\end{proof}

Let us now recall $D_2^p$, $\pi_{l \setminus l-1}$, and $\pi_l$ from
Definition \ref{def:D2}.
\begin{lemma}\label{lemma:netapprox}
For any symmetric matrix $M \in \R^{p \times p}$,
$\|M\| \leq 10\sup_{y \in D_2^p} y^TMy$.
\end{lemma}
\begin{proof}
For any $x \in \R^p$ with $\|x\|<1$, we may construct $y \in D_2^p$ such that
\[y_i=\begin{cases} 2^{-\frac{l}{2}}\sign(x_i) & 2^{-l} \leq x_i^2<2^{-l+1}\\
0 & x_i^2<2^{-m-3}. \end{cases}\]
Then $\|y\| \leq \|x\|<1$ and, letting $c=(1-1/\sqrt{2})^2$,
\begin{align*}
\|x-y\|^2&=\sum_{i:\,x_i^2 \geq 2^{-m-3}} (x_i-y_i)^2+\sum_{i:\,x_i^2<2^{-m-3}}
x_i^2
\leq \sum_{i:\,x_i^2 \geq 2^{-m-3}}cx_i^2
+\sum_{i:x_i^2<2^{-m-3}} x_i^2\\
&<c+(1-c)\sum_{i:x_i^2<2^{-m-3}} x_i^2
\leq c+\frac{1-c}{8}<(9/20)^2.
\end{align*}
The result then follows from Lemma 5.4 of \cite{vershynin}.
\end{proof}

\begin{lemma}\label{lemma:netsize}
Let $m=\lceil \log_2 p \rceil$. For some $C>0$ and all
$l \in \{0,1,\ldots,m+3\}$,
\[\log |\{\pi_l(y):y \in D_2^p\}| \leq C(m+4-l)2^l.\]
\end{lemma}
\begin{proof}
Let $C>0$ denote a constant that may change from instance to instance.
For any $l \in \{0,1,\ldots,m\}$,
\[|\{\pi_{l \setminus l-1}(y):y \in D_2^p\}|
\leq \sum_{k=0}^{2^l} \binom{p}{k}2^k,\]
as there are at most $2^l$ non-zero entries of $\pi_{l \setminus l-1}(y)$,
and for each non-zero entry there are two choices of sign. Using $\binom{p}{k}
\leq \left(\frac{ep}{k}\right)^k$, and noting that $k \mapsto (2ep)^kk^{-k}$ is
monotonically increasing over $k \in [0,2p]$ and that $2^l \leq 2p$ for
$l \leq m$, this implies
\[\log |\{\pi_{l \setminus l-1}(y):y \in D_2^p\}| \leq
\log\left(1+2^l\left(\frac{2ep}{2^l}\right)^{2^l}\right)
\leq \log\left(1+2^l\left(2e2^{m-l}\right)^{2^l}\right)
\leq C(m+1-l)2^l.\]
For $l \in \{m+1,m+2,m+3\}$, we use the bound
$|\{\pi_{l \setminus l-1}(y):y \in D_2^p\}| \leq 3^p$, as each coordinate of
$\pi_{l \setminus l-1}(y)$ takes one of three values. Then
\[\log |\{\pi_{l \setminus l-1}(y):y \in D_2^p\}| \leq C2^m \leq C(m+4-l)2^l.\]
Combining these bounds,
\begin{align*}
\log |\pi_l(y)| &\leq \sum_{j=0}^l \log |\pi_{j \setminus j-1}(y)|
\leq C\sum_{j=0}^l (m+4-j)2^j \leq C(m+4-l)2^l.
\end{align*}
\end{proof}

\begin{lemma}\label{lemma:supbound}
Under the setup of Theorem \ref{thm:concentration}, let
$m=\lceil \log_2 p \rceil$ and let $\G(\alpha,\beta)$ be as in Lemma
\ref{lemma:goodset}. Then there are constants
$C,c>0$ depending only on $\alpha$ and $\beta$ such
that for any $l \in \{0,1,\ldots,m+3\}$, $j=l$ or $j=l-1$ (if $l \geq 1$),
and $t>0$,
\[\P\left[\sup_{y \in D_2^p} \pi_j(y)^TK(X)\pi_{l \setminus l-1}(y)>t
\text{ and } (X,X') \in \G(\alpha,\beta)\right] \leq
2\exp\left(C(m+4-l)2^l-\frac{ct^22^{l/2}n^2}{p^{1/2}(n+p)}\right).\]
\end{lemma}
\begin{proof}
For notational convenience, define the event
$\cE:=\{(X,X') \in \G(\alpha,\beta)\}$.
Applying Lemma \ref{lemma:netsize} and a union bound over $\{\pi_l(x):x \in
D_2^p\}$, for any $\lambda>0$,
\begin{align*}
&\phantom{=}
\P\left[\sup_{y \in D_2^p} \pi_j(y)K(X)
\pi_{l \setminus l-1}(y)>t \text{ and } \cE\right]\\
&\leq e^{C(m+4-l)2^l}\sup_{y \in \{\pi_l(x):x \in D_2^p\}}
\P\left[\pi_j(y)K(X)\pi_{l\setminus l-1}(y)>t \text{ and }\cE\right]\\
&\leq e^{C(m+4-l)2^l}e^{-\lambda t}\sup_{y \in \{\pi_l(x):x \in D_2^p\}}
\E\left[e^{\lambda \pi_j(y)K(X)\pi_{l\setminus l-1}(y)}\I\{\cE\}
\right].
\end{align*}
Let $\Lambda$ be the set of all diagonal matrices in $\R^{p \times p}$ with all
diagonal entries in $\{-1,1\}$. Note that $(X,X') \in \G(\alpha,\beta)$ if and
only if $(X,DX') \in \G(\alpha,\beta)$ for all $D \in \Lambda$.
Then, conditional on $X$ and the event $\cE$,
$X'$ equals $DX'$ in law for $D$ uniformly distributed over $\Lambda$. Hence
\[\E[K(X')|X,\cE]=\E[K(DX')|X,\cE]=\E[\E[K(DX')|X',X,\cE]|X,\cE]=0,\]
where the last equality follows from $\E[K(DX')|X']=0$ as the kernel
function $k$ is odd. Then Jensen's inequality yields, for
any $y \in D_2^p$ and $\lambda>0$,
\[\E\left[e^{-\lambda \pi_j(y)K(X')\pi_{l \setminus l-1}(y)}
\Big| X,\cE\right] \geq 1,\]
and so
\begin{align*}
\E\left[e^{\lambda \pi_j(y)K(X)
\pi_{l\setminus l-1}(y)}\I\{\cE\}\right]
&=\E\left[e^{\lambda \pi_j(y)K(X)\pi_{l\setminus l-1}(y)}\Big|\cE\right]
\P[\cE]\\
&\leq \E\left[e^{\lambda \pi_j(y)K(X)\pi_{l\setminus l-1}(y)}\E\left[
e^{-\lambda \pi_j(y)K(X')\pi_{l\setminus l-1}(y)} \Big|X,\cE\right]
\Big|\cE\right]\P[\cE]\\
&=\E\left[e^{\lambda \pi_j(y)(K(X)-K(X'))\pi_{l\setminus l-1}(y)}
\Big|\cE\right]\P[\cE]\\
&=\E\left[e^{\lambda \pi_j(y)(K(X)-K(X'))\pi_{l\setminus l-1}(y)}
\I\{\cE\}\right]\\
&\leq 2\exp\left(\frac{C\lambda^2 p^{1/2}(n+p)}{2^{l/2}n^2}\right),
\end{align*}
where the last line applies Lemma \ref{lemma:MGF}
and the bound $\|\pi_{l \setminus l-1}(y)\|_\infty \leq 2^{-l/2}$.
Optimizing over $\lambda$ yields the desired result.
\end{proof}

We now conclude the proof of Theorem \ref{thm:concentration}.
\begin{proof}[Proof of Theorem \ref{thm:concentration}]
For each $l=0,\ldots,m+3$, set
\[t_l^2=\frac{C_0(m+4-l)2^{l/2}p^{1/2}(n+p)}{n^2}\]
for a constant $C_0:=C_0(\alpha,\beta)$.
Let $X'$ be an independent copy of $X$. Then by Lemma \ref{lemma:supbound},
for each $l=0,\ldots,m+3$ and $j=l$ or $j=l-1$,
\[\P\left[\sup_{y \in D_2^p} \pi_j(y)^TK(X)\pi_{l \setminus l-1}(y)
>t_l \text{ and } (X,X') \in \G(\alpha,\beta)\right] \leq 2e^{-(C-cC_0)(m+4-l)
2^l}.\]
Recalling $m=\lceil \log_2 p \rceil$,
we may pick $C_0$ sufficiently large such that
\[\sum_{l=0}^{m+3} 4e^{-(C-cC_0)(m+4-l)2^l}
\leq 4(m+4)e^{-(C-cC_0)(m+4)} \leq C'p^{-\alpha}\]
for a constant $C':=C'(\alpha,\beta)$. Then (\ref{eq:yKy}) and a union bound
imply
\[\P\left[\sup_{y \in D_2^p} y^TK(X)y>2\sum_{l=0}^{m+3} t_l
\text{ and } (X,X') \in \G(\alpha,\beta)\right] \leq C'p^{-\alpha}.\]
Finally, the bound
\begin{align*}
2\sum_{l=0}^{m+3} t_l&<\frac{2C_0^{1/2}p^{1/4}(n+p)^{1/2}}{n}
\sum_{l=0}^{m+3} (m+4-l)2^{\frac{l}{4}}\\
&=\frac{2C_0^{1/2}p^{1/4}(n+p)^{1/2}}{n}
\sum_{l=0}^{m+3}\sum_{j=0}^l 2^{\frac{j}{4}}
\leq \frac{Cp^{1/4}(n+p)^{1/2}p^{1/4}}{n} \leq C\max\left(\frac{p}{n},
\sqrt{\frac{p}{n}}\right),
\end{align*}
the decomposition (\ref{eq:yKy}), and Lemmas \ref{lemma:goodset} and
\ref{lemma:netapprox} yield the desired result.
\end{proof}

\section{Decomposition of Hermite polynomials of sums of IID random variables}
\label{sec:hermiteapprox}
In this section, we prove the approximation (\ref{eqapproxhermiteidentity})
formalized as the following proposition:
\begin{proposition}\label{propdecomposition}
Let $Z=(z_j:1 \leq j \leq n) \in \R^n$, where $z_j$ are
IID random variables such that $\E[z_j]=\E[z_j^3]=0$, $\E[z_j^2]=1$, and
$\E[|z_j|^l]<\infty$ for each $l \geq 1$. Let $h_d$ denote the orthonormal
Hermite polynomial of degree $d$. Define
\begin{align}
q_{d,n}(Z)&=\sqrt{\frac{1}{n^dd!}}
\mathop{\sum_{j_1,\ldots,j_d=1}^n}_{j_1 \neq j_2 \neq \ldots \neq j_d}
\prod_{i=1}^d z_{j_i},\label{eqqd}\\
r_{d,n}(Z)&=\begin{cases}
0 & d=1\\
\displaystyle \sqrt{\frac{1}{n^dd!}}\binom{d}{2}
\mathop{\sum_{j_1,\ldots,j_{d-1}=1}^n}_{j_1 \neq j_2 \neq \ldots \neq
j_{d-1}} \left((z_{j_1}^2-1) \prod_{i=2}^{d-1} z_{j_i}\right)
& d \geq 2
\end{cases},\label{eqrd}\\
s_{d,n}(Z)&=h_{d}\left(\frac{1}{\sqrt{n}}\sum_{j=1}^n z_j\right)-q_{d,n}(Z)
-r_{d,n}(Z).\label{eqsd}
\end{align}
Then, for each $d \geq 1$ and any $\alpha,\beta>0$,
$\P[|s_{d,n}(Z)|>n^{-1+\alpha}]<n^{-\beta}$ for
all sufficiently large $n$ (i.e. for $n \geq N$ where $N$ may depend on
$\alpha$, $\beta$, $d$, and the distribution of $z_j$).
\end{proposition}

The following lemma shows that
$\P\left[|q_{d,n}(Z)|>n^\alpha\right]<n^{-\beta}$ and
$\P\left[|r_{d,n}(Z)|>n^{-\frac{1}{2}+\alpha}\right]<n^{-\beta}$ for any
$\alpha,\beta>0$ and all sufficiently large $n$. Hence Proposition
\ref{propdecomposition} may be interpreted as decomposing
$h_d(n^{-1/2}\sum_{j=1}^n z_j)$ into the sum of an
$O(1)$ term $q_{d,n}(Z)$, an $O(n^{-1/2})$ term $r_{d,n}(Z)$,
and an $O(n^{-1})$ term $s_{d,n}(Z)$.

\begin{lemma}\label{lemmaindependentsum}
Suppose $z_1,\ldots,z_n$ are IID random variables, with $\E[|z_j|^l]<\infty$ 
for all $l \geq 1$. Let $p_1,\ldots,p_d:\R \to \R$ be any polynomial functions
such that $\E[p_i(z_j)]=0$ for each $i=1,\ldots,d$. Then for any
$\alpha,\beta>0$,
\[\P\left[n^{-\frac{d}{2}}\left|
\mathop{\sum_{j_1,\ldots,j_d=1}^n}_{j_1 \neq j_2 \neq \ldots
\neq j_d} \prod_{i=1}^d p_i(z_{j_i})\right|>n^\alpha\right]<n^{-\beta}\]
for all sufficiently large $n$.
\end{lemma}
\begin{proof}
Fix $\alpha,\beta>0$. Let
\[f(z_1,\ldots,z_n)=n^{-\frac{d}{2}}\left|\mathop{\sum_{j_1,\ldots,j_d=1}^n}_{
j_1 \neq j_2 \neq \ldots \neq j_d} \prod_{i=1}^d p_i(z_{j_i})\right|,\]
and let $l$ be an even integer such that $\alpha l>\beta$. Then
\[\P[f(z_1,\ldots,z_n)>n^\alpha] \leq \frac{\E[f(z_1,\ldots,z_n)^l]}{n^{\alpha
l}},\]
and it suffices to show $\E[f(z_1,\ldots,z_n)^l] \leq C$ for a constant $C$
independent of $n$. Note that
\[\E[f(z_1,\ldots,z_n)^l]=n^{-\frac{ld}{2}}
\mathop{\sum_{j_1^1,\ldots,j_d^1=1}^n}_{j_1^1 \neq \ldots \neq j_d^1}\ldots
\mathop{\sum_{j_1^l,\ldots,j_d^l=1}^n}_{j_1^l \neq \ldots \neq j_d^l}
\E\left[\prod_{i=1}^d \prod_{k=1}^l p_i\left(z_{j_i^k}\right)\right].\]
For each term of the above sum, if there is some $j$ such that $j=j_i^k$ for
exactly one pair of indices $i \in \{1,\ldots,d\}$ and $k \in \{1,\ldots,l\}$,
then the expectation of that term is 0 as $\E[p_i(z_j)]=0$ and $z_j$ is
independent of $z_1,\ldots,z_{j-1},z_{j+1},\ldots,z_n$. Hence, for terms in
the sum with non-zero expectation, there are at most $\frac{ld}{2}$ distinct
values of $j_i^k$. Then the number of such terms is at most
$Cn^{\frac{ld}{2}}$, and the magnitude of each such term is
at most $C'$, for some constants $C,C'$ independent of $n$,
establishing $\E[f(z_1,\ldots,z_n)^l] \leq C$.
\end{proof}

\begin{proof}[Proof of Proposition \ref{propdecomposition}]
Let $S=\frac{1}{\sqrt{n}}\sum_{j=1}^n z_j$.
It will be notationally convenient to work with the monic Hermite polynomials
$\tilde{h}_d=\sqrt{d!}h_d$.
Let us accordingly define $\tilde{q}_{d,n}=q_{d,n}\sqrt{d!}$,
$\tilde{r}_{d,n}=r_{d,n}\sqrt{d!}$,
and $\tilde{s}_{d,n}=s_{d,n}\sqrt{d!}$. Then
\[\tilde{h}_d(S)=\tilde{q}_{d,n}(Z)+\tilde{r}_{d,n}(Z)+\tilde{s}_{d,n}(Z),\]
and we wish to show for any $\alpha,\beta>0$,
$\P\left[|\tilde{s}_d(Z)|>n^{-1+\alpha}\right]<n^{-\beta}$
for all sufficiently large $n$.

We proceed by induction on $d$. Note that $\tilde{h}_0(x)=1$,
$\tilde{h}_1(x)=x$, and $\tilde{h}_2(x)=x^2-1$. Then for $d=1$,
$\tilde{h}_1(S)=S=\tilde{q}_{1,n}(Z)$, and for $d=2$,
\[\tilde{h}_2(S)=S^2-1
=n^{-1}\left(\mathop{\sum_{j_1,j_2=1}^n}_{j_1 \neq j_2}
z_{j_1}z_{j_2}+\sum_{j=1}^n (z_j^2-1)\right)
=\tilde{q}_{2,n}(Z)+\tilde{r}_{2,n}(Z).\]
Hence the proposition holds with $\tilde{s}_{1,n}(Z)=\tilde{s}_{2,n}(Z)=0$.

Let us assume by induction that the proposition holds for $d-1$ and $d$. Recall
that the monic Hermite polynomials satisfy the three-term recurrence
$\tilde{h}_{d+1}(x)=x\tilde{h}_d(x)-d\tilde{h}_{d-1}(x)$ (c.f. eq. (5.5.8) of
\cite{szego}). We may compute
\begin{align*}
S\tilde{q}_{d,n}(Z)
&=n^{-\frac{d+1}{2}}\sum_{j=1}^n z_j\mathop{\sum_{j_1,\ldots,j_d=1}^n}_{j_1
\neq \ldots \neq j_d} \prod_{i=1}^d z_{j_i}\\
&=n^{-\frac{d+1}{2}}\left(\mathop{\sum_{j_1,\ldots,j_{d+1}=1}^n}_{j_1
\neq \ldots \neq j_{d+1}} \prod_{i=1}^{d+1} z_{j_i}
+d\mathop{\sum_{j_1,\ldots,j_d=1}^n}_{j_1 \neq \ldots \neq j_d}
z_{j_1}^2\prod_{i=2}^d z_{j_i}\right)\\
&=\tilde{q}_{d+1,n}(Z)+\frac{2}{d+1}
\tilde{r}_{d+1,n}(Z)+\frac{d(n-d+1)}{n}\tilde{q}_{d-1,n}(Z),\\
S\tilde{r}_{d,n}(Z) &=n^{-\frac{d+1}{2}}\binom{d}{2}\sum_{j=1}^n z_j
\mathop{\sum_{j_1,\ldots,j_{d-1}=1}^n}_{j_1 \neq \ldots \neq
j_{d-1}} \left((z_{j_1}^2-1)
\prod_{i=2}^{d-1} z_{j_i}\right)\\
&=n^{-\frac{d+1}{2}}\binom{d}{2}
\left(\mathop{\sum_{j_1,\ldots,j_d=1}^n}_{j_1 \neq \ldots \neq
j_d} \left((z_{j_1}^2-1) \prod_{i=2}^d z_{j_i}\right)
+\mathop{\sum_{j_1,\ldots,j_{d-1}=1}^n}_{j_1 \neq \ldots \neq j_{d-1}}
\left((z_{j_1}^3-z_{j_1})\prod_{i=2}^{d-1} z_{j_i}\right)\right.\\
&\hspace{0.5in}\left.
+(d-2)\mathop{\sum_{j_1,\ldots,j_{d-1}=1}^n}_{j_1 \neq \ldots \neq j_{d-1}}
\left((z_{j_1}^2-1)z_{j_2}^2\prod_{i=3}^{d-1} z_{j_i}\right)\right)\\
&=\frac{d-1}{d+1}\tilde{r}_{d+1,n}(Z)+n^{-\frac{d+1}{2}}\binom{d}{2}
\mathop{\sum_{j_1,\ldots,j_{d-1}=1}^n}_{
j_1 \neq \ldots \neq j_{d-1}}\left((z_{j_1}^3-z_{j_1})
\prod_{i=2}^{d-1} z_{j_i}\right)\\
&\hspace{0.5in}+n^{-\frac{d+1}{2}}\binom{d}{2}(d-2)
\mathop{\sum_{j_1,\ldots,j_{d-1}=1}^n}_{j_1 \neq
\ldots \neq j_{d-1}} (z_{j_1}^2-1)(z_{j_2}^2-1)\prod_{i=3}^{d-1}z_{j_i}
+\frac{d(n-d+2)}{n}\tilde{r}_{d-1,n}(Z).\\
\end{align*}
Substituting these expressions into the three-term recurrence,
\begin{align*}
\tilde{h}_{d+1}(S)
&=S\left(\tilde{q}_{d,n}(Z)+\tilde{r}_{d,n}(Z)+\tilde{s}_{d,n}(Z)\right)
-d\left(\tilde{q}_{d-1,n}(Z)+\tilde{r}_{d-1,n}(Z)+\tilde{s}_{d-1,n}(Z)\right)\\
&=\tilde{q}_{d+1,n}(Z)+\tilde{r}_{d+1,n}(Z)+\tilde{s}_{d+1,n}(Z)
\end{align*}
for
\begin{align*}
\tilde{s}_{d+1,n}(Z)
&:=-\frac{d(d-1)}{n}\tilde{q}_{d-1,n}(Z)
+n^{-\frac{d+1}{2}}\binom{d}{2}\mathop{\sum_{j_1,\ldots,j_{d-1}=1}^n}_{
j_1 \neq \ldots \neq j_{d-1}}
\left((z_{j_1}^3-z_{j_1})\prod_{i=2}^{d-1} z_{j_i}\right)\\
&\hspace{0.2in}
+n^{-\frac{d+1}{2}}\binom{d}{2}(d-2)
\mathop{\sum_{j_1,\ldots,j_{d-1}=1}^n}_{j_1 \neq \ldots \neq j_{d-1}}\left(
(z_{j_1}^2-1)(z_{j_2}^2-1) \prod_{i=3}^{d-1} z_{j_i}\right)
-\frac{d(d-2)}{n}\tilde{r}_{d-1,n}(Z)\\
&\hspace{0.2in}+S\tilde{s}_{d,n}(Z)-d\tilde{s}_{d-1,n}(Z)
=:I+II+III+IV+V+VI.
\end{align*}
Fix $\alpha,\beta>0$. Note that $\E[z_j]=0$, $\E[z_j^2-1]=0$, and
$\E[z_j^3-z_j]=0$, so by Lemma \ref{lemmaindependentsum},
\[\max\left(\P\left[|I|>n^{-1+\frac{\alpha}{2}}\right],
\P\left[|II|>n^{-1+\frac{\alpha}{2}}\right],
\P\left[|III|>n^{-1+\frac{\alpha}{2}}\right],
\P\left[|IV|>n^{-1+\frac{\alpha}{2}}\right]\right)<n^{-2\beta}\]
for all large $n$. By the induction hypothesis,
$\P[|\tilde{s}_{d,n}|>n^{-1+\frac{\alpha}{4}}]<n^{-2\beta}/2$
for all large $n$, and also $\P[|S|>n^{\frac{\alpha}{4}}]<n^{-2\beta}/2$ for
all large $n$ by Lemma \ref{lemmaindependentsum} (applied to the simple case
where $d=1$ and $p_1(x)=x$). Then
$\P[|V|>n^{-1+\frac{\alpha}{2}}]<n^{-2\beta}$
for all large $n$. Similarly, the induction hypothesis implies
$\P[|VI|>n^{-1+\frac{\alpha}{2}}]<n^{-2\beta}$ for all large $n$.
Putting this together,
\[\P\left[|\tilde{s}_{d+1,n}(Z)|>n^{-1+\alpha}\right]
\leq \P\left[|I+II+III+IV+V+VI|>6n^{-1+\frac{\alpha}{2}}\right]
<6n^{-2\beta}<n^{-\beta}\]
for all large $n$, completing the induction.
\end{proof}

\section{Bounding the dominant matrix}\label{sec:Q}
Consider the polynomial kernel matrix $K(X)$ in Theorem \ref{thm:polynomial}.
Throughout this section, we let $D<\infty$ denote the (fixed) degree of
the polynomial $k$, and we write
\[k(x)=\sum_{d=1}^D a_dh_d(x).\]
Corresponding to the decomposition of $h_d$ given in Proposition
\ref{propdecomposition}, we consider the following decomposition of $K(X)$:
\begin{definition}\label{defQRS}
Define $Q(X)=(q_{ii'}:1 \leq i,i' \leq p) \in \R^{p \times p}$ with entries
\[q_{ii'}=\begin{cases}\displaystyle \frac{1}{\sqrt{n}}\sum_{d=1}^D
a_dq_{d,n}(x_{i1}x_{i'1},\ldots,x_{in}x_{i'n}), & i \neq i'\\
0, & i=i',
\end{cases}\]
where $q_{d,n}$ is as in (\ref{eqqd}). Define
$R(X) \in \R^{p \times p}$ and $S(X) \in \R^{p \times p}$
analogously with $r_{d,n}$ and $s_{d,n}$ in place of $q_{d,n}$,
where $r_{d,n}$ and $s_{d,n}$ are as in (\ref{eqrd}) and (\ref{eqsd}).
\end{definition}

With the above definitions, $K(X)=Q(X)+R(X)+S(X)$.
In this section, we establish the following result:
\begin{proposition}\label{propQnorm}
Under the conditions of Theorem \ref{thm:polynomial}, letting $Q(X)$ be as in
Definition \ref{defQRS},
$\limsup_{n,p \to \infty} \|Q(X)\| \leq \|\mu_{a,\nu,\gamma}\|$ almost
surely.
\end{proposition}
Our proof uses the moment method and the moment comparison argument described in
Section \ref{sec:proofoverview}. The following definitions of an $l$-graph and 
a multi-labeling of such a graph will correspond to the
primary combinatorial object of interest in the subsequent analysis.
\begin{definition}\label{deflgraph}
For any integer $l \geq 2$, an {\bf $\bm{l}$-graph} is a graph consisting of a
single cycle with $2l$ vertices and $2l$ edges, with the vertices alternatingly
denoted as {\bf $\bm{p}$-vertices} and {\bf $\bm{n}$-vertices}.
\end{definition}
We will consider the vertices of the $l$-graph to be ordered by
picking an arbitrary $p$-vertex as the first vertex and ordering the remaining
vertices according to a traversal along the cycle. A vertex $V$ ``follows'' or
``precedes'' another vertex $W$ if $V$ comes before or after $W$, respectively,
in this ordering, and the last vertex of the cycle (which is an $n$-vertex) is
followed by the first $p$-vertex.
\begin{definition}\label{defmultilabeling}
A {\bf multi-labeling} of an $l$-graph is an assignment of a
{\bf $\bm{p}$-label} in $\{1,2,3,\ldots\}$ to each $p$-vertex and an ordered
tuple of {\bf $\bm{n}$-labels} in $\{1,2,3,\ldots\}$ to each $n$-vertex, such
that the following conditions are satisfied:
\begin{enumerate}
\item The $p$-label of each $p$-vertex is distinct from those of the two
$p$-vertices immediately preceding and following it in the cycle.
\item The number $d_s$ of $n$-labels in the tuple for each $s^\text{th}$
$n$-vertex
satisfies $1 \leq d_s \leq D$, and these $d_s$ $n$-labels are distinct.
\item For each distinct $p$-label $i$ and distinct $n$-label $j$, there are an
even number of edges in the cycle (possibly 0) such that its $p$-vertex
endpoint is labeled $i$ and its $n$-vertex endpoint has label $j$ in its tuple.
\end{enumerate}
A {\bf $\bm{(p,n)}$-multi-labeling} is a multi-labeling with all $p$-labels in
$\{1,\ldots,p\}$ and all $n$-labels in $\{1,\ldots,n\}$.
\end{definition}
A key bound on the number of possible distinct $p$-labels and $n$-labels that
appear in a multi-labeling of an $l$-graph is provided by the following lemma.
We will always consider $p$-labels to be distinct from $n$-labels, even though
(for notational convenience) we use the same label set $\{1,2,3,\ldots\}$ for
both.
\begin{restatable}{lemma}{lemmadistinctmultilabels}
\label{lemmadistinctmultilabels}
Suppose a multi-labeling of an $l$-graph has $d_1,\ldots,d_l$
$n$-labels on the first through $l^\text{th}$ $n$-vertices, respectively, and
suppose that it has $m$
total distinct $p$-labels and $n$-labels. Then $m \leq \frac{l+\sum_{s=1}^l
d_s}{2}+1$.
\end{restatable}
\begin{figure}[t]
\includegraphics[width=2.5in]{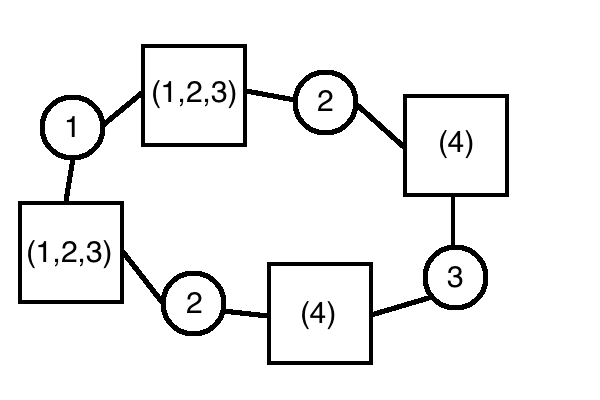}
\caption{A multi-labeling of an $l$-graph for $l=4$ and $D=3$.
$p$-vertices are depicted with a circle and $n$-vertices are depicted with a
square.}
\label{figlabeling}
\end{figure}
We defer the proof of Lemma \ref{lemmadistinctmultilabels} to Appendix
\ref{appendixcombinatorics}. 
Figure \ref{figlabeling} shows an example of a multi-labeling
of an $l$-graph for $l=4$ and $D=3$. In this multi-labeling,
$\sum_{s=1}^4 d_s=3+3+1+1=8$ and the number of total distinct labels is
$m=3+4=7$, so Lemma \ref{lemmadistinctmultilabels} holds with equality.

The non-negative quantity $\frac{l+\sum_{s=1}^l d_s}{2}+1-m$
appears in many of our combinatorial lemmas, and we give it a name:
\begin{definition}\label{defmultiexcess}
Suppose a multi-labeling of an $l$-graph has $d_1,\ldots,d_l$ $n$-labels on the
first through $l^\text{th}$ $n$-vertices, respectively, and suppose that it has 
$m$ total distinct $p$-labels and $n$-labels. The {\bf excess} of the
multi-labeling is $\Delta:=\frac{l+\sum_{s=1}^l d_s}{2}+1-m$.
\end{definition}
A high-level intuition, which we make precise in various ways in Appendix
\ref{appendixcombinatorics}, is that multi-labelings with zero or small
excess satisfy many regularity properties. For example, we prove the following
in Appendix \ref{appendixcombinatorics}:
\begin{restatable}{lemma}{lemmatripleijpairs}
\label{lemmatripleijpairs}
Suppose a multi-labeling of an $l$-graph has excess $\Delta$. For each
$i \in \{1,2,3,\ldots\}$ and $j \in \{1,2,3,\ldots\}$,
let $b_{ij}$ be the number of edges in the $l$-graph such that the $p$-vertex
endpoint is labeled $i$ and the $n$-vertex endpoint has label $j$ in its tuple.
Then $\sum_{i,j:b_{ij}>2} b_{ij} \leq 12\Delta$.
\end{restatable}
\noindent In particular, a multi-labeling with excess $\Delta=0$
has either $b_{ij}=2$ or $b_{ij}=0$ for every label-pair $(i,j)$, by the above
lemma and condition (3) of Definition \ref{defmultilabeling}. This indeed holds
for the example of Figure \ref{figlabeling}.

\begin{definition}\label{defmultiequivclass}
Two multi-labelings of an $l$-graph are {\bf equivalent} if there is a
permutation $\pi_p$ of $\{1,2,3,\ldots\}$ and a permutation $\pi_n$ of
$\{1,2,3,\ldots\}$ such that one labeling is the image of the other upon
applying $\pi_p$ to all of its $p$-labels and $\pi_n$ to all of its $n$-labels.
For any fixed $l$, the equivalence classes under this relation will be called
{\bf multi-labeling equivalence classes}.
\end{definition}
The number of distinct $p$-labels, number of
distinct $n$-labels, number of $n$-labels $d_1,\ldots,d_l$ on each of the
$l$ $n$-vertices, and excess $\Delta$ are equivalence class properties,
i.e.\ they are the same for all labelings in the same multi-labeling
equivalence class. The connection between Definition \ref{defmultilabeling}
of a multi-labeling and our matrix of interest $Q(X)$
is provided by the following lemma:
\begin{lemma}\label{lemmatraceQl}
Let $Q(X)$ be as in Proposition \ref{propQnorm}, and let $l \geq 2$ be an
even integer. Let $\mathcal{C}$ denote the set of all multi-labeling equivalence
classes for an $l$-graph. For each multi-labeling equivalence class
$\mathcal{L} \in \mathcal{C}$, let $\Delta(\mathcal{L})$ be the excess,
$r(\mathcal{L})$ the number of distinct $p$-labels,
and $d_1(\mathcal{L}),\ldots,d_l(\mathcal{L})$ the number of $n$-labels on the
first to $l^\text{th}$ $n$-vertices, respectively. Then, for $\alpha>0$ as in
(\ref{eq:momentassump}) and with the convention $0^0=1$,
\begin{equation}\label{eqtraceQl}
\E[\Tr Q(X)^l] \leq n\sum_{\mathcal{L} \in \mathcal{C}}
\left(\frac{(12\Delta(\mathcal{L}))^{12\alpha}}{n}\right)^{
\Delta(\mathcal{L})} \left(\frac{p}{n}\right)^{r(\mathcal{L})}
\left(\prod_{s=1}^l \frac{|a_{d_s(\mathcal{L})}|}{(d_s(\mathcal{L})!)^{1/2}}
\right).
\end{equation}
\end{lemma}
\begin{proof}
By Definition \ref{defQRS}, letting $i_{l+1}:=i_1$ for notational convenience,
\begin{align*}
\E[\Tr Q(X)^l] &= 
\mathop{\sum_{i_1,\ldots,i_l=1}}_{i_1 \neq i_2,i_2 \neq i_3,\ldots,
i_l \neq i_1}^p \E\left[\prod_{s=1}^l q_{i_si_{s+1}}\right]\\
&=\mathop{\sum_{i_1,\ldots,i_l=1}}_{i_1 \neq i_2,i_2 \neq i_3,\ldots,
i_l \neq i_1}^p
n^{-\frac{l}{2}} \E\left[\prod_{s=1}^l
\left(\sum_{d=1}^D a_d \sqrt{\frac{1}{n^dd!}}
\mathop{\sum_{j_1,\ldots,j_d=1}^n}_{j_1 \neq j_2 \neq \ldots \neq j_d}
\prod_{a=1}^d x_{i_sj_a}x_{i_{s+1}j_a}\right)\right]\\
&=\mathop{\sum_{i_1,\ldots,i_l=1}}_{i_1 \neq i_2,i_2 \neq i_3,\ldots,
i_l \neq i_1}^p \sum_{d_1,\ldots,d_l=1}^D
\mathop{\sum_{j^1_1,\ldots,j^1_{d_1}=1}^n}_{j^1_1\neq \ldots \neq j^1_{d_1}}
\ldots \mathop{\sum_{j^l_1,\ldots,j^l_{d_l}=1}^n}_{j^l_1 \neq \ldots \neq
j^l_{d_l}}\\
&\hspace{1in}n^{-\frac{l+\sum_{s=1}^l
d_s}{2}} \left(\prod_{s=1}^l \frac{a_{d_s}}{(d_s!)^{1/2}}\right)
\E\left[\prod_{s=1}^l \prod_{a=1}^{d_s} x_{i_sj^s_a}x_{i_{s+1}j^s_a}\right].
\end{align*}

Note that as $x_{ij} \overset{L}{=} -x_{ij}$ by assumption,
$\E[x_{ij}^c]=0$ for any positive
odd integer $c$. Hence, if any $x_{ij}$ appears an odd number of times in the
expression $\prod_{s=1}^l \prod_{a=1}^{d_s} x_{i_sj^s_a}x_{i_{s+1}j^s_a}$,
then as the entries of $X$ are independent, 
$\E\left[\prod_{s=1}^l \prod_{a=1}^{d_s} x_{i_sj^s_a}x_{i_{s+1}j^s_a}\right]=0$.
We identify the combination of sums above, over the
remaining non-zero terms, as the sum over all possible $(p,n)$-multi-labelings
of an $l$-graph. Here, the first sum over $i_1,\ldots,i_l$ is over all choices
of $p$-labels, with condition (1) in Definition \ref{defmultilabeling}
corresponding to the constraints
$i_1 \neq i_2,i_2 \neq i_3,\ldots,i_l \neq i_1$ in the sum.
The sum over $d_1,\ldots,d_l$ is over all choices of the number of
$n$-labels in the tuple for each $n$-vertex, and the sum over
$j^s_1,\ldots,j^s_{d_s}$ is over all choices of $d_s$ $n$-labels for
the $s^\text{th}$ $n$-vertex, with condition (2) in Definition
\ref{defmultilabeling} corresponding to the
constraint that $j^s_1,\ldots,j^s_{d_s}$ are distinct. The product expression
$\prod_{s=1}^l \prod_{a=1}^{d_s} x_{i_sj^s_a}x_{i_{s+1}j^s_a}$ then corresponds
to a product, over all $n$-vertices, all $d_s$ $n$-labels for that
$n$-vertex, and both $p$-vertices immediately preceding and immediately
following that $n$-vertex, of $x_{ij}$, where $j \in \{1,\ldots,n\}$ is the
$n$-label and $i \in \{1,\ldots,p\}$ is the $p$-label of the $p$-vertex. The
condition that each $x_{ij}$ appears an even number of times so that
this term has non-zero expectation is precisely condition (3) in Definition
\ref{defmultilabeling}. Thus, to summarize,
\[\E[\Tr Q(X)^l]=\sum_{l \text{-graph } (p,n)\text{-multi-labelings}}
n^{-\frac{l+\sum_{s=1}^l
d_s}{2}} \left(\prod_{s=1}^l \frac{a_{d_s}}{(d_s!)^{1/2}}\right)
\E\left[\prod_{s=1}^l \prod_{a=1}^{d_s} x_{i_sj^s_a}x_{i_{s+1}j^s_a}\right],\]
where $d_1,\ldots,d_l$ are the numbers of $n$-labels for the first through
$l^\text{th}$ $n$-vertices, respectively.

Consider a fixed $(p,n)$-multi-labeling and write $\prod_{s=1}^l
\prod_{a=1}^{d_s}
x_{i_sj^s_a}x_{i_{s+1}j^s_a}=\prod_{j=1}^n \prod_{i=1}^p x_{ij}^{b_{ij}}$, where
$b_{ij}$ is the number of times $x_{ij}$ appears as a term in this product.
Note that each $b_{ij}$ is even (possibly 0).
As $\E[x_{ij}^2]=1$, $\E[|x_{ij}|^k] \leq k^{\alpha k}$,
and the entries of $X$ are independent,
\[\E\left[\prod_{s=1}^l \prod_{a=1}^{d_s} x_{i_sj^s_a}x_{i_{s+1}j^s_a}\right]
=\prod_{i,j:b_{ij}>2} \E\left[x_{ij}^{b_{ij}}\right]
\leq \prod_{i,j:b_{ij}>2} b_{ij}^{\alpha b_{ij}}
\leq \left(\sum_{i,j:b_{ij}>2} b_{ij}\right)^{\alpha
\sum_{i,j:b_{ij}>2} b_{ij}} \leq (12\Delta)^{12\alpha\Delta},\]
where the last inequality applies Lemma \ref{lemmatripleijpairs}
and we use the convention $0^0=1$.
(\ref{eqtraceQl}) then follows upon noting that each
$(p,n)$-multi-labeling with $r$ distinct $p$-labels and $m-r$ distinct
$n$-labels has
$\frac{p!}{(p-r)!}\frac{n!}{(n-m+r)!} \leq n^m\left(\frac{p}{n}\right)^r$
$(p,n)$-multi-labelings in its equivalence class, and $n^{-\frac{l+\sum_{s=1}^l
d_s}{2}+m}=n^{1-\Delta}$.
\end{proof}

We wish to compare the upper bound in (\ref{eqtraceQl}) to an analogous
quantity for a deformed GUE matrix:
\begin{definition}\label{defdeformedGUE}
For $\tilde{n},\tilde{p} \geq 1$, let
$W=(w_{ii'}:1 \leq i,i' \leq \tilde{p}) \in \C^{\tilde{p} \times 
\tilde{p}}$ be distributed according to the GUE, i.e. $\{w_{ii}:1 \leq i
\leq \tilde{p}\} \cup \{\sqrt{2}\Re w_{ii'},\sqrt{2}\Im w_{ii'}:1 \leq i<i'
\leq \tilde{p}\}$ are IID $\N(0,1)$, and
$w_{ii'}=\overline{w_{i'i}}$ for $i>i'$.
Let $V \in \R^{\tilde{p} \times \tilde{p}}$ be
standard real Wishart-distributed with $\tilde{n}$
degrees of freedom and zero diagonal, i.e.\ $V
=ZZ^T-\diag(\|Z_i\|_2^2)$
where $Z=(z_{ij}:1 \leq i \leq \tilde{p},1 \leq j \leq
\tilde{n}) \in \R^{\tilde{p} \times \tilde{n}}$, $z_{ij} \overset{IID}{\sim}
\N(0,1)$, and $ZZ^T-\diag(\|Z_i\|_2^2)$ denotes $ZZ^T$ with its diagonal
set to 0. Take $V$ and $W$ to be independent, and define
\[M=\sqrt{\frac{\gamma(\nu-a^2)}{\tilde{p}}}W
+\frac{a}{\tilde{n}}V \in \C^{\tilde{p} \times \tilde{p}}.\]
\end{definition}

As $\tilde{n},\tilde{p} \to \infty$ with $\tilde{p}/\tilde{n} \to
\gamma$, the limiting spectral distribution of $M$ is
also $\mu_{a,\nu,\gamma}$.
It follows from the results of \cite{capitaineetal} that, in fact,
a norm convergence result holds for
$M$, i.e. $\lim_{\tilde{n},\tilde{p} \to \infty}
\|M\|=\|\mu_{a,\nu,\gamma}\|$, using which we may establish the following
Proposition:
\begin{proposition}\label{propMmomentbound}
Let $M$ be as in Definition \ref{defdeformedGUE}. Suppose
$l$ is an even integer and $\tilde{n},\tilde{p},l \to \infty$ with
$\tilde{p}/\tilde{n} \to \gamma$ and
$l \leq C\log \tilde{n}$ for some constant $C>0$. Then, for any $\eps>0$
and all sufficiently large $\tilde{n}$,
\[\E[\|M\|^l] \leq (\|\mu_{a,\nu,\gamma}\|+\eps)^l.\]
\end{proposition}
The proof of Proposition \ref{propMmomentbound} is deferred to Appendix
\ref{appendixdeformedGUE}.
As $\tilde{p}^{-1}\E[\Tr M^l] \leq \E[\|M\|^l]$, our
strategy for proving Proposition \ref{propQnorm}
will be to show that the upper bound in
(\ref{eqtraceQl}) can in turn be bounded above using the quantity $\E[\Tr
M^l]$, for some choices of $\tilde{p}$ and $\tilde{n}$.
To analyze $\E[\Tr M^l]$, we
consider the following notion of a simple-labeling of an $l$-graph:
\begin{definition}\label{defsimplelabeling}
A {\bf simple-labeling} of an $l$-graph is an assignment of a
{\bf $\bm{p}$-label} in $\{1,2,3,\ldots\}$ to each $p$-vertex and either one
{\bf $\bm{n}$-label} in $\{1,2,3,\ldots\}$ or the empty label $\emptyset$ to
each $n$-vertex, such that the following conditions are satisfied:
\begin{enumerate}
\item The $p$-label of each $p$-vertex is distinct from those of the two
$p$-vertices immediately preceding and following it in the cycle.
\item For each distinct $p$-label $i$ and distinct non-empty $n$-label $j$,
there are an even number of
edges in the cycle (possibly 0) such that its $p$-vertex endpoint is labeled
$i$ and its $n$-vertex endpoint is labeled $j$.
\item For any two distinct $p$-labels $i$ and $i'$, the
number of occurrences (possibly 0) of the three consecutive labels
$i,\emptyset,i'$ on a $p$-vertex, its following $n$-vertex, and its following
$p$-vertex is equal to the number of
occurrences of the three consecutive labels $i',\emptyset,i$.
\end{enumerate}
A {\bf $\bm{(p,n)}$-simple-labeling} is a simple-labeling with all $p$-labels in
$\{1,\ldots,p\}$ and all non-empty $n$-labels in $\{1,\ldots,n\}$.
\end{definition}
Analogous to Lemma \ref{lemmadistinctmultilabels}, the following lemma provides
a key bound on the number of possible distinct $p$-labels and $n$-labels that
appear in a simple-labeling of an $l$-graph.
\begin{restatable}{lemma}{lemmadistinctsimplelabels}
\label{lemmadistinctsimplelabels}
Suppose a simple-labeling of an $l$-graph has
$\tilde{k}$ $n$-vertices with non-empty label and
$\tilde{m}$ total distinct $p$-labels and distinct non-empty $n$-labels. Then
$\tilde{m} \leq \frac{l+\tilde{k}}{2}+1$.
\end{restatable}
The proof of Lemma \ref{lemmadistinctsimplelabels} is deferred to Appendix
\ref{appendixcombinatorics}. We may then define the excess of a
simple-labeling, analogous to Definition \ref{defmultiexcess}, and note that the
excess is always nonnegative.
\begin{definition}
Suppose a simple-labeling of an $l$-graph has
$\tilde{k}$ $n$-vertices with non-empty label and
$\tilde{m}$ total distinct $p$-labels and distinct non-empty $n$-labels. 
The {\bf excess} of the simple-labeling is
$\tilde{\Delta}:=\frac{l+\tilde{k}}{2}+1-\tilde{m}$.
\end{definition}
\begin{figure}
\includegraphics[width=2.5in]{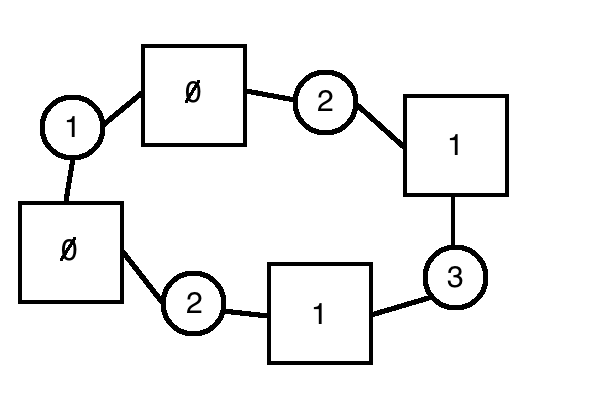}
\caption{A simple labeling of an $l$-graph for $l=4$.
$p$-vertices are depicted with a circle and $n$-vertices are depicted with a
square.}
\label{figsimplelabeling}
\end{figure}
Figure \ref{figsimplelabeling} shows a
simple-labeling of an $l$-graph for $l=4$, with $\tilde{k}=2$ $n$-vertices
having non-empty label and $\tilde{m}=3+1=4$ distinct $p$-labels and non-empty
$n$-labels. Hence in this example, Lemma \ref{lemmadistinctsimplelabels} holds
with equality, and the excess is $\tilde{\Delta}=0$.

\begin{definition}\label{defsimpleequivclass}
Two simple-labelings of an $l$-graph are {\bf equivalent} if there is a
permutation $\pi_p$ of $\{1,2,3,\ldots\}$ and a permutation $\pi_n$ of
$\{1,2,3,\ldots\}$ such that one labeling is the image of the other upon
applying $\pi_p$ to all of its $p$-labels and $\pi_n$ to all of its $n$-labels.
(The empty $n$-label remains empty under any such permutation $\pi_n$.)
For any fixed $l$, the equivalence classes under this relation will be called
{\bf simple-labeling equivalence classes}.
\end{definition}
Motivation for Definition \ref{defsimplelabeling} of a simple labeling is
provided by the following lemma, which gives a lower bound for the
quantity $\E[\Tr M^l]$:
\begin{lemma}\label{lemmatraceMl}
Let $M$ be as in Definition \ref{defdeformedGUE}, and let
$l \geq 2$ be an even integer. Let $\tilde{\mathcal{C}}$ denote the set of all
simple-labeling equivalence classes for an $l$-graph.
For each simple-labeling equivalence class $\tilde{\mathcal{L}} \in
\tilde{\mathcal{C}}$, let $\tilde{\Delta}(\tilde{\mathcal{L}})$ be its excess,
$\tilde{k}(\tilde{\mathcal{L}})$
be the number of $n$-vertices with non-empty label, and
$\tilde{r}(\tilde{\mathcal{L}})$ be the number of distinct $p$-labels. Then,
with the convention $0^0=1$,
\begin{equation}\label{eqtraceMl}
\E[\Tr M^l]
\geq \tilde{n} \left(\frac{\tilde{p}-l}{\tilde{p}}\right)^l
\left(\frac{\tilde{n}-l}{\tilde{n}}\right)^l
\sum_{\tilde{\mathcal{L}} \in \tilde{\mathcal{C}}}
\left(\frac{1}{\tilde{n}}\right)^{\tilde{\Delta}(\tilde{\mathcal{L}})}
\left(\frac{\tilde{p}}{\tilde{n}}\right)^{\tilde{r}(\tilde{\mathcal{L}})
-\frac{l-\tilde{k}(\tilde{\mathcal{L}})}{2}}
|a|^{\tilde{k}(\tilde{\mathcal{L}})}(\gamma(\nu-a^2))
^{\frac{l-\tilde{k}(\tilde{\mathcal{L}})}{2}}.
\end{equation}
\end{lemma}
\begin{proof}
By Definition \ref{defdeformedGUE}, letting $i_{l+1}:=i_1$ for notational
convenience,
\begin{align*}
\E\left[\Tr M^l\right]
&=\E\left[\Tr \left(\sqrt{\frac{\gamma(\nu-a^2)}{\tilde{p}}}W
+\frac{a}{\tilde{n}}V\right)^l\right]\\
&=\sum_{i_1,\ldots,i_l=1}^{\tilde{p}} \E\left[
\prod_{s=1}^l \left(\sqrt{\frac{\gamma(\nu-a^2)}{\tilde{p}}}w_{i_si_{s+1}}
+\frac{a}{\tilde{n}}v_{i_si_{s+1}}\right)\right]\\
&=\sum_{i_1,\ldots,i_l=1}^{\tilde{p}} \sum_{S \subseteq \{1,\ldots,l\}}
\left(\frac{a}{\tilde{n}}\right)^{|S|}
\left(\frac{\gamma(\nu-a^2)}{\tilde{p}}\right)^{\frac{l-|S|}{2}}
\E\left[\prod_{s \in S} v_{i_si_{s+1}}\right]
\E\left[\prod_{s \notin S} w_{i_si_{s+1}}\right]\\
&=\sum_{S \subseteq \{1,\ldots,l\}}
\mathop{\sum_{i_1,\ldots,i_l=1}^{\tilde{p}}}_{i_s \neq i_{s+1} \forall s \in S}
\tilde{n}^{-\frac{l+|S|}{2}}\left(\frac{\tilde{p}}{\tilde{n}}\right)
^{-\frac{l-|S|}{2}}a^{|S|}(\gamma(\nu-a^2))^{\frac{l-|S|}{2}}
\E\left[\prod_{s \in S} v_{i_si_{s+1}}\right]
\E\left[\prod_{s \notin S} w_{i_si_{s+1}}\right]\\
&=\sum_{S \subseteq \{1,\ldots,l\}}
\mathop{\sum_{i_1,\ldots,i_l=1}^{\tilde{p}}}_{i_s \neq i_{s+1} \forall s \in S}
\sum_{(j_s:s \in S) \in \{1,\ldots,\tilde{n}\}^{|S|}}\\
&\hspace{1in}
\tilde{n}^{-\frac{l+|S|}{2}}\left(\frac{\tilde{p}}{\tilde{n}}\right)
^{-\frac{l-|S|}{2}}a^{|S|}(\gamma(\nu-a^2))^{\frac{l-|S|}{2}}
\E\left[\prod_{s \in S} z_{i_sj_s}z_{i_{s+1}j_s}\right]
\E\left[\prod_{s \notin S} w_{i_si_{s+1}}\right].
\end{align*}
In the fourth line above, we restricted the summation to
$i_s \neq i_{s+1}\,\forall s \in S$, as $v_{ii}=0$ for each $i=1,\ldots,
\tilde{p}$ by Definition \ref{defdeformedGUE}.

Let us write $\prod_{s \in S} z_{i_sj_s}z_{i_{s+1}j_s}=\prod_{i=1}^{\tilde{p}}
\prod_{j=1}^{\tilde{n}} z_{ij}^{c_{ij}}$ where $c_{ij}$ is the number of times
$z_{ij}$
appears in this product, and let us write $\prod_{s \notin S} w_{i_si_{s+1}}
=\prod_{i=1}^{\tilde{p}} w_{ii}^{a_{ii}}\prod_{1 \leq i<i' \leq \tilde{p}}
w_{ii'}^{a_{ii'}}w_{i'i}^{b_{ii'}}$, where $a_{ii'}$ and $b_{ii'}$ are the
numbers of times $w_{ii'}$ and $w_{i'i}$ appear in this product, respectively.
$\E[\prod_{i=1}^{\tilde{p}} \prod_{j=1}^{\tilde{n}} z_{ij}^{c_{ij}}] \neq 0$
only if each $c_{ij}$ is even (possibly zero), in which case this quantity is at
least 1. Similarly, note that if $w=re^{i\theta}$ is such that $\sqrt{2}\Re w,
\sqrt{2}\Im w \overset{IID}{\sim} \N(0,1)$, then $r$ and $\theta$ are
independent with $r^2 \sim \chi^2_2/2$ and
$\theta \sim \operatorname{Unif}[0,2\pi)$. Then $\E[w^a\overline{w}^b]
=\E[r^{a+b}]\E[e^{i(a-b)\theta}]$ for all nonnegative integers $a,b$, and this
is 0 if $a \neq b$ and at least 1 if $a=b \geq 0$. Hence
$\E[\prod_{i=1}^{\tilde{p}} w_{ii}^{a_{ii}}\prod_{1 \leq i<i' \leq \tilde{p}}
w_{ii'}^{a_{ii'}}w_{i'i}^{b_{ii'}}]=0$ unless $a_{ii'}=b_{ii'}$ 
for each $i'>i$ and $a_{ii}$ is even (possibly zero) for each $1 \leq i \leq
\tilde{p}$, in which case this quantity is also at least 1.

The above arguments imply, in particular, that
$\E\left[\prod_{s \notin S} w_{i_si_{s+1}}\right]=0$ unless $l-|S|$ is even. As
$l$ is even by assumption, then $|S|$ must also be even,
in which case $a^{|S|}=|a|^{|S|} \geq 0$. Hence each term of the sum in the
above expression for $\E[\Tr M^l]$ is nonnegative, so a lower bound is obtained
if we further restrict the summation to $i_s \neq i_{s+1}\,\forall s \in
\{1,\ldots,l\}$ (rather than just $\forall s \in S$), i.e.
\begin{align*}
\E\left[\Tr M^l\right]
&\geq \sum_{S \subseteq \{1,\ldots,l\}}
\mathop{\sum_{i_1,\ldots,i_l=1}^{\tilde{p}}}_{i_1 \neq i_2,i_2 \neq i_3,\ldots,
i_l \neq i_1} \sum_{(j_s:s \in S) \in \{1,\ldots,\tilde{n}\}^{|S|}}\\
&\hspace{1in}
\tilde{n}^{-\frac{l+|S|}{2}}\left(\frac{\tilde{p}}{\tilde{n}}\right)
^{-\frac{l-|S|}{2}}|a|^{|S|}(\gamma(\nu-a^2))^{\frac{l-|S|}{2}}
\E\left[\prod_{s \in S} z_{i_sj_s}z_{i_{s+1}j_s}\right]
\E\left[\prod_{s \notin S} w_{i_si_{s+1}}\right].
\end{align*}

We identify the combination of these sums
as a sum over all $(\tilde{p},\tilde{n})$-simple-labelings of an $l$-graph.
Here, the first sum over $S$ is over all choices of the subset of
$n$-vertices having non-empty label. The second sum over $i_1,\ldots,i_l$ is
over all choices of $p$-labels, with condition (1) in Definition
\ref{defsimplelabeling} corresponding to the constraints
$i_1 \neq i_2,i_2 \neq i_3,\ldots,i_l \neq i_1$.
The last sum over $(j_s:s \in S)$ is over all choices of
$n$-labels for the $n$-vertices that have nonempty label. The product expression
$\prod_{s \in S} z_{i_sj_s}z_{i_{s+1}j_s}$ then corresponds to a product over
all $n$-vertices with non-empty label and both $p$-vertices immediately
preceding and following that $n$-vertex, and the condition that each $z_{ij}$
appears an even number of times corresponds to condition (2) in Definition
\ref{defsimplelabeling}. Similarly, the product expression $\prod_{s \notin S}
w_{i_si_{s+1}}$ corresponds to a product over all $n$-vertices with empty
label, and the condition that each $w_{ii'}$ appears the same number of times
as $w_{i'i}$ is precisely condition (3) in Definition \ref{defsimplelabeling}.
(By restricting the sum to $i_s \neq i_{s+1}$ for all $s$, no diagonal terms
$w_{ii}$ appear in this product.) Applying the bound
$\E[\prod_{s \in S} z_{i_sj_s}z_{i_{s+1}j_s}]\E[\prod_{s \notin S}
w_{i_si_{s+1}}] \geq 1$ whenever this quantity is nonzero,
\[\E\left[\Tr M^l\right]
\geq \sum_{l \text{-graph } (\tilde{p},\tilde{n}) \text{-simple-labelings}}
\tilde{n}^{-\frac{l+\tilde{k}}{2}}\left(\frac{\tilde{p}}{\tilde{n}}\right)
^{-\frac{l-\tilde{k}}{2}}|a|^{\tilde{k}}
(\gamma(\nu-a^2))^{\frac{l-\tilde{k}}{2}},\]
where $\tilde{k}=|S|$ is the number of $n$-vertices in the simple-labeling with
non-empty label. Any simple labeling
with $\tilde{r}$ distinct $p$-labels and at most $\tilde{m}-\tilde{r}$
distinct non-empty $n$-labels has at most
$\frac{\tilde{n}!}{(\tilde{n}-\tilde{m}+\tilde{r})!}
\frac{\tilde{p}!}{(\tilde{p}-\tilde{r})!} \geq
\tilde{n}^{\tilde{m}} \left(\frac{\tilde{p}}{\tilde{n}}\right)^{\tilde{r}}
\left(\frac{\tilde{p}-l}{\tilde{p}}\right)^l
\left(\frac{\tilde{n}-l}{\tilde{p}}\right)^l$ labelings in its equivalence class
(where we have used $\tilde{m}-\tilde{r} \leq l$ and $\tilde{r} \leq l$). The
desired result then follows upon identifying $\tilde{n}^{1-\tilde{\Delta}}
=\tilde{n}^{-\frac{l+\tilde{k}}{2}+\tilde{m}}$.
\end{proof}
The remainder of the proof of Proposition \ref{propQnorm} involves a comparison
of the upper bound in (\ref{eqtraceQl}) and the lower bound in
(\ref{eqtraceMl}). The intuition for the comparison is the following:
The dominant contributions to the sums in (\ref{eqtraceQl}) and
(\ref{eqtraceMl}) come from labelings with small excess. Focusing on
labelings with excess 0, if we take any multi-labeling equivalence class
$\mathcal{L}$ with $\Delta(\mathcal{L})=0$ and replace the labels of
$n$-vertices having more than one $n$-label with $\emptyset$, then it may
be shown that we obtain a valid simple-labeling equivalence class
$\tilde{\mathcal{L}}$ with $\tilde{\Delta}(\tilde{\mathcal{L}})=0$. For example,
the multi-labeling of Figure \ref{figlabeling} is mapped to the simple
labeling of Figure \ref{figsimplelabeling} under this procedure.
Furthermore, for any $\tilde{\mathcal{L}}$ with
$\tilde{\Delta}(\tilde{\mathcal{L}})=0$, we may show
\[\sum_{\mathcal{L}:\mathcal{L} \text{ maps to } \tilde{\mathcal{L}}}\;\;
\prod_{s=1}^l \frac{|a_{d_s}(\mathcal{L})|}{(d_s(\mathcal{L})!)^{1/2}}
=|a|^{|\tilde{k}(\tilde{\mathcal{L}})|}(\nu-a^2)^{\frac{l-\tilde{k}
(\tilde{\mathcal{L}})}{2}}.\]
(The arguments that establish these claims are a specialization of our
combinatorial lemmas in Appendix \ref{appendixcombinatorics} to the cases of
$\Delta=0$ and $\tilde{\Delta}=0$.) Hence, this mapping yields an exact
correspondence between terms in (\ref{eqtraceQl}) with excess
$\Delta(\mathcal{L})=0$ and terms in (\ref{eqtraceMl}) with excess
$\tilde{\Delta}(\mathcal{L})=0$.

As we must consider $l \asymp \log n$ to establish a tight bound in spectral
norm, we need to also handle terms in (\ref{eqtraceQl}) where
$\Delta(\mathcal{L}) \neq 0$. We do so by extending the above mapping to
all multi-labeling equivalence classes $\mathcal{L}$, in the
case $a \neq 0$. The properties of this mapping that we will need are
summarized in the following proposition.

\begin{proposition}\label{proplabelmapping}
Suppose $a \neq 0$ and $l \geq 2$.
Let $\mathcal{C}$ and $\tilde{\mathcal{C}}$ denote the set of all
multi-labeling and simple labeling equivalence classes of an $l$-graph,
respectively. For $\mathcal{L}
\in \mathcal{C}$, let $\Delta(\mathcal{L})$ be its excess and $r(\mathcal{L})$
be the number of distinct $p$-labels, and for $\tilde{\mathcal{L}} \in
\tilde{\mathcal{C}}$, let $\tilde{\Delta}(\mathcal{L})$ be its excess, 
$\tilde{r}(\tilde{\mathcal{L}})$ be the number of distinct $p$-labels, and
$\tilde{k}(\tilde{\mathcal{L}})$ be the number of $n$-vertices with non-empty
label. Then there
exists a map $\varphi:\mathcal{C} \to \tilde{\mathcal{C}}$ such that, for some
constants $C_1,C_2,C_3,C_4>0$ depending only on $D$,
\begin{enumerate}
\item For all $\mathcal{L} \in \mathcal{C}$,
$r(\mathcal{L})=\tilde{r}(\tilde{\mathcal{L}})$,
\item For all $\mathcal{L} \in \mathcal{C}$,
$\tilde{\Delta}(\varphi(\mathcal{L})) \leq C_1\Delta(\mathcal{L})$, and
\item For any $\tilde{\mathcal{L}} \in \tilde{\mathcal{C}}$ and $\Delta_0 \geq
0$,
\begin{equation}\label{eq:proplabelmapping3}
\mathop{\sum_{\mathcal{L} \in \varphi^{-1}(\tilde{\mathcal{L}})}}_{
\Delta(\mathcal{L})=\Delta_0}
\prod_{s=1}^l \frac{|a_{d_s(\mathcal{L})}|}{(d_s(\mathcal{L})!)^{1/2}}
\leq \left(\frac{\sqrt{\nu}}{|a|}\right)
^{C_2\Delta_0}|a|^{\tilde{k}(\tilde{\mathcal{L}})} (\nu-a^2)
^{\frac{l-\tilde{k}(\tilde{\mathcal{L}})}{2}}l^{C_3+C_4\Delta_0}.
\end{equation}
\end{enumerate}
\end{proposition}

The proof of this proposition and the explicit construction of the map 
$\varphi$ require some detailed combinatorial arguments, which we defer to
Appendix \ref{appendixcombinatorics}. Using this result, we may complete the
proof of Proposition \ref{propQnorm} in the case $a \neq 0$.

\begin{proof}[Proof of Proposition \ref{propQnorm} (Case $a \neq 0$)]
For any $\eps>0$ and even integer $l \geq 2$,
\[\P\left[\|Q(X)\|>(1+\eps)\|\mu_{a,\nu,\gamma}\|\right]
\leq \P\left[\Tr Q(X)^l>\left((1+\eps)\|\mu_{a,\nu,\gamma}\|\right)^l
\right]
\leq \frac{\E[\Tr Q(X)^l]}{(1+\eps)^l\|\mu_{a,\nu,\gamma}\|^l}.\]
By Lemma \ref{lemmatraceQl}, Definition \ref{defmultiexcess}, and Proposition
\ref{proplabelmapping}, 
\begin{align*}
\E[\Tr Q(X)^l] &\leq n
\sum_{\mathcal{L} \in \mathcal{C}}
\left(\frac{(12(\frac{l+Dl}{2}))^{12\alpha}}{n}
\right)^{\Delta(\mathcal{L})}
\left(\frac{p}{n}\right)^{r(\mathcal{L})}
\left(\prod_{s=1}^l
\frac{|a_{d_s(\mathcal{L})}|}{(d_s(\mathcal{L})!)^{1/2}}\right)\\
&=n
\sum_{\tilde{\mathcal{L}} \in \tilde{\mathcal{C}}}
\sum_{\Delta_0=\left\lceil \frac{\tilde{\Delta}(\tilde{\mathcal{L}})}{C_1}
\right \rceil}^{\frac{l+Dl}{2}}
\mathop{\sum_{\mathcal{L} \in \varphi^{-1}
(\tilde{\mathcal{L}})}}_{\Delta(\mathcal{L})=\Delta_0}
\left(\frac{(6l+6Dl)^{12\alpha}}{n} \right)^{\Delta_0}
\left(\frac{p}{n}\right)^{\tilde{r}(\tilde{\mathcal{L}})}
\prod_{s=1}^l
\frac{|a_{d_s(\mathcal{L})}|}{(d_s(\mathcal{L})!)^{1/2}}\\
&\leq n
\sum_{\tilde{\mathcal{L}} \in \tilde{\mathcal{C}}}
\left(\frac{p}{n}\right)^{\tilde{r}(\tilde{\mathcal{L}})}
\sum_{\Delta_0=\left\lceil \frac{\tilde{\Delta}(\tilde{\mathcal{L}})}{C_1}
\right \rceil}^{\frac{l+Dl}{2}}
\left(\frac{(6l+6Dl)^{12\alpha}}{n} \right)^{\Delta_0}
\left(\frac{\sqrt{\nu}}{|a|}\right)^{C_2\Delta_0}
|a|^{\tilde{k}(\tilde{\mathcal{L}})}
(\nu-a^2)^{\frac{l-\tilde{k}(\tilde{\mathcal{L}})}{2}}
l^{C_3+C_4\Delta_0}\\
&\leq nl^{C_3}\left(\tfrac{l+Dl}{2}+1\right)
\sum_{\tilde{\mathcal{L}} \in \tilde{\mathcal{C}}}
\left(\frac{p}{n}\right)^{\tilde{r}(\tilde{\mathcal{L}})}
|a|^{\tilde{k}(\tilde{\mathcal{L}})}
(\nu-a^2)^{\frac{l-\tilde{k}(\tilde{\mathcal{L}})}{2}}
\left(\frac{(6l+6Dl)^{12\alpha}
\left(\frac{\sqrt{\nu}}{|a|}\right)^{C_2}
l^{C_4}}{n}\right)^{\frac{\tilde{\Delta}(\tilde{\mathcal{L}})}{C_1}},
\end{align*}
where the last line holds for all sufficiently large $n$ if
$l \asymp \log n$. Let
\[\tilde{n}=\left\lfloor
\frac{n^{\frac{1}{C_1}}}{(6l+6Dl)^{\frac{12\alpha}{C_1}}
\left(\frac{\sqrt{\nu}}{|a|}\right)^{\frac{C_2}{C_1}}l^{\frac{C_4}{C_1}}}
\right\rfloor,\]
and let $\tilde{p}=\lfloor \frac{\tilde{n}p}{n} \rfloor$. Then
for all sufficiently large $n$ and $l \asymp \log n$,
$nl^{C_3}\left(\frac{l+Dl}{2}+1\right) \leq n^2$,
and also $(\frac{p}{n})^{\tilde{r}(\tilde{\mathcal{L}})} \leq
(\frac{\tilde{p}}{\tilde{n}})^{\tilde{r}(\tilde{\mathcal{L}})}
(1+\frac{\eps}{4})^l$ (as $\tilde{r}(\tilde{L}) \leq l$, $p/n \to \gamma$, and
$\tilde{p}/\tilde{n} \to \gamma$).
Then
\[\E[\Tr Q(X)^l] \leq n^2\left(1+\tfrac{\eps}{4}\right)^l
\sum_{\tilde{\mathcal{L}} \in
\tilde{\mathcal{C}}} \left(\frac{1}{\tilde{n}}\right)^{\tilde{\Delta}
(\tilde{\mathcal{L}})} \left(\frac{\tilde{p}}{\tilde{n}}\right)^{\tilde{r}
(\tilde{\mathcal{L}})} |a|^{\tilde{k}(\tilde{\mathcal{L}})}
(\nu-a^2)^{\frac{l-\tilde{k}(\tilde{\mathcal{L}})}{2}}.\]
On the other hand, by Lemma \ref{lemmatraceMl},
\[\left(1-\tfrac{\eps}{4}\right)^l\tilde{n}
\sum_{\tilde{\mathcal{L}} \in
\tilde{\mathcal{C}}} \left(\frac{1}{\tilde{n}}\right)^{\tilde{\Delta}
(\tilde{\mathcal{L}})} \left(\frac{\tilde{p}}{\tilde{n}}\right)^{\tilde{r}
(\tilde{\mathcal{L}})} |a|^{\tilde{k}(\tilde{\mathcal{L}})}
(\nu-a^2)^{\frac{l-\tilde{k}(\tilde{\mathcal{L}})}{2}}
\leq \E\left[\Tr M^l\right]\]
for all sufficiently large $n$. Since $\tilde{p}/\tilde{n} \to \gamma$
and $l \sim BC_1\log \tilde{n}$ if $l \sim B\log n$, Proposition
\ref{propMmomentbound} implies $\E[\Tr M^l]
\leq \tilde{p}\,\E[\|M\|^l] \leq
\tilde{p}\left(\|\mu_{a,\nu,\gamma}\|\left(1+\frac{\eps}{4}\right)\right)^l$
for all large $n$. Thus
\[\P\left[\|Q(X)\|>(1+\eps)\|\mu_{a,\nu,\gamma}\|\right]
\leq n^2\frac{\tilde{p}}{\tilde{n}}\left(
\frac{\left(1+\tfrac{\eps}{4}\right)^2}{\left(1-\tfrac{\eps}{4}\right)
(1+\eps)}\right)^l.\]
Taking $l \sim B \log n$ with $B>0$ sufficiently large such that
$B \log \frac{\left(1+\tfrac{\eps}{4}\right)^2}{\left(1-\tfrac{\eps}{4}\right)
(1+\eps)}<-4$ (which is possible for any sufficiently small $\eps>0$), this
implies
$\P\left[\|Q(X)\|>(1+\eps)\|\mu_{a,\nu,\gamma}\|\right]<n^{-2}$ for
all large $n$. Then $\limsup_{n,p \to \infty} \|Q(X)\| \leq
(1+\eps)\|\mu_{a,\nu,\gamma}\|$ almost surely,
and taking $\eps \to 0$ concludes the proof.
\end{proof}

As $\|\mu_{a,\nu,\gamma}\|$ is continuous in $a$, $\nu$, and $\gamma$,
Proposition \ref{propQnorm}
in the case $a=0$ may be established via a continuity argument:
\begin{proof}[Proof of Proposition \ref{propQnorm} (Case $a=0$)]
For any $a>0$, let $k_a(x)=k(x)+ax$, and let $Q_a(X)$ be the matrix as
defined in Definition \ref{defQRS} for the kernel function $k_a$. Then
$Q_a(X)=Q(X)+\frac{a}{n}V(X)$, where $V(X)$ has zero
diagonal and equals $XX^T$ off of the diagonal. By Proposition \ref{propQnorm}
for the $a \neq 0$ case, established above, $\limsup_{n,p \to \infty}
\|Q_a(X)\| \leq \|\mu_{a,(\nu+a^2),\gamma}\|$. By standard results for
covariance matrices (see e.g.\ \cite{geman}),
$\limsup_{n,p \to \infty} \|\frac{1}{n}V(X)\| \leq C_\gamma$ almost surely
under the assumption (\ref{eq:momentassump}), for a constant $C_\gamma>0$. This
implies $\limsup_{n,p \to \infty} \|Q(X)\| \leq \|\mu_{a,(\nu+a^2),
\gamma}\|-aC_\gamma$
for any $a>0$, and the desired result follows by taking $a \to 0$.
\end{proof}

\section{Analyzing the remainder matrices}\label{sec:RS}
To conclude the proof of Theorem \ref{thm:polynomial}, we analyze in this
section the remainder matrices $R(X)$ and $S(X)$ of Definition \ref{defQRS}.
\begin{lemma}\label{lemmaS}
As $n,p \to \infty$ with $p/n \to \gamma \in (0,\infty)$,
$\|S(X)\| \to 0$ almost surely.
\end{lemma}
\begin{proof}
Note that
$\|S(X)\| \leq \|S(X)\|_F \leq
p\max_{1 \leq i,i' \leq p} |s_{ii'}|$ where $\|\cdot\|_F$
is the Frobenius norm. By Definition \ref{defQRS} and Proposition
\ref{propdecomposition}, for any $1 \leq i,i' \leq p$ and $\alpha>0$,
$|s_{ii'}| \leq n^{-\frac{3}{2}+\alpha} \sum_{d=1}^D |a_d|$ with probability at
least $1-n^{-4}$, for all large $n$. Then
$p\max_{1 \leq i,i' \leq p} |s_{ii'}| \leq Cpn^{-\frac{3}{2}+\alpha}$ with
probability at least $1-p^2n^{-4}$. Taking any $\alpha<1/2$
yields the desired result.
\end{proof}

\begin{definition}\label{defRd}
For $d \geq 2$, define
$R_d(X)=(r_{ii'}:1 \leq i,i' \leq p) \in \R^{p \times p}$ with entries
\[r_{ii'}=\begin{cases}
\displaystyle
\frac{\binom{d}{2}}{\sqrt{d!}}
n^{-\frac{d+1}{2}} \mathop{\sum_{j_1,\ldots,j_{d-1}=1}^n}_{j_1 \neq j_2 \neq
\ldots \neq j_{d-1}} \left((x_{ij_1}^2x_{i'j_1}^2-1)\prod_{a=2}^{d-1}
x_{ij_a}x_{i'j_a}\right) & i \neq i'\\
0 & i=i'.
\end{cases}\]
\end{definition}
Note that $R(X)$ in Definition \ref{defQRS} is given by
$R(X)=\sum_{d=2}^D a_dR_d(X)$.

\begin{lemma}\label{lemmaRdgeq3}
As $n,p \to \infty$ with $p/n \to \gamma \in (0,\infty)$,
$\|R_d(X)\| \to 0$ almost surely for any $d \geq 3$.
\end{lemma}
\begin{proof}
Letting $i_7:=i_1$ for notational convenience, note that
\[\E\left[\Tr R_d(X)^6\right]
=\mathop{\sum_{i_1,\ldots,i_6=1}^p}_{i_1 \neq i_2,i_2 \neq
i_3,\ldots, i_6 \neq i_1} \E\left[\prod_{s=1}^6 r_{i_si_{s+1}}\right]
=\frac{\binom{d}{2}^6}{(d!)^3}
n^{-3(d+1)}V_d\]
where we set
\begin{equation}\label{eq:Vd}
V_d:=\mathop{\sum_{i_1,\ldots,i_6=1}^p}_{i_1 \neq i_2,i_2 \neq
i_3,\ldots, i_6 \neq i_1}
\mathop{\sum_{j_1^1,\ldots,j_{d-1}^1=1}^n}_{j_1^1 \neq j_2^1 \neq
\ldots \neq j_{d-1}^1} \ldots
\mathop{\sum_{j_1^6,\ldots,j_{d-1}^6=1}^n}_{j_1^6 \neq j_2^6 \neq
\ldots \neq j_{d-1}^6} X(\mathbf{i},\mathbf{j})
\end{equation}
and
\begin{equation}\label{eq:Xij}
X(\mathbf{i},\mathbf{j})
:=\E\left[\prod_{s=1}^6 \left((x_{i_sj_1^s}^2x_{i_{s+1}j_1^s}^2-1)
\prod_{a=2}^{d-1} x_{i_sj_a^s}x_{i_{s+1}j_a^s}\right)\right].
\end{equation}
If there is some $j^* \in \{1,\ldots,n\}$ such that $j_a^s=j^*$ for
exactly one pair of indices $(s,a) \in \{1,\ldots,6\} \times \{1,\ldots,d-1\}$,
then as $i_s \neq i_{s+1}$, independence of the entries of $X$ implies
$X(\mathbf{i},\mathbf{j})=0$. Hence, we may restrict the sum in (\ref{eq:Vd})
to terms where each index $j_a^s$ equals some other index $j_{a'}^{s'}$.
Then the number of distinct indices $j_a^s$ is at most $6(d-1)/2=3d-3$.
Furthermore, it is clear that $|X(\mathbf{i},\mathbf{j})| \leq C$ always,
for a constant $C$ independent of $n$ and $p$.

We now consider several cases for a nonzero term $X(\mathbf{i},\mathbf{j})$,
depending on the number of distinct indices among $\{i_1,\ldots,i_6\}$:

{\bf Case 1: $|\{i_1,\ldots,i_6\}| \leq 4$.} Letting $V_{d,4}$ denote the sum
over terms of (\ref{eq:Vd}) belonging to this case, the above implies
$V_{d,4} \leq Cp^4n^{3d-3}$ for a constant $C$ independent of $n$ and $p$.

{\bf Case 2: $|\{i_1,\ldots,i_6\}|=5$.} Then either
$i_s=i_{s+2}$ or $i_s=i_{s+3}$ for some $s$ (where $s+2$ and $s+3$
are taken modulo 6), with the remaining indices all distinct. Suppose without
loss of generality that $i_2$ and $i_3$ are distinct from each other and from
$\{i_1,i_4,i_5,i_6\}$.
Let $i^*=i_2$ and $j^*=j_2^1$ (which exists when $d \geq 3$). By the
distinctness conditions in (\ref{eq:Vd}), $j^* \neq j_a^1$ for all $a \neq 2$.
If furthermore $j^* \neq j_a^2$ for all $a \in \{1,\ldots,d-1\}$, then
$x_{i^*j^*}$ appears exactly once in (\ref{eq:Xij}), so
$X(\mathbf{i},\mathbf{j})=0$. If $j^*=j_1^2$, then
$x_{i^*j^*}$ appears twice, once as the term $x_{i_2j_2^1}$ and
once in the term $(x_{i_2j_1^2}^2x_{i_3j_1^2}^2-1)$. The product of these
terms is $x_{i^*j^*}^3x_{i_3j_1^2}-x_{i^*j^*}$, and as $\E[x_{i^*j^*}^3]=0$
and $\E[x_{i^*j^*}]=0$, this also implies $X(\mathbf{i},\mathbf{j})=0$. Hence
we must have $j^*=j_2^1=j_a^2$ for some $a \geq 2$.
The same argument applied to $i^*:=i_3$ and $j^*:=j_a^2$ shows that
we must have $j^*=j_a^2=j_{a'}^3$ for some $a' \geq 2$.
Then $j_2^1=j_a^2=j_{a'}^3$, so there cannot be exactly $3d-3$ distinct indices
$j_a^s$. Then there are at most $3d-4$ such distinct indices, and letting
$V_{d,5}$ denote the sum over terms of (\ref{eq:Vd}) belonging to this case,
we obtain $V_{d,5} \leq Cp^5n^{3d-4}$ for a constant $C$.

{\bf Case 3: $|\{i_1,\ldots,i_6\}|=6$.} Then all indices $i_1,\ldots,i_6$ are
distinct. Applying the argument of Case 2, there exists $a \geq 2$ such that
$j_2^1=j_a^2$. There exists further $a' \geq 2$ such that $j_a^2=j_{a'}^3$,
$a'' \geq 2$ such that $j_{a'}^3=j_{a''}^4$, etc., and for each $s=1,\ldots,6$
we obtain some $a \geq 2$ such that $j_2^1=j_a^s$. Then the number of distinct
indices $j_a^s$ is at most $\frac{6(d-1)-6}{2}+1=3d-5$. Letting $V_{d,6}$ denote
the sum over terms of (\ref{eq:Vd}) belonging to this case, we obtain
$V_{d,6} \leq Cp^6n^{3d-5}$.

Putting the cases together,
\[\E\left[\Tr R_d(X)^6\right] \leq Cn^{-3(d+1)}(V_{d,4}+V_{d,5}+V_{d,6})
\leq Cn^{-2}\]
for a constant $C>0$ and all large $n$ and $p$. Then for any $\eps>0$,
\[\P\left[\|R_d(X)\|>\eps\right] \leq \frac{\E[\Tr
R_d(X)^6]}{\eps^6} \leq \frac{C}{\eps^6n^2},\]
so $\limsup_{n,p \to \infty} \|R_d(X)\|
\leq \eps$ almost surely, and the result follows by taking $\eps \to 0$.
\end{proof}

\begin{lemma}\label{lemmaR2}
Let $R_d(X)$ be as in Definition \ref{defRd}, and let $\tilde{R}(X)$
be as in (\ref{eq:tildeR}).
As $n,p \to \infty$ with $p/n \to \gamma \in (0,\infty)$,
$\|a_2R_2(X)-\tilde{R}(X)\| \to 0$ almost surely.
\end{lemma}
\begin{proof}
Let $T(X)=a_2R_2(X)-\tilde{R}(X)$. Then $T(X)$ has entries
\begin{align*}
t_{ii'}&=\begin{cases}
\frac{a_2}{\sqrt{2}}n^{-\frac{3}{2}}\sum_{j=1}^n
\left((x_{ij}^2x_{i'j}^2-1)-(x_{ij}^2-1)-(x_{i'j}^2-1)\right) & i \neq i'\\
0 & i=i'
\end{cases}\\
&=\begin{cases}
\frac{a_2}{\sqrt{2}}n^{-\frac{3}{2}}\sum_{j=1}^n
(x_{ij}^2-1)(x_{i'j}^2-1) & i \neq i'\\
0 & i=i'
\end{cases}.
\end{align*}
Thus, excluding the diagonal, $T(X)$ equals
$\frac{a_2}{\sqrt{2}}n^{-\frac{3}{2}}YY^T$ where $Y=(y_{ij}) \in \R^{p \times
n}$ and $y_{ij}=x_{ij}^2-1$. Under the assumption (\ref{eq:momentassump}),
$\frac{1}{n}\|YY^T\|$ converges to a finite limit almost surely (see
e.g.\ \cite{geman}), so $n^{-\frac{3}{2}}\|YY^T\| \to 0$. Furthermore,
$T(X)-\frac{a_2}{\sqrt{2}}n^{-\frac{3}{2}}YY^T$ is diagonal and
$\left\|T(X)-\frac{a_2}{\sqrt{2}}n^{-\frac{3}{2}}YY^T\right\| \to 0$
is easily verified by a union bound, implying $\|T(X)\| \to 0$.
\end{proof}

We now conclude the proof of Theorem \ref{thm:polynomial}.
\begin{proof}[Proof of Theorem \ref{thm:polynomial}]
Recall Definitions \ref{defQRS} and \ref{defRd} and the decompositions
$K(X)=Q(X)+R(X)+S(X)$ and $R(X)=\sum_{d=2}^D a_dR_d(X)$.
Proposition \ref{propQnorm} and Lemmas \ref{lemmaS}, \ref{lemmaRdgeq3},
and \ref{lemmaR2} imply
$\limsup_{n,p \to \infty} \|K(X)-\tilde{R}(X)\|
\leq \|\mu_{a,\nu,\gamma}\|$. The conditions of Theorem \ref{thm:ESD} are
verified for the kernel function $k(x)$ as in the proof of Corollary
\ref{cor:largesteig} in Section \ref{sec:proofoverview}. Furthermore, $K(X)$
and $\tilde{K}(X):=K(X)-\tilde{R}(X)$ have the same limiting empirical
spectral distribution, since $\tilde{R}(X)$ has finite rank. Then
Theorem \ref{thm:ESD} implies $\liminf_{n,p \to \infty}
\|\tilde{K}(X)\| \geq \|\mu_{a,\nu,\gamma}\|$ almost surely, and this
establishes property (1) of Theorem \ref{thm:polynomial}.

To verify the claim regarding the non-zero eigenvalues of $\tilde{R}(X)$ in
property (2) of Theorem \ref{thm:polynomial}, we compute from (\ref{eq:tildeR})
$\Tr \tilde{R}(X)=\frac{a_2\sqrt{2}}{n} v(X)^T\1$
and $\Tr \tilde{R}(X)^2=\frac{a_2^2}{n^2}((
v(X)^T\1)^2+p\|v(X)\|^2)$. If $\lambda_1$ and $\lambda_2$ are the
two non-zero eigenvalues of $\tilde{R}(X)$, then
$\lambda_1+\lambda_2=\Tr \tilde{R}(X)$ and
$\lambda_1\lambda_2=\frac{1}{2}((\lambda_1+\lambda_2)^2-\lambda_1^2-\lambda_2^2)
=\frac{1}{2}((\Tr \tilde{R}(X))^2-\Tr \tilde{R}(X)^2)$, so
$\lambda_1$ and $\lambda_2$ are the roots of the equation
\[\lambda^2-\left(\frac{a_2\sqrt{2}}{n}v(X)^T\1\right)\lambda+
\frac{a_2^2}{2n^2}\left((v(X)^T\1)^2-p\|v(X)\|^2\right)=0.\]
By the law of large numbers, $n^{-1}v(X)^T\1 \to 0$
and $(p/n^2)\|v(X)\|^2 \to \gamma^2(\E x_{ij}^4-1)$ almost surely.
Since the roots of a polynomial are
continuous in its coefficients, the result follows.
\end{proof}

\appendix
\section{Combinatorial results}\label{appendixcombinatorics}
This appendix contains the proofs of Lemmas \ref{lemmadistinctmultilabels},
\ref{lemmatripleijpairs}, and \ref{lemmadistinctsimplelabels} used in
Section \ref{sec:Q}, as well as the proof of Proposition \ref{proplabelmapping}
and the explicit construction of the map $\varphi$ in that proposition.

\subsection{Proof of Lemmas \ref{lemmadistinctmultilabels},
\ref{lemmatripleijpairs}, and \ref{lemmadistinctsimplelabels}}
We restate the lemmas using their original numbering.
\lemmadistinctsimplelabels*
\begin{proof}
Let $I=\{1,\ldots,p\}$ and $J=\{1,\ldots,n\}$, and consider an undirected
graph $G$ on the vertex set $I \sqcup J$ (the disjoint union of $I$ and $J$
with $n+p$ elements, treating elements of $I$ and the elements of $J$ as
distinct). Let $G$ have an edge between $i,i' \in I$ if there are three
consecutive vertices ($p$, $n$, $p$) of the $l$-graph
with the labels $i$, $\emptyset$, $i'$ or $i'$, $\emptyset$, $i$. Let $G$ have
an edge between $i \in I$ and $j \in J$ if there are two
consecutive vertices of the $l$-graph such that the $p$-vertex has label $i$ and
the $n$-vertex has label $j$. The number of vertices of $G$ incident to at least
one edge is $\tilde{m}$, and $G$ must be connected, so it has at least
$\tilde{m}-1$ edges. An edge in $G$ between $i,i' \in I$ corresponds
to at least two consecutive pairs of $p$-vertices in the $l$-graph having an
$n$-vertex with empty label in between, by condition (3) of Definition
\ref{defsimplelabeling}, so the number of such edges is at
most $\frac{l-\tilde{k}}{2}$. Similarly, an edge in $G$ between $i \in I$ and
$j \in J$ corresponds to at least two pairs of consecutive $n$ and $p$-vertices
of the $l$-graph such that the $n$-vertex has non-empty label, by condition (2)
of \ref{defsimplelabeling}, so the
number of such edges is at most $\frac{2\tilde{k}}{2}$. Then $\tilde{m}-1 \leq
\frac{l+\tilde{k}}{2}$.
\end{proof}

Turning now to multi-labelings, for each $j \in \{1,2,3,\ldots\}$
and a given multi-labeling, let us denote throughout
\[N_j:=\text{number of appearances of } j \text{ as an } n\text{-label}.\]
Then the following two lemmas hold:
\begin{lemma}\label{lemmatwicenlabels}
In any multi-labeling of an $l$-graph, each $j$ that appears as an $n$-label
has $N_j \geq 2$.
\end{lemma}
\begin{proof}
Suppose that an $n$-label $j$ appears only once. The two $p$-vertices
preceding and following that $n$-vertex must have distinct labels, say $i_1$
and $i_2$, by condition (1) of Definition \ref{defmultilabeling}. Then exactly
one edge in the $l$-graph has $p$-vertex endpoint labeled $i_1$ and
$n$-vertex endpoint having label $j$ (and similarly for $i_2$ and $j$),
contradicting condition (3) of Definition \ref{defmultilabeling}.
\end{proof}
\begin{lemma}\label{lemmafewdistinctp}
Suppose a multi-labeling of an $l$-graph has at most $\frac{l}{2}$
distinct $p$-labels. If this multi-labeling has excess $\Delta$, then
\[\sum_{j:N_j \geq 3} N_j \leq 6\Delta-6.\]
Consequently, the number of $n$-vertices having any label $j$ for which
$N_j \geq 3$ is also at most $6\Delta-6$.
\end{lemma}
\begin{proof}
Observe that if $m$ total distinct $p$-labels and $n$-labels appear in the
labeling, and at most $\frac{l}{2}$ of these are $p$-labels, then the labeling
has at least $m-\frac{l}{2}$ distinct $n$-labels. Let $c=|\{j:N_j=2\}|$.
Then Lemma \ref{lemmatwicenlabels} 
implies $2c+3\left(m-\frac{l}{2}-c\right) \leq \sum_{s=1}^l d_s$
(where $d_1,\ldots,d_l$ are the numbers of $n$-labels on the $l$ $n$-vertices),
so $c \geq 3m-\frac{3l}{2}-\sum_{s=1}^l d_s$. Then the $n$-labels in
$\{j:N_j=2\}$ account for at least
$6m-3l-2\sum_{s=1}^l d_s$ of the $\sum_{s=1}^l d_s$ total $n$-labels,
implying that at most $3l+3\sum_{s=1}^l d_s-6m=6\Delta-6$ total
$n$-labels remain. This establishes the first claim, and the second follows
directly from the first.
\end{proof}

We will prove many subsequent claims regarding multi-labelings by
induction on $l$. The following two lemmas describe the base case of the
induction and the basic inductive step.
\begin{lemma}\label{lemmal23}
Suppose $l=2$ or $l=3$. Then for any multi-labeling of the $l$-graph, all $l$
$p$-labels are distinct, and all $l$ $n$-vertices have the same tuple of
$n$-labels, up to reordering.
\end{lemma}
\begin{proof}
That all $l$ $p$-labels are distinct is a consequence of condition (1) of
Definition \ref{defmultilabeling}. Then by conditions (2) and (3) of Definition
\ref{defmultilabeling}, the $n$-vertices immediately preceding and following
each $p$-vertex must have the same tuple of $n$-labels, up to reordering.
\end{proof}
\begin{lemma}\label{lemmainductionstep}
In a multi-labeling of an $l$-graph with $l \geq 4$, suppose a 
$p$-vertex $V$
is such that its $p$-label appears on no other $p$-vertices. Let the $n$-vertex
preceding $V$ be $U$, the $p$-vertex preceding $U$ be $T$, the $n$-vertex
following $V$ be $W$, and the $p$-vertex following $W$ be $X$.
\begin{enumerate}
\item If $T$ and $X$ have different $p$-labels, then the graph obtained by
deleting $V$ and $W$ and connecting $U$ to $X$ is an $(l-1)$-graph with valid
multi-labeling.
\item If $T$ and $X$ have the same $p$-label, then the graph obtained by
deleting $U$, $V$, $W$, and $X$ and connecting $T$ to the $n$-vertex after $X$
is an $(l-2)$-graph with valid multi-labeling.
\end{enumerate}
\end{lemma}
\begin{proof}
First consider case (1). As $T$ and $X$ have distinct $p$-labels, it remains
true that no two consecutive $p$-vertices in the $(l-1)$-graph have the same
$p$-label, so condition (1) of Definition \ref{defmultilabeling} holds.
Condition (2) of Definition \ref{defmultilabeling} clearly still holds as well.
If $V$ has $p$-label $i$ and $W$ has $n$-labels $(j_1,\ldots,j_d)$, then $U$ has
$n$-labels $(j_1,\ldots,j_d)$ as well, up to reordering, by conditions (2) and
(3) of Definition \ref{defmultilabeling} and the fact
that $V$ is the only $p$-vertex with label $i$. Then in the $(l-1)$-graph
obtained by deleting $V$ and $W$,
the number of edges with $p$-vertex endpoint labeled $i$ and
$n$-vertex endpoint having label $j_s$ for any $s=1,\ldots,d$ is zero, and the
number of edges with $p$-vertex endpoint labeled $i'$ and $n$-vertex endpoint
having label $j'$ is the same as in the original $l$-graph for all other pairs
$(i',j')$. Thus condition (3) of Definition \ref{defmultilabeling} still holds
as well, so the $(l-1)$-graph has a valid multi-labeling.

Now consider case (2). $X$ and the $p$-vertex after $X$ must have different
$p$-labels in the original $l$-graph, by condition (1) of Definition
\ref{defmultilabeling}. As $T$ and $X$ have the same $p$-label, this implies $T$
and the $p$-vertex after $X$ must have different $p$-labels, so condition
(1) of Definition \ref{defmultilabeling} still holds in the $(l-2)$-graph.
Condition (2) of Definition \ref{defmultilabeling} clearly still holds in the
$(l-2)$-graph as well. Suppose $V$ has $p$-label $i_1$, $T$ and $X$ have
$p$-label $i_2$, and $W$ has $n$-labels $(j_1,\ldots,j_d)$. As in case (1),
$U$ must also have $n$-labels $(j_1,\ldots,j_d)$ up to reordering. Then in the
$(l-2)$-graph obtained by deleting $U$, $V$, $W$, and $X$, the
number of edges with $p$-vertex endpoint labeled $i_1$ and $n$-vertex endpoint
having label $j_s$ for any $s=1,\ldots,d$ is zero, the number of edges with
$p$-vertex endpoint labeled $i_2$ and $n$-vertex endpoint having label $j_s$ for
any $s=1,\ldots,d$ is two less than in the original $l$-graph, and the number of
edges with $p$-vertex endpoint labeled $i'$ and $n$-vertex endpoint having label
$j'$ is the same as in the original $l$-graph for all other pairs $(i',j')$.
Hence condition (3) of Definition \ref{defmultilabeling} still holds as well, so
the $(l-2)$-graph has a valid multi-labeling.
\end{proof}

\lemmadistinctmultilabels*
\begin{proof}
We induct on $l$. For $l=2$, a multi-labeling must have $d_1=d_2$ and $m=d_1+2$,
and for $l=3$, a multi-labeling must have $d_1=d_2=d_3$ and $m=d_1+3$, by Lemma
\ref{lemmal23}. The result is then easily verified in these two cases.

Suppose by induction that the result holds for $l-2$ and $l-1$, and consider a
multi-labeling of an $l$-graph with $l \geq 4$.
If each distinct $p$-label appears at least twice,
then there are at most $\frac{l}{2}$ distinct $p$-labels. Lemma
\ref{lemmatwicenlabels} implies there are
at most $\frac{\sum_{s=1}^l d_s}{2}$ distinct $n$-labels, so $m \leq
\frac{l+\sum_{s=1}^l d_s}{2}$, establishing the result.

Thus, suppose that some $p$-vertex $V$ has a label that appears exactly
once, and let $T,U,W,X$ be as in Lemma \ref{lemmainductionstep}. If $T$ and $X$
have different $p$-labels, follow procedure (1) in Lemma
\ref{lemmainductionstep} to obtain a multi-labeling of an $(l-1)$-graph.
This multi-labeling now has $m-1$
total distinct $p$-labels and $n$-labels, and so the induction hypothesis
implies $m-1 \leq \frac{l-1+\sum_{s=1}^l d_s-d}{2}+1$ where $d$ is the number of
$n$-labels of the deleted $n$-vertex $W$. Hence
$m \leq \frac{l+\sum_{s=1}^l d_s}{2}-\frac{d+1}{2}+2 \leq
\frac{l+\sum_{s=1}^l d_s}{2}+1$.

If $T$ and $X$ have the same $p$-label, follow procedure (2) of Lemma
\ref{lemmainductionstep} to obtain a multi-labeling of an $(l-2)$-graph.
This multi-labeling has between $m-d-1$ and $m-1$ (inclusive) total
distinct $p$-labels and $n$-labels, where $d$ is the number of $n$-labels of
the deleted $n$-vertex $W$. The induction hypothesis implies
$m-d-1 \leq \frac{l-2+\sum_{s=1}^l d_s-2d}{2}+1$,
so $m \leq \frac{l+\sum_{s=1}^l d_s}{2}+1$. This completes the induction in both
cases, establishing the desired result.
\end{proof}

\lemmatripleijpairs*
\begin{proof}
We induct on $l$. For $l=2$ or 3, we must have $b_{ij}=0$ or 2 for all $(i,j)$
by Lemma \ref{lemmal23}, and $\Delta \geq 0$ by Lemma
\ref{lemmadistinctmultilabels}, so the result holds.

Suppose the result holds for $l-2$ and $l-1$, and consider a multi-labeling of
an $l$-graph with $l \geq 4$. If each distinct $p$-label appears at
least twice, then there are at most $\frac{l}{2}$ distinct $p$-labels, so Lemma
\ref{lemmafewdistinctp} applies. For any $j$ with $N_j=2$, we have
$b_{ij}=2$ or $b_{ij}=0$ for all $i$,
by conditions (1) and (3) of Definition
\ref{defmultilabeling}. For any $j$ with $N_j \geq 3$, we apply the bound
$\sum_{i:b_{ij}>2} b_{ij} \leq 2N_j$.
Then $\sum_{i,j:b_{ij}>2} b_{ij} \leq 2(6\Delta-6) \leq 12\Delta$ by Lemma
\ref{lemmafewdistinctp}.

Now suppose that some $p$-vertex $V$ has a $p$-label appearing exactly once.
Consider the $(l-1)$-graph or $(l-2)$-graph obtained by Lemma
\ref{lemmainductionstep}. In the case of the $(l-1)$-graph, it is easily
verified that $\sum_{i,j:b_{ij}>2} b_{ij}$ is the same as in the original
$l$-graph, so the induction hypothesis implies $\sum_{i,j:b_{ij}>2} b_{ij}
\leq 12\left(\frac{l-1+\sum_{s=1}^l d_s-d}{2}+1-(m-1)\right)
\leq 12\Delta$, where $d \geq 1$ is the number of $n$-labels on the deleted
$n$-vertex $W$.

In the case of the $(l-2)$-graph, suppose the deleted $n$-vertex $W$ (and $U$)
has $d$ $n$-labels, of which $d'$ also appear on an $n$-vertex different from
$W$ and $U$. If $j$ does not appear on $W$ or $U$, then
clearly $b_{ij}$ is the same in the $(l-2)$-graph and the original $l$-graph for
all $i$. If $j$ is one of the $d-d'$ $n$-label values appearing only on
$W$ and $U$, then $b_{ij}=0$ or 2 in both the $(l-2)$-graph and the original
$l$-graph for all $i$. If $j$ is one of the other $d'$ $n$-label values 
appearing on $W$ and $U$, then
in deleting $U$, $V$, $W$, and $X$, we may have reduced $b_{ij}$ by 2 for at
most two distinct values of $i$ (corresponding to the $p$-labels of $V$ and
$X$). This implies that
$\sum_{i:b_{ij}>2} b_{ij}$ reduces by at most 8 for this $j$, with the maximal
reduction occurring if $b_{ij}=4$ for both of these values of $i$ in the
original $l$-graph. Then
by the induction hypothesis, $\sum_{i,j:b_{ij}>2} b_{ij}-8d' \leq 12\left(
\frac{l-2+\sum_{s=1}^l d_s-2d}{2}+1-(m-1-(d-d')\right)$, as the $(l-2)$-graph
has $m-1-(d-d')$ total distinct $n$ and $p$-labels. Then $\sum_{i,j:b_{ij}>2}
b_{ij} \leq 12\left(\frac{l+\sum_{s=1}^l d_s}{2}+1-m-d'\right)+8d' \leq
12\Delta$, so the result holds in this case as well, completing the induction.
\end{proof}

\subsection{Construction of the map $\varphi$}
\begin{definition}\label{defsingle}
In an $l$-graph with a multi-labeling, an $n$-vertex is {\bf single}
if it has only one $n$-label. It is a {\bf good single} if it is single and if
its $n$-label $j$ appears only on single $n$-vertices. Otherwise,
it is a {\bf bad single}.
\end{definition}
\begin{definition}\label{defgoodpair}
In an $l$-graph with a multi-labeling, a pair $(V,V')$ of distinct (not
necessarily consecutive) $n$-vertices is a {\bf good pair} if the following
conditions hold:
\begin{enumerate}
\item $V$ and $V'$ have the same tuple of $n$-labels, up to reordering,
\item $V$ and $V'$ are not single, and
\item $N_j=2$ for each $j$ appearing as an $n$-label on $V$ and $V'$ (i.e.\ this
label $j$ appears on no other $n$-vertices).
\end{enumerate}
If an $n$-vertex $V$ is not single and not part of any good pair,
then $V$ is a {\bf bad non-single}.
\end{definition}

Thus, every $n$-vertex is either a good single, a bad single, a bad
non-single, or part of a good pair.
Conditions (1) and (3) of Definition \ref{defmultilabeling} require that,
if $(V,V')$ is a good pair, then the two (distinct) $p$-labels of the
$p$-vertices preceding and following $V$ are the same
as those of the $p$-vertices preceding and following $V'$ (but not necessarily
in the same order).

\begin{definition}
Suppose $(V,V')$ is a good pair of $n$-vertices. Let the $p$-vertices preceding
and following $V$ be $U$ and $W$, respectively, and let the $p$-vertices
preceding and following $V'$ be $U'$ and $W'$, respectively. Then the good pair
$(V,V')$ is {\bf proper} if $U$ has the same label as $W'$ and $U'$ has the same
label as $W$, and it is {\bf improper} if $U$ has the same label as $U'$ and $W$
has the same label as $W'$.
\end{definition}

\begin{definition}\label{deflabelmap}
The {\bf label-simplifying map} is the map from $(p,n)$-multi-labelings of an
$l$-graph to $(p,n+1)$-simple-labelings of an $l$-graph, defined by the
following procedure:
\begin{enumerate}
\item While there exists an improper good pair of $n$-vertices $(V,V')$, iterate
the following: Let $W$ be the $p$-vertex following $V$ and $W'$ be the
$p$-vertex following $V'$, and reverse the sequence of vertices starting at $W$
and ending at $W'$ (together with their labels).
\item For each $n$-vertex in a good pair, relabel it with the empty label.
\item For each $n$-vertex that is a bad single or a bad non-single,
relabel it with the single label $n+1$.
\end{enumerate}
\end{definition}
\begin{remark}
In the case where there are multiple improper good pairs in step (1) of this
procedure, it will not be important for our later arguments in which order the
pairs $(V,V')$ are selected and which vertex we choose as $V$ and which as $V'$.
For concreteness, we may always select $\{V,V'\}$
to be the improper good pair whose sorted $n$-label-tuple is smallest
lexicographically, and we may take $V$ to come before $V'$ in the $l$-graph
cycle.
\end{remark}
\begin{lemma}\label{lemmalabelmapvalidity}
The following are true for the label-simplifying map in Definition
\ref{deflabelmap}:
\begin{enumerate}
\item Step (1) of the procedure in Definition \ref{deflabelmap} always
terminates in a valid $(p,n)$-multi-labeling with no improper good pairs.
\item The image of any $(p,n)$-multi-labeling under the map is
a valid $(p,n+1)$-simple-labeling.
\item If two multi-labelings are equivalent, then their image simple-labelings
are also equivalent.
\end{enumerate}
\end{lemma}
\begin{proof}
Clearly each reversal in step (1) of the procedure
preserves condition (2) of Definition \ref{defmultilabeling}
as well as the number of good pairs and $n$-labels of each good pair. As $W$ and
$W'$ have the same $p$-label because $(V,V')$ is improper, it also preserves
conditions (1) and (3) of Definition \ref{defmultilabeling}, so the resulting
labeling is still a valid $(p,n)$-multi-labeling. Each time this reversal is
performed, $V$ and $V'$ become consecutive $n$-vertices in the $l$-graph, and
the pair $(V,V')$ becomes a proper good pair. As $V$ and $V'$ are consecutive,
they must remain consecutive under each subsequent reversal, so their
properness is preserved. Hence the procedure must
terminate after a number of iterations at
most the total number of good pairs in the multi-labeling, and the final
multi-labeling is such that all good pairs are proper. This establishes (1).

To prove (2), note that the image labeling has either one $n$-label or the
empty label for each $n$-vertex. Condition (1) of
Definition \ref{defsimplelabeling} holds for the image labeling by condition
(1) of Definition \ref{defmultilabeling}, as the
$p$-labels are preserved. As all good pairs in the multi-labeling obtained after
applying step (1) of the procedure are proper, and step (2) of the procedure
maps their labels to the empty label, condition (3) of Definition
\ref{defsimplelabeling} holds for the image labeling. Finally, note
that if $j$ is an $n$-label appearing on good single
vertices in the multi-labeling, then condition (2) of Definition
\ref{defsimplelabeling} holds in the image labeling for this $j$ and all
$p$-labels $i$ by condition (3) in Definition \ref{defmultilabeling}.
For the new $n$-label $n+1$ created in step (3) of the map, note that for each
$i \in \{1,2,3,\ldots\}$ there must be an even number of
edges in the $l$-graph with $p$-endpoint
labeled $i$. Of these, there must be an even number with $n$-endpoint $j$ for
any good single label $j$, by the above argument, and there must also be an
even number with $n$-endpoint belonging to a good pair 
since these edges must come in pairs. Hence the number of remaining edges
adjacent to any $p$-vertex with label $i$ must also be even. These are precisely
the edges with $p$-endpoint labeled $i$ and $n$-endpoint labeled
$n+1$ in the image labeling, so condition (2) of Definition
\ref{defsimplelabeling} holds for the new $n$-label $n+1$ and all $p$-labels
$i$ as well. Hence the image labeling is a valid $(p,n+1)$-simple-labeling,
establishing (2).

(3) is evident, as equivalent multi-labelings have the same proper and improper
good pairs of $n$-vertices and the same good single $n$-vertices.
\end{proof}

\begin{definition}\label{defphi}
Let $\mathcal{C}$ and $\tilde{\mathcal{C}}$ be the set of all multi-labeling
equivalence classes and simple-labeling equivalence classes, respectively,
of an $l$-graph. For $\mathcal{L} \in \mathcal{C}$ and any multi-labeling in
$\mathcal{L}$, let $\tilde{\mathcal{L}} \in \tilde{\mathcal{C}}$ contain
its image simple-labeling under the label-simplifying map of Definition
\ref{deflabelmap}, and define $\varphi:\mathcal{C} \to \tilde{\mathcal{C}}$ by
$\varphi(\mathcal{L})=\tilde{\mathcal{L}}$.
\end{definition}

\subsection{Verification of Proposition \ref{proplabelmapping}, properties
(1) and (2)}
For the map $\varphi$ of Definition \ref{defphi},
property (1) of Proposition \ref{proplabelmapping} is evident as the $p$-labels
are preserved. We verify property (2) by bounding the number of bad non-single
$n$-vertices.

For each pair $i,i' \in \{1,2,3,\ldots\}$ with $i<i'$, and for a given
multi-labeling, let us denote
\[P_{i,i'}:=\text{number of appearances of } i,i' \text{ as the }
p\text{-labels of two consecutive } p\text{-vertices (in some order)}.\]

\begin{lemma}\label{lemmaconsecutivep}
Suppose a multi-labeling of an $l$-graph has excess $\Delta$. Then
\[\sum_{i<i':P_{i,i'} \geq 3} P_{i,i'} \leq 42\Delta.\]
\end{lemma}
\begin{proof}
We induct on $l$. For $l=2$ and 3, $P_{i,i'}=0$ or 1 for all pairs $i<i'$,
and $\Delta \geq 0$ by Lemma
\ref{lemmadistinctmultilabels}, so the result holds.

Suppose by induction that the result holds for $l-2$ and $l-1$, and consider a
multi-labeling of an $l$-graph with $l \geq 4$. First suppose each distinct
$p$-label appears at least twice, so there are at most $\frac{l}{2}$ distinct
$p$-labels. If an $n$-label $j$ is such that $N_j=2$, then the pairs of
$p$-vertices before and after the two $n$-vertices with label $j$ must have the
same pairs of $p$-labels, by conditions (1) and (3) of Definition
\ref{defmultilabeling}. Thus the number of pairs $i<i'$ with $P_{i,i'}=1$
is at most
the number of $n$-vertices for which $N_j \geq 3$ for all of its $n$-labels $j$.
This is at most $6\Delta$ by Lemma \ref{lemmafewdistinctp}.
On the other hand, the number of distinct $p$-labels
is at most one more than the number of distinct pairs of consecutive
$p$-labels. (This is easily seen by considering the
undirected graph with vertices $\{1,\ldots,p\}$ having an edge between
$i,i'$ if and only if some consecutive pair of $p$-vertices
have labels $i$ and $i'$, and noting that this graph is connected.)
Lemma \ref{lemmatwicenlabels} implies
there are at most $\frac{\sum_{s=1}^l d_s}{2}$ distinct $n$-labels,
and hence at least $m-\frac{\sum_{s=1}^l d_s}{2}-1=\frac{l}{2}-\Delta$
distinct pairs $i<i'$ of consecutive $p$-labels. At least
$\frac{l}{2}-7\Delta$ of these have $P_{i,i'} \geq 2$.
If $c$ of these have $P_{i,i'}=2$, then
$2c+3\left(\frac{l}{2}-7\Delta-c\right) \leq l$,
so $c \geq \frac{l}{2}-21\Delta$. These account for at least $l-42\Delta$
pairs of consecutive $p$-vertices, implying that at most $42\Delta$ pairs of
consecutive $p$-vertices remain. This establishes the result in this case.

Now suppose that there is some $p$-vertex $V$ whose $p$-label appears only
once. Consider the $(l-1)$-graph or $(l-2)$-graph obtained by Lemma
\ref{lemmainductionstep}. It is easily verified that $\sum_{i<i':P_{i,i'} \geq
3} P_{i,i'}$ is the same in this graph as in the original $l$-graph, because
if $P_{i,i'} \geq 3$ in the original $l$-graph, then neither $i$ nor $i'$
can be the $p$-label of $V$. On the other hand, our proof of Lemma
\ref{lemmadistinctmultilabels} verified that this $(l-1)$-graph or $(l-2)$-graph
has excess at most that of the original $l$-graph, so the
desired result follows from the induction hypothesis.
\end{proof}

The next lemma bounds the number of bad non-single $n$-vertices, i.e.\ it
shows that in any multi-labeling with small excess $\Delta$,
most of the non-single $n$-vertices must belong to a good pair.

\begin{lemma}\label{lemmagoodpairscount}
Suppose a multi-labeling of an $l$-graph has excess $\Delta$ and $k$
single $n$-vertices. Then there are at least $\frac{l-k}{2}-48\Delta$ good
pairs of $n$-vertices.
\end{lemma}
\begin{proof}
Let $m$ be the number of distinct $n$ and $p$-labels and let $d_1,\ldots,d_l$ be
the numbers of $n$-labels on the $l$ $n$-vertices. We induct on $l$.
If $l=2$, then Lemma \ref{lemmal23} implies 
$d_1=d_2$, $m=d_1+2$, and $\Delta=0$. If $d_1=d_2=1$, then $k=2$ and there are
no good pairs, and if $d_1=d_2 \geq 2$, then $k=0$ and there is one good pair.
Hence the result holds. If $l=3$, then Lemma \ref{lemmal23} implies 
$d_1=d_2=d_3$, $m=d_1+3$, and $\Delta=\frac{d_1-1}{2}$. If $d_1=d_2=d_3=1$, then
$k=3$, $\Delta=0$, and there are no good pairs. If $d_1=d_2=d_3 \geq 2$, then
$k=0$, $\Delta \geq \frac{1}{2}$, and there are still no good pairs. In either
case, the result also holds.

Consider $l \geq 4$, and
assume by induction that the result holds for $l-2$ and $l-1$.
First suppose each distinct $p$-label appears at least twice,
so there are at most $\frac{l}{2}$ distinct $p$-labels. By Lemma
\ref{lemmafewdistinctp} there are at most $6\Delta$ $n$-vertices with some
$n$-label $j$ such that $N_j \geq 3$, so there are at least $l-k-6\Delta$
non-single $n$-vertices for which each of its $n$-labels $j$ has $N_j=2$.
Let $V$ be one such $n$-vertex. We consider three cases:

{\bf Case 1:} $V$ has two $n$-labels $j_1$ and $j_2$ that
appear on two different other $n$-vertices $W_1$ and $W_2$. Then
Definition \ref{defmultilabeling} implies that the three pairs
of consecutive $p$-vertices around $V$, $W_1$, and $W_2$ must have the same
pair of $p$-labels. By Lemma \ref{lemmaconsecutivep},
there are at most $42\Delta$ such $n$-vertices $V$.

{\bf Case 2:} All $n$-labels of $V$ appear on a single other $n$-vertex
$W_1$, but $W_1$ has some additional $n$-label $j$ not appearing on
$V$. Then either all such additional $n$-labels $j$ have $N_j \geq 3$, or there
is some such $j$ with $N_j=2$. In the former case, the number of
such vertices $W_1$ is at most $6\Delta$ by Lemma \ref{lemmafewdistinctp}. As
$V$ is the unique $n$-vertex sharing an $n$-label $j$ with $W_1$ for which
$N_j=2$, this implies the number of such vertices $V$ is also at most
$6\Delta$. In the latter case, $j$ appears on a vertex $W_2$ distinct from $V$
and $W_1$. Then the three pairs of
$p$-vertices around $V$, $W_1$, and $W_2$ must have the same pair of
$p$-labels, and by Lemma \ref{lemmaconsecutivep} the number of
such vertices $V$ is at most $42\Delta$. Hence the number of $n$-vertices $V$
belonging to this case is at most $48\Delta$

{\bf Case 3:} $V$ forms a good pair with some other vertex $V'$.
By the bounds in cases 1 and 2, there are at least $l-k-96\Delta$ such 
vertices $V$, hence at least $\frac{l-k}{2}-48\Delta$ good pairs,
and the result holds.

Now suppose there is some $p$-vertex $V$ whose $p$-label appears
only once. Let $T,U,W,X$ be as in Lemma \ref{lemmainductionstep},
and recall that $U$ and $W$ have the same $n$-labels up to reordering.
Consider four cases:

{\bf Case 1:} $T$ and $X$ have different $p$-labels, and $U$ and $W$ are single.
Lemma \ref{lemmainductionstep} yields an $(l-1)$-graph with $k-1$ single
$n$-vertices, $\sum_{s=1}^l d_s-1$ total
$n$-labels, and $m-1$ total distinct $p$- and $n$-labels. By the
induction hypothesis, this $(l-1)$-graph has at least
\[\frac{(l-1)-(k-1)}{2}-48\left(\frac{(l-1)+(\sum_{s=1}^l d_s-1)}{2}+1-(m-1)
\right)=\frac{l-k}{2}-48\Delta\]
good pairs, which are also good pairs in the $l$-graph.

{\bf Case 2:} $T$ and $X$ have different $p$-labels, and $U$ and $W$ each have
$d \geq 2$ $n$-labels. Lemma \ref{lemmainductionstep} yields an $(l-1)$-graph
with $k$ single $n$-vertices,
$\sum_{s=1}^l d_s-d$ total $n$-labels, and $m-1$ distinct $p$- and
$n$-labels. By the induction hypothesis, this $(l-1)$-graph has at 
least
\[\frac{(l-1)-k}{2}-48\left(\frac{(l-1)+(\sum_{s=1}^l d_s-d)}{2}+1-(m-1)
\right)>\frac{l-k}{2}-48\Delta+1\]
good pairs. It can have at most one more good
pair than the original $l$-graph (which occurs if $W$
has a tuple of $n$-labels appearing on exactly three different $n$-vertices in
the $l$-graph).

{\bf Case 3:} $T$ and $X$ have the same $p$-label, and $U$ and $W$ are single.
Lemma \ref{lemmainductionstep} yields an $(l-2)$-graph with $k-2$ single
$n$-vertices, $\sum_{s=1}^l d_s-2$ total
$n$-labels, and either $m-2$ distinct $p$- and $n$-labels if $U$ and $W$
have an $n$-label appearing only those two
times, or $m-1$ distinct $p$- and $n$-labels otherwise. Supposing the
former, this $(l-2)$-graph has at least
\[\frac{(l-2)-(k-2)}{2}-48\left(\frac{(l-2)+(\sum_{s=1}^l d_s-2)}{2}+1-(m-2)
\right)=\frac{l-k}{2}-48\Delta\]
good pairs, and it has the
same number of good pairs as the original $l$-graph. Supposing the latter,
this $(l-2)$-graph has at least
\[\frac{(l-2)-(k-2)}{2}-48\left(\frac{(l-2)+(\sum_{s=1}^l d_s-2)}{2}+1-(m-1)
\right)>\frac{l-k}{2}-48\Delta+1\]
good pairs, and it can have at most one more good pair than the original
$l$-graph (which occurs if the
$(l-2)$-graph has a good pair containing the $n$-label
of the removed vertices $U$ and $W$).

{\bf Case 4:}
$T$ and $X$ have the same $p$-label, and $U$ and $W$ each have $d \geq 2$
$n$-labels. Lemma \ref{lemmainductionstep} yields an
$(l-2)$-graph with $k$ single $n$-vertices,
$\sum_{s=1}^l d_s-2d$ total $n$-labels, and between $m-d-1$ and
$m-1$ (inclusive) distinct $p$- and $n$-labels. If it has exactly $m-d-1$
distinct $p$- and $n$-labels, then we must have removed a good pair, and
the $(l-2)$-graph has at
least
\[\frac{(l-2)-k}{2}-48\left(\frac{(l-2)+(\sum_{s=1}^l d_s-2d)}{2}+1-
(m-d-1)\right)=\frac{l-k}{2}-48\Delta-1\]
good pairs. If,
instead, the $(l-2)$-graph has $m-c-1$ distinct $p$- and $n$-labels for
$0 \leq c < d$, then $U$ and
$W$ cannot be a good pair in the original $l$-graph as they have $d-c$
$n$-labels $j$ for which $N_j \geq 3$, and the
$(l-2)$-graph can have at most $d-c$ more good pairs than the $l$-graph, one
for each such $j$. The $(l-2)$-graph has at least
\[\frac{(l-2)-k}{2}-48\left(\frac{(l-2)+(\sum_{s=1}^l d_s-2d)}{2}+1
-(m-c-1)\right)>\frac{l-k}{2}-48\Delta+d-c\]
good pairs.
In all cases, we establish that the $l$-graph has at least
$\frac{l-k}{2}-48\Delta$ good pairs, completing the induction.
\end{proof}

\begin{proof}[Proof of Proposition \ref{proplabelmapping}, property (2)]
Let $\mathcal{L} \in \mathcal{C}$ be any multi-labeling equivalence class.
Let $\varphi(\mathcal{L})$ have $\tilde{k}$ $n$-vertices with non-empty label.
This means $\mathcal{L}$ has $\tilde{k}$ $n$-vertices that do not belong to a
good pair. These vertices have at least $\tilde{k}$ total $n$-labels in
$\mathcal{L}$, implying that there are at most $\sum_{s=1}^l d_s-\tilde{k}$
total $n$-labels on the good pair vertices.
These good pair vertices account for at most
$\frac{\sum_{s=1}^l d_s-\tilde{k}}{2}$ distinct $n$-labels in $\mathcal{L}$,
and these are mapped to the empty label under the label-simplifying map.
Furthermore,
by Lemma \ref{lemmagoodpairscount}, there are at most $96\Delta(\mathcal{L})$
bad non-single $n$-vertices,
and these have at most $96D\Delta(\mathcal{L})$ additional
distinct $n$-labels that are mapped to the new $n$-label $n+1$.
Any bad single $n$-vertex has
an $n$-label that is the same as one of these $96D\Delta(\mathcal{L})$ distinct
$n$-labels (otherwise it is a good single by definition), and the $n$-label of
any good single $n$-vertex is preserved under the label-simplifying map.
Hence, if $m$ is the number of total
distinct $p$- and $n$-labels in $\mathcal{L}$ and $\tilde{m}$ is the
number of total distinct $p$-labels and non-empty $n$-labels in
$\varphi(\mathcal{L})$, then
$\tilde{m} \geq m-\frac{\sum_{s=1}^l d_s-\tilde{k}}{2}-96D\Delta(\mathcal{L})$,
so $\tilde{\Delta}(\varphi(\mathcal{L}))
=\frac{l+\tilde{k}}{2}+1-\tilde{m} \leq (96D+1)\Delta(\mathcal{L})$. Hence
property (2) holds.
\end{proof}

\subsection{Verification of Proposition \ref{proplabelmapping}, property (3)}
Recall that we order the vertices of an $l$-graph according to a cyclic
traversal starting from a (arbitrary) $p$-vertex.
\begin{definition}
The {\bf canonical simple labeling} in a simple labeling equivalence class
$\tilde{\mathcal{L}}$ is the one in which each
$i^\text{th}$ new $p$-vertex label that appears in the cyclic traversal is $i$,
and each $j^\text{th}$ new non-empty $n$-vertex label is $j$.

The {\bf canonical multi-labeling} in a multi-labeling equivalence class
$\mathcal{L}$ is the one in which each $i^\text{th}$ new $p$-vertex label is
$i$ and each $j^\text{th}$ new $n$-vertex label is $j$, with the new $n$-vertex
labels in the label-tuple for each $n$-vertex appearing in sorted order.
\end{definition}

Each $\tilde{\mathcal{L}}$ has a unique canonical simple-labeling, which is
an $(l,l)$-simple labeling, and each $\mathcal{L}$ has a unique canonical
multi-labeling, which is an $(l,Dl)$-multi-labeling.

For each $\tilde{\mathcal{L}} \in \tilde{\mathcal{C}}$ and $\Delta_0 \geq 0$,
property (3) of Proposition \ref{proplabelmapping} is a bound on a certain
weighted cardinality of the set
\[\mathcal{S}(\Delta_0,\tilde{\mathcal{L}}):=\varphi^{-1}(\tilde{\mathcal{L}})
\cap \{\mathcal{L}:\Delta(\mathcal{L})=\Delta_0\}.\]
We describe a series of non-determined steps by which the mapping $\varphi$ may
be ``inverted'' to obtain the canonical multi-labeling $L$
of any $\mathcal{L} \in \varphi^{-1}(\tilde{\mathcal{L}})$, given
$\tilde{\mathcal{L}}$:
\begin{enumerate}
\item Choose a non-empty $n$-label value appearing in $\tilde{\mathcal{L}}$ to
be ``n+1'', or assume there is no such label. (The $n$-vertices with empty label
will be the good pairs, and the remaining $n$-vertices with label different
from ``n+1'' will be the good singles.)
\item Choose a subset $S$ of $n$-vertices with label ``n+1'' to be
the bad non-singles in $L$. (The remaining $n$-vertices with label ``n+1''
will be the bad singles.)
\item For each $n$-vertex in $S$, choose the size of its
$n$-label tuple in $L$ to be between 2 and $D$ (inclusive),
and pick $n$-labels from $\{1,\ldots,Dl\}$ for that tuple.
\item For each $n$-vertex with label ``n+1'' not in $S$,
pick a single value in $\{1,\ldots,Dl\}$ for its $n$-label in $L$.
\item For all $n$-vertices with empty label in $\tilde{\mathcal{L}}$, pair
them up into good pairs for $L$.
\item For each good pair, choose the size of its $n$-label tuple in
$L$ to be between 2 and $D$ (inclusive), and choose a permutation of
the second $n$-label tuple of the pair that matches the first.
\item Let $\mathcal{G}$ be the set of good pairs $(V,V')$
that are consecutive $n$-vertices in the $l$-graph and such that the $p$-label
(in $\tilde{\mathcal{L}}$) of the $p$-vertex between them appears at least
twice. Choose an ordered subset of $\mathcal{G}$. For each $(V,V')$ in this
subset, if $W$ is the $p$-vertex between $V$ and $V'$, choose some
other $p$-vertex $W'$ having the same $p$-label as $W$, and reverse the
sequence of vertices from $W$ to $W'$ or from $W'$ to $W$.
\item Choose $p$-labels for $L$ such that the resulting labeling is 
canonical and two $p$-vertices have the same label if and only if they do in 
$\tilde{\mathcal{L}}$. Choose the remaining $n$-labels for $L$
(corresponding to the good pairs and good singles) such that the resulting
labeling is canonical, the properties of Definitions \ref{defsingle}
and \ref{defgoodpair} are satisfied, and two good single vertices have the same
$n$-label if and only if they do in $\tilde{\mathcal{L}}$.
\end{enumerate}
These steps are non-determined in the sense that each step may be performed in
multiple ways, yielding many possible output multi-labelings $L$.
They ``invert'' $\varphi$ in the following sense:

\begin{lemma}
For any $\mathcal{L} \in \varphi^{-1}(\mathcal{\tilde{L}})$,
the canonical multi-labeling $L$ of $\mathcal{L}$ is a possible output of the 
above procedure.
\end{lemma}
\begin{proof}
Let $L^*$ denote the $(l,Dl)$-multi-labeling obtained by applying step (1) of 
the label-simplifying map in Definition \ref{deflabelmap} to $L$. (It is
an $(l,Dl)$-multi-labeling by Lemma \ref{lemmalabelmapvalidity}.)

$L$ may be obtained by the above procedures as follows: Perform steps (1) and
(2) to correctly partition the $n$-vertices into the good pair, good single,
bad single, and bad non-single $n$-vertices of $L^*$. 
Perform steps (3) and (4) to recover the $n$-labels in $L^*$ of the bad single
and bad non-single $n$-vertices. Perform steps (5) and (6) to
correctly identify the good pairs of $L^*$ and the permutation that maps the
label-tuple of the second vertex to that of the first vertex in each pair.
Perform step (7) to invert the reversals that mapped $L$ to $L^*$ (in the
reverse order of how they were applied in the label-simplifying map): This is
possible because each reversal in step (1) of the label-simplifying map
causes an additional good pair $(V,V')$ of $n$-vertices to become consecutive
in the $l$-graph, with the $p$-vertex between them having $p$-label appearing at
least twice, and these three vertices remain consecutive after each subsequent
reversal. Finally, perform step (8) to recover the $p$-labels and the good 
single and good pair $n$-labels of $L$, which is possible because
(by assumption) $L$ is a valid canonical multi-labeling.
\end{proof}

To obtain the desired weighted cardinality bound for
$\mathcal{S}(\Delta_0,\tilde{\mathcal{L}})$, we bound the number of ways 
each of the above 8 steps may be performed such that the 
final output $L$ is the canonical multi-labeling for some
$\mathcal{L} \in \mathcal{S}(\Delta_0,\tilde{\mathcal{L}})$. The bounds for
all but steps (4) and (7) follow from our preceding combinatorial
estimates. The following simple lemma will yield a bound for step (7):
\begin{lemma}\label{lemmareverseoptions}
Suppose a multi-labeling of an $l$-graph has excess $\Delta$. Then there are at
most $2\Delta$ good pairs of $n$-vertices such that the two vertices in the pair
are consecutive in the $l$-graph cycle and the $p$-label of the $p$-vertex
between them appears at least twice in the labeling.
\end{lemma}
\begin{proof}
Call a $p$-vertex ``sandwiched'' if it is between two consecutive
$n$-vertices that form a good pair. Let $i$ be a $p$-label appearing on
a total of $b \geq 2$ $p$-vertices, of which $c \geq 1$ are sandwiched.
If $b>c$, then change the $c$ appearances of $i$ on the sandwiched $p$-vertices
to $c$ new $p$-labels not yet appearing in the labeling. Otherwise if $b=c$ (so
$c \geq 2$), then change $c-1$ appearances of $i$ on the sandwiched $p$-vertices
to $c-1$ new $p$-labels not yet appearing in the labeling.
Do this for every such $i$. Note that changing the $p$-label of any sandwiched
$p$-vertex does not violate any of the 
conditions of Definition \ref{defmultilabeling}, so the resulting labeling is
still a valid multi-labeling. If $x$ is the number of good pairs originally
satisfying the
condition of the lemma, then we have added at least $\frac{x}{2}$ new
$p$-labels to the labeling. Hence Lemma \ref{lemmadistinctmultilabels}
implies $m+\frac{x}{2} \leq \frac{l+\sum_{s=l} d_s}{2}+1$, so $x \leq 2\Delta$.
\end{proof}

The remaining challenge is to bound the number of ways of performing
step (4). This bound is not straightforward because
the number of bad singles is not necessarily small when $\Delta$ is small.
We instead show that the number of bad singles that we may
``freely label'' is small:

\begin{definition}\label{defconnected}
In a multi-labeling of an $l$-graph, $i \in \{1,2,3,\ldots\}$ is
a {\bf connector} if it appears as a $p$-label and, among all $n$-vertices that 
are adjacent to any $p$-vertex with label $i$, exactly two are bad
singles and none are bad non-singles; these two bad singles
are {\bf connected}. A sequence of bad singles
$W_1,\ldots,W_a$ is a {\bf connected cycle} if $W_1$ is connected
to $W_2$, $W_2$ is connected to $W_3$, etc., and $W_a$ is connected to $W_1$.
\end{definition}
Note that ``connector'' refers to a label $i$, not to any specific $p$-vertex
having $i$ as its label, and two ``connected''
bad singles are adjacent to $p$-vertices having the connector
label $i$, but these $p$-vertices may be distinct in the $l$-graph. Each
bad single $n$-vertex may be connected to at most two other bad single
$n$-vertices (where the connectors are the $p$-labels of its two adjacent
$p$-vertices), and hence this notion of connectedness
partitions the set of bad single $n$-vertices into connected components that are
either individual vertices, linear chains, or cycles.

Motivation for this definition comes from the observation that if two bad
single $n$-vertices are connected, then they must have the same $n$-label,
as follows from condition (3) of Definition \ref{defmultilabeling} and the
fact that $n$-labels appearing on good singles and good pairs must be distinct
from those appearing on the remaining $n$-vertices.

\begin{lemma}\label{lemmaconnectedbadsingles}
Suppose a multi-labeling of an $l$-graph has excess $\Delta$ and $k$ single
$n$-vertices, of which $k'$ are good single and $k-k'$ are bad single. Then at
least
$k-k'-(288D+2)\Delta$ distinct $p$-labels are connectors, and there are at most
$(192D+1)\Delta$ connected cycles of bad single $n$-vertices.
\end{lemma}
\begin{proof}
Suppose the multi-labeling is a $(p,n)$-multi-labeling.
Construct an undirected multi-graph $G$ with vertex set
$\{1,\ldots,p\}$, where each edge of $G$ has one label in $\{1,\ldots,n\}$,
as follows: For each $n$-vertex $V$ in the $l$-graph and each $n$-label $j$
of $V$, if $V$ is preceded and
followed by $p$-vertices with labels $i_1$ and $i_2$, then add an edge
$i_1 \sim i_2$ in $G$ with label $j$. (Thus $G$ has
$\sum_{s=1}^l d_s$ total edges.) Condition (3) of Definition
\ref{defmultilabeling} implies for any $j$, each vertex of $G$
has even degree in the sub-graph consisting of only edges with label $j$.

We will sequentially remove edges of $G$ corresponding to good
pairs and good singles, until only edges corresponding to bad singles and bad
non-singles remain. 
At any stage of this removal process, let us call a vertex of $G$
``active'' if there is at least one edge still adjacent to that vertex.
Let us define a ``component'' as the set of active vertices that may be
reached by traversing the remaining edges of $G$ from a particular active
vertex. (Hence a component of $G$ is a connected component, in the standard
sense, that contains at least two vertices.) We will track the quantity
\[M=\#\{\text{active vertices}\}+\#\{\text{distinct edge labels}\}-
\#\{\text{components}\}.\]

Initially, $G$ has $m$ active vertices plus distinct edge labels (where $m$ is
the number of distinct $n$- and $p$-vertices of the $l$-graph), and one
component, so $M=m-1$. Let us remove the edges of $G$ corresponding to good
pairs. If an $n$-vertex of a good pair has $d$ $n$-labels, then
the good pair corresponds to $2d$ edges between a single pair of vertices
in $G$ whose edge labels do not appear elsewhere in $G$.
Removing these $2d$ edges removes $d$ distinct edge labels,
and if this also changes the connectivity structure of $G$, then either
$\#\{\text{components}\}$ increases by 1, $\#\{\text{active vertices}\}$
decreases by 1, or $\#\{\text{components}\}$
decreases by 1 and $\#\{\text{active vertices}\}$ decreases by 2. In all
cases, $M$ decreases by at most $d+1$.
Then after removing all edges of $G$ corresponding to good pairs,
$M \geq m-1-\left(\frac{\sum_{s=1}^l d_s-k}{2}\right)-\left(\frac{l-k}{2}\right)
=k-\Delta$,
as there are at most $\frac{\sum_{s=1}^l d_s-k}{2}$ distinct $n$-labels
for the good pairs and at most $\frac{l-k}{2}$ good pairs.

Let us now remove the edges of $G$ corresponding to good singles. Let $j$ be 
an $n$-label of a good single, and consider removing the edges of $G$ with
label $j$ one at a time.
As each vertex of $G$ has even degree in the subgraph of edges with
label $j$, when the first such edge is removed, the number of components and
active vertices cannot change. Subsequently, the removal of each additional
edge might increase $\#\{\text{components}\}-\#\{\text{active vertices}\}$ by
1 upon considering the same three cases as above. When the last such
edge is removed, there are no longer any edges with label $j$ by the definition
of a good single, so $\#\{\text{distinct edge labels}\}$ decreases by 1. Hence
removing all edges with label $j$ decreases $M$ by at most the number of such
edges, and $M \geq k-k'-\Delta$ after removing the edges corresponding to all
$k'$ good singles.

Call the resulting graph $G'$. Every vertex of $G'$ still has even
degree in the subgraph of edges with label $j$, for any $j$. In
particular, every active vertex of $G'$ has degree at least two.
By Definition \ref{defconnected}, $i \in \{1,\ldots,p\}$ is a 
connector if and only if $i$ has degree exactly two in $G'$, in which
case the edges incident to $i$ in $G'$ must have the same label $j$, and the
$n$-vertices with label $j$ in the $l$-graph are the bad singles connected by
$i$. A connected cycle of bad singles
corresponds to the edges of a cycle of (necessarily distinct) vertices in $G'$
with degree exactly two.

The number of distinct edge labels in $G'$ equals the number of distinct
$n$-labels in the $l$-graph appearing on bad non-singles (as any $n$-label
appearing on a bad single also appears on some bad non-single).
By Lemma \ref{lemmagoodpairscount} this is at most $96D\Delta$.
Hence $\#\{\text{active vertices}\}-\#\{\text{components}\} \geq
k-k'-(96D+1)\Delta$ for $G'$. The number of total
edges in $G'$ is at most $k-k'+96D\Delta$, with $k-k'$ of them corresponding to
bad singles and at most $96D\Delta$ corresponding to bad non-singles.
Then the total vertex degree of $G'$ is at most $2(k-k'+96D\Delta)$. As each
active vertex in $G'$ has degree at least two, this implies
$\#\{\text{active vertices}\} \leq k-k'+96D\Delta$. Then
$\#\{\text{components}\} \leq (192D+1)\Delta$, so there are at most
$(192D+1)\Delta$ connected cycles of bad singles. Furthermore,
if there are $x$ connectors (i.e.\ active vertices with degree exactly two),
then since there are at least $k-k'-(96D+1)\Delta$ active vertices,
$2x+4(k-k'-(96D+1)\Delta-x) \leq 2(k-k'+96D\Delta)$, so $x \geq
k-k'-(288D+2)\Delta$.
\end{proof}
\begin{proof}[Proof of Proposition \ref{proplabelmapping}, property (3)]
Let $C$ denote a positive constant that may depend on $D$ and that may change
from instance to instance. Fix $\Delta_0 \geq 0$ and $\tilde{\mathcal{L}}$.
We upper bound the number of ways in which steps (1)--(8) of the inversion
procedure may be performed, such that the resulting multi-labeling $L$ is
canonical for some $\mathcal{L} \in \mathcal{S}(\Delta_0,\tilde{\mathcal{L}})$:

There are at most $l+1$ ways of performing step (1).

By Lemma \ref{lemmagoodpairscount}, to yield $L$ with excess $\Delta_0$, there
can be at most $C\Delta_0$ bad non-single $n$-vertices, and hence
we must take $|S| \leq C\Delta_0$ in step (2).

To perform step (3), for each vertex in $S$, we may first choose the number
of $n$-labels $d$ between 2 and $D$, and then there are at most
$(Dl)^d$ ways of choosing the $n$-labels for that vertex.

For step (4), suppose $k'$ good single and $k-k'$ bad single
$n$-vertices were identified in steps (1) and (2). By Lemma
\ref{lemmaconnectedbadsingles}, there are at least $k-k'-C\Delta_0$
connectors, and any two connected bad single $n$-vertices 
must be given the same $n$-label. (The $p$-labels of $\tilde{\mathcal{L}}$
are known and are preserved in $L$,
so after steps (1) and (2) we know which labels are
connectors and which bad singles must be connected in $L$.)
Going through the connectors one-by-one, each successive connector constrains
the $n$-label of one more bad single $n$-vertex, unless that
connector closes a connected cycle. But as there are at most
$C\Delta_0$ connected cycles by Lemma \ref{lemmaconnectedbadsingles},
the number of bad single $n$-vertices that we can freely label at most
$C\Delta_0$. Then there are at most $(Dl)^{C\Delta_0}$ ways to perform step (4).

For step (5), recall that the pairs of $p$-vertices
surrounding the two $n$-vertices of a good pair must have the same pair of
$p$-labels. By Lemma \ref{lemmaconsecutivep}, for all but at
most $C\Delta_0$ of the $n$-vertices with empty label, this pairing is
uniquely determined, so there are at most
$(C\Delta_0)^{C\Delta_0}$ ways of performing step (5).

For step (6), there are $(l-\tilde{k}(\tilde{\mathcal{L}}))/2$ good pairs,
and for each pair we may choose the number of $n$-labels $d$ between 2 and $D$
and then one of $d!$ permutations.

Lemma \ref{lemmareverseoptions} shows that $|\mathcal{G}| \leq 2\Delta_0$
for step (7). For each element that we add to the ordered subset of
$\mathcal{G}$, there are at most $2\Delta_0$ choices for this element and
at most $2l$ ways of choosing $W'$ and which half of the cycle to reverse,
or we may choose to not add any more elements. We make such a choice at most
$2\Delta_0$ times, so there are at most $(4\Delta_0 l+1)^{2\Delta_0}$ ways of
performing step (7).

Finally, there is at most one way to perform step (8), as the labels on the good
single and good pair $n$-vertices are distinct from those on the bad single and
bad non-single $n$-vertices, and each new $n$-label and $p$-label has a unique
choice to make $L$ canonical.

We may incorporate the product $\prod_{s=1}^l |a_{d_s}(\mathcal{L})|/
(d_s(\mathcal{L})!)^{1/2}$ on the left side of (\ref{eq:proplabelmapping3})
into the cardinality count by noting that this
product contributes $|a_d|/(d!)^{1/2}$ for each vertex in $S$ having $d$
$n$-labels, $a_d^2/d!$ for each good pair having $d$ $n$-labels per vertex
of the pair, and $|a_1|$ for each of the $\tilde{k}(\tilde{\mathcal{L}})-|S|$
single vertices in $L$. Combining the above bounds then
yields
\begin{align*}
\mathop{\sum_{\mathcal{L} \in \varphi^{-1}(\tilde{\mathcal{L}})}}
_{\Delta(\mathcal{L})=\Delta_0} \prod_{s=1}^l \frac{|a_{d_s}(\mathcal{L})|}
{(d_s(\mathcal{L})!)^{1/2}}
&\leq (l+1)\sum_S \left(\sum_{d=2}^D (Dl)^d\frac{|a_d|}{(d!)^{1/2}}
\right)^{|S|}(Dl)^{C\Delta_0}(C\Delta_0)^{C\Delta_0}\\
&\hspace{1in}\left(\sum_{d=2}^D d!\frac{a_d^2}{d!}
\right)^{\frac{l-\tilde{k}(\tilde{\mathcal{L}})}{2}}(4\Delta_0l+1)^{2\Delta_0}
|a_1|^{\tilde{k}(\tilde{\mathcal{L}})-|S|}\\
&\leq (l+1)(Cl)^{C\Delta_0}|a|^{\tilde{k}(\tilde{\mathcal{L}})}
(\nu-a^2)^{\frac{l-\tilde{k}(\tilde{\mathcal{L}})}{2}}\sum_S |a|^{-|S|}
\left(\sum_{d=2}^D (Dl)^d\frac{|a_d|}{(d!)^{1/2}}\right)^{|S|},
\end{align*}
where $\sum_S$ denotes the sum over all possible sets $S$
selected by step (2), and the second line applies $\Delta_0 \leq Cl$ and
$\sum_{d=2}^D a_d^2=\nu-a^2$. As $|S| \leq C\Delta_0$, this implies
by Cauchy-Schwarz
\[|a|^{-|S|}\left(\sum_{d=2}^D (Dl)^d
\frac{|a_d|}{(d!)^{1/2}}\right)^{|S|} \leq
(Cl)^{C\Delta_0} |a|^{-|S|}\left(\sum_{d=2}^D
\frac{a_d^2}{d!}\right)^{\frac{|S|}{2}} \leq 
(Cl)^{C\Delta_0} \left(\frac{\sqrt{\nu}}{|a|}\right)^{C\Delta_0}.\]
The sum is over at most $l^{C\Delta_0}$ possible sets $S$, so this verifies
condition (3) of the proposition upon noting that
$(Cl)^{C\Delta_0} \leq l^{C_3+C_4\Delta_0}$ for some
constants $C_3,C_4>0$ and all $l \geq 2$.
\end{proof}

\section{Moment bound for a deformed GUE matrix}\label{appendixdeformedGUE}
In this appendix, we prove Proposition \ref{propMmomentbound}. Recall
Definition \ref{defdeformedGUE} of $M$, $W$, $V$, and $Z$, which implicitly
depend on $\tilde{p}$ and $\tilde{n}$. Throughout this
section, we will use $p$ and $n$ in place of $\tilde{p}$ and $\tilde{n}$.
\begin{lemma}\label{lemmaMnorm}
Suppose $n,p \to \infty$ with $p/n \to \gamma$. Then
$\|M\| \to \|\mu_{a,\nu,\gamma}\|$ almost surely
\end{lemma}
\begin{proof}
Recall $M=\sqrt{\frac{\gamma(\nu-a^2)}{p}}W+\frac{a}{n}V$, where
$V=ZZ^T-D$ and $D=\diag(\|Z_i\|_2^2)$. The
empirical spectral distribution of $\frac{1}{n}ZZ^T$ converges weakly almost
surely to $\mu_{\MP,\gamma}$. By a chi-squared tail bound and a union bound,
$\|\frac{1}{n}D-\Id\| \to 0$, so
the empirical spectral distribution of $\frac{a}{n}V$ converges weakly almost
surely to $a(\mu_{\MP,\gamma}-1)$. Furthermore, the maximal distance between
an eigenvalue of $\frac{a}{n}V$ and the support of $a(\mu_{\MP,\gamma}-1)$
converges to 0 almost surely by the results of
\cite{yinetallargest} and \cite{baiyinsmallest}.

Let $V=O\Lambda O^T$ where $O$ is the real
orthogonal matrix that diagonalizes $V$. Then the spectrum of
$M$ is the same as that of $\sqrt{\frac{\gamma(\nu-a^2)}{p}}O^TWO
+\frac{a}{n}\Lambda$, and $O^TWO$ is still distributed as the GUE.
Conditional on $V$, the above arguments and Proposition 8.1
of \cite{capitaineetal} imply
$\|\sqrt{\frac{\gamma(\nu-a^2)}{p}}O^TWO+\frac{a}{n}\Lambda\| \to
\|\mu_{a,\nu,\gamma}\|$ almost surely. As this convergence holds
almost surely in $V$, it holds unconditionally as well.
\end{proof}
\begin{lemma}\label{lemmaGUEwisharttailbound}
Suppose $n,p \to \infty$ with $p/n \to \gamma$, let $l:=l(n)$ be such
that $l(n)/n \to 0$, and let $\mathcal{B}_n$ be any event. Then there
exist positive constants $C:=C_{a,\nu,\gamma}$ and $c:=c_{a,\nu,\gamma}$ such
that $\E[\|M\|^l\I\{\mathcal{B}_n\}] \leq C^l\P[\mathcal{B}_n]+e^{-cn}$
for all large $n$.
\end{lemma}
\begin{proof}
Note
\[\|M\| \leq \sqrt{\frac{\gamma(\nu-a^2)}{p}}\|W\|+\frac{|a|}{n}\|ZZ^T\|
+\frac{|a|}{n}\max_{1 \leq i \leq p} \|Z_i\|_2^2.\]
Applying standard tail bounds (e.g.\ Corollary 2.3.5 of \cite{taobook},
Corollary 5.35 of \cite{vershynin}, and Lemma 1 of \cite{laurentmassart}),
there exist constants $C,\eps>0$ depending on $a,\nu,\gamma$ such
that, for all $t \geq C$ and sufficiently large $n$,
$\P[\|M\|>t] \leq e^{-\eps tn}$. Then we may write
\begin{align*}
\E\left[\|M\|^l\I\{\mathcal{B}_n\}\right]
&=\E\left[\|M\|^l\I\{\mathcal{B}_n\}\I\{\|M\| \leq C\}\right]
+\E\left[\|M\|^l\I\{\mathcal{B}_n\}\I\{\|M\|>C\}\right]\\
&\leq C^l\P[\mathcal{B}_n]+\int_{C^l}^\infty \P\left[\|M\|^l>t\right]dt\\
&=C^l\P[\mathcal{B}_n]+\int_C^\infty \P[\|M\|>s]\cdot ls^{l-1}ds\\
&\leq C^l\P[\mathcal{B}_n]+l\int_C^\infty e^{-\eps sn+(l-1)\log s}ds\\
&\leq C^l\P[\mathcal{B}_n]+l\int_C^\infty e^{-(\eps n-l)s}ds\\
&=C^l\P[\mathcal{B}_n]+\frac{l}{\eps n-l}e^{-(\eps n-l)C}
\end{align*}
for all large $n$. As $l=o(n)$, the result follows upon setting
$c=C\eps/2$.
\end{proof}
\begin{lemma}\label{lemmaGUEwishartmean}
Suppose $n,p \to \infty$ with $p/n \to \gamma$. Then
$\E[\|M\|] \to \|\mu_{a,\nu,\gamma}\|$.
\end{lemma}
\begin{proof}
Lemma \ref{lemmaMnorm} and Fatou's lemma imply
$\liminf \E[\|M\|] \geq \|\mu_{a,\nu,\gamma}\|$. For any $\eps>0$, let
$\mathcal{B}_n=\left\{\|M\|>\|\mu_{a,\nu,\gamma}\|
+\eps\right\}$. Then
\[\E[\|M\|]=\E[\|M\|\I\{\mathcal{B}_n^C\}]
+\E[\|M\|\I\{\mathcal{B}_n\}]
\leq \|\mu_{a,\nu,\gamma}\|+\eps
+\E[\|M\|\I\{\mathcal{B}_n\}].\]
Lemma \ref{lemmaMnorm} implies $\P[\mathcal{B}_n] \to 0$, so
Lemma \ref{lemmaGUEwisharttailbound} (with $l=1$) implies
$\E[\|M\|\I\{\mathcal{B}_n\}] \to 0$ as well. Then $\E[\|M\|] \leq
\|\mu_{a,\nu,\gamma}\|+2\eps$ for all large $n$, and the result follows by
taking $\eps \to 0$.
\end{proof}
\begin{lemma}\label{lemmalipschitzextension}
Suppose $F:\R^d \to \R$ is $L$-Lipschitz on a set $G \subseteq \R^k$, i.e.
$|F(x)-F(y)| \leq L\|x-y\|_2$ for all $x,y \in G$. Let $\xi \sim N(0,I_d)$. Then
there exists a function $\tilde{F}:\R^d \to \R$ such that $\tilde{F}(x)=F(x)$
for all $x \in G$, $|\tilde{F}(x)-\tilde{F}(y)| \leq L\|x-y\|_2$ for all
$x,y \in \R^k$, and, for all $\Delta>0$,
\[\P[F(\xi)-\E F(\xi) \geq \Delta+|\E F(\xi)-\E \tilde{F}(\xi)| \text{ and }
\xi \in G] \leq e^{-\frac{\Delta^2}{2L^2}}.\]
\end{lemma}
\begin{proof}
Let $\tilde{F}(x)=\inf_{x' \in G} (F(x')+L\|x-x'\|_2)$. Note that if $x \in G$,
then $F(x) \leq F(x')+L\|x-x'\|_2$ for all $x' \in G$, so $\tilde{F}(x)=F(x)$.
Also, for any $x,y \in \R^k$ and $\eps>0$, there exists $x' \in G$ such that
$\tilde{F}(x) \geq F(x')+L\|x-x'\|_2-\eps$. Then by definition,
$\tilde{F}(y) \leq F(x')+L\|y-x'\|_2$, so $\tilde{F}(y)-\tilde{F}(x) \leq
L\|y-x'\|_2-L\|x-x'\|_2+\eps \leq L\|x-y\|_2+\eps$. Similarly,
$\tilde{F}(x)-\tilde{F}(y) \leq L\|x-y\|_2+\eps$. This holds for all
$\eps>0$, so $|\tilde{F}(x)-\tilde{F}(y)| \leq L\|x-y\|_2$. Finally, applying
Gaussian concentration of measure for the Lipschitz function $\tilde{F}$,
\begin{align*}
&\P[F(\xi)-\E F(\xi) \geq \Delta+|\E F(\xi)-\E \tilde{F}(\xi)|
\text{ and } \xi \in G]\\
&\hspace{0.5in}=
\P[\tilde{F}(\xi) \geq \Delta+|\E F(\xi)-\E \tilde{F}(\xi)|+\E F(\xi)
\text{ and } \xi \in G]\\
&\hspace{0.5in}\leq
\P[\tilde{F}(\xi) \geq \Delta+\E \tilde{F}(\xi)]
\leq e^{-\frac{\Delta^2}{2L^2}}.
\end{align*}
\end{proof}
\begin{lemma}\label{lemmaGUEwishartconcentration}
Suppose $n,p \to \infty$ with $p/n \to \gamma$, and let $\eps>0$. Then
there exist $c:=c_{a,\nu,\gamma}>0$ and $N:=N_{a,\nu,\gamma,\eps}>0$ and a set
$G:=G_{n,p} \subset \R^{p \times n}$ with
$\P[Z \in G] \geq 1-2e^{-\frac{n}{2}}$, such that for all $t>\eps$ and $n>N$,
\[\P[\|M\| \geq \|\mu_{a,\nu,\gamma}\|+t \text{ and } Z \in G]
\leq e^{-cnt^2}.\]
\end{lemma}
\begin{proof}
Recall $M=\sqrt{\frac{\gamma(\nu-a^2)}{p}}W+\frac{a}{n}(ZZ^T
-\diag(\|Z_i\|_2^2))$. Denote
\[\mathcal{W}=\big((w_{ii})_{1 \leq i \leq p},(\sqrt{2} \Re w_{ij},\sqrt{2}
\Im w_{ij})_{1 \leq i<j \leq p}\big) \in \R^{p^2},\]
so that the entries of $\mathcal{W}$ and $Z$ are IID $\N(0,1)$.
Define $f:\R^{p^2+np} \to \R$ and
$f_v:\R^{p^2+np} \to \R$ for $v \in \C^p$ by
$f(\mathcal{W},Z)=\|M\|$ and $f_v(\mathcal{W},Z)=v^*Mv$,
so that $f(\mathcal{W},Z)=\sup_{v \in \C^p:\|v\|_2=1}
|f_v(\mathcal{W},Z)|$. By elementary calculations, and denoting $Z_i$ as the
$i^\text{th}$ row of $Z$,
\begin{align*}
&\frac{\partial f_v(\mathcal{W},Z)}{\partial w_{ii}}
=\sqrt{\frac{\gamma(\nu-a^2)}{p}}|v_i|^2,\;\;
\frac{\partial f_v(\mathcal{W},Z)}{\partial (\sqrt{2}\Re w_{ij})}
=\sqrt{\frac{2\gamma(\nu-a^2)}{p}}\Re(\overline{v_i}v_j),\\
&\frac{\partial f_v(\mathcal{W},Z)}{\partial (\sqrt{2}\Im w_{ij})}
=-\sqrt{\frac{2\gamma(\nu-a^2)}{p}}\Im(\overline{v_i}v_j),\;\;
\nabla_{Z_i} f_v(\mathcal{W},Z)=\frac{2a}{n}\mathop{\sum_{j=1}^p}_{j
\neq i} \Re(\overline{v_i}v_j)Z_j.
\end{align*}
Then, for any $v \in \C^p$ such that $\|v\|_2=1$,
\begin{align*}
\|\nabla f_v(\mathcal{W},Z)\|_2^2
&=\frac{\gamma(\nu-a^2)}{p}\left(\sum_{i=1}^p |v_i|^4+2\sum_{1 \leq i<j \leq p}
|\overline{v_i}v_j|^2\right)
+\frac{4a^2}{n^2}\sum_{i=1}^p \left\|\mathop{\sum_{j=1}^p}_{j \neq i}
\Re(\overline{v_i}v_j)Z_j\right\|_2^2\\
&\leq \frac{\gamma(\nu-a^2)}{p}\left(\sum_{i=1}^p |v_i|^2\right)^2
+\frac{4a^2}{n^2} \sum_{i=1}^p |v_i|^2 \|Z\|^2\|v\|_2^2\\
&=\frac{\gamma(\nu-a^2)}{p}+\frac{4a^2\|Z\|^2}{n^2}.
\end{align*}
Take $G=\{Z \in \R^{p \times n}:\|Z\| \leq 2\sqrt{n}+\sqrt{p}\}$. Then by
Corollary 5.35 of \cite{vershynin}, $\P[Z \notin G] \leq 2e^{-\frac{n}{2}}$. As
$\R^{p^2} \times G$ is convex, the above inequality implies
$f_v(\mathcal{W},Z)$ is $L$-Lipschitz on
$\R^{p^2} \times G$ for $L=O(n^{-1/2})$. Then
\begin{align*}
f(\mathcal{W},Z)-f(\mathcal{W}',Z') &\leq \sup_{v \in \C^p:\|v\|_2=1}
\big(|f_v(\mathcal{W},Z)|-|f_v(\mathcal{W}',Z')|\big)\\
&\leq \sup_{v \in \C^p:\|v\|_2=1} \big|f_v(\mathcal{W},Z)-f_v(\mathcal{W}',Z')
\big| \leq L\|(\mathcal{W},Z)-(\mathcal{W}',Z')\|_2
\end{align*}
for all $\mathcal{W},\mathcal{W}' \in \R^{p^2}$ and $Z,Z' \in G$, so
$f$ is also $L$-Lipschitz on $\R^{p^2} \times G$.

Let $\tilde{f}:\R^{p^2+np} \to \R$ be the $L$-Lipschitz extension of
$f$ on $\R^{p^2} \times G$ given by Lemma \ref{lemmalipschitzextension}.
Note that
\begin{align*}
|\E f(\mathcal{W},Z)-\E \tilde{f}(\mathcal{W},Z)|
&=|\E[(f(\mathcal{W},Z)-\tilde{f}(\mathcal{W},Z))\I\{Z \notin G\}]|\\
&\leq \E|f(\mathcal{W},Z)\I\{Z \notin G\}|
+\E|\tilde{f}(\mathcal{W},Z)\I\{Z \notin G\}|.
\end{align*}
Lemma \ref{lemmaGUEwisharttailbound} (with $l=1$) implies
$\E|f(\mathcal{W},Z)\I\{Z \notin G\}|=\E[\|M\|\I\{Z \notin G\}]=o(1)$. As
$\tilde{f}$ is $L$-Lipschitz,
\[|\tilde{f}(\mathcal{W},Z)| \leq |\tilde{f}(0,0)|+L\|(\mathcal{W},Z)\|_2
=|f(0,0)|+L\|(\mathcal{W},Z)\|_2=L\|(\mathcal{W},Z)\|_2.\]
Let $\mathcal{A}_n=\left\{\left\|(\mathcal{W},Z)\right\|_2 \leq \sqrt{2(p^2+np)}
\right\}$. As $\|(\mathcal{W},Z)\|_2^2$ is chi-squared distributed with $p^2+np$
degrees of freedom, a standard tail bound gives
$\P\left[\left\|(\mathcal{W},Z)\right\|_2^2 \geq p^2+np+t\right]
\leq e^{-\frac{t^2}{8(p^2+np)}}$. Then
\begin{align*}
\E\left[\left\|(\mathcal{W},Z)\right\|_2^2 \I\{\mathcal{A}_n^C\}\right]
&=\int_{p^2+np}^\infty \P\left[\left\|(\mathcal{W},Z)\right\|_2^2
\geq p^2+np+t\right] dt \leq \int_{p^2+np}^\infty e^{-\frac{t^2}{8(p^2+np)}}dt\\
&=2\sqrt{p^2+np}\int_{\frac{\sqrt{p^2+np}}{2}}^\infty e^{-\frac{s^2}{2}}ds
\sim 4e^{-\frac{p^2+np}{8}}.
\end{align*}
This implies
\begin{align*}
\E|\tilde{f}(\mathcal{W},Z)\I\{Z \notin G\}|
&\leq \E[|\tilde{f}(\mathcal{W},Z)|\I\{Z \notin G\}\I\{\mathcal{A}_n\}]
+\E[|\tilde{f}(\mathcal{W},Z)|\I\{Z \notin G\}\I\{\mathcal{A}_n^C\}]\\
&\leq L\sqrt{2(p^2+np)}\P[Z \notin G]
+L\E\left[\left\|(\mathcal{W},Z)\right\|_2^2\I\{\mathcal{A}_n^C\}\right]^{1/2}
=o(1).
\end{align*}
Then $|\E f(\mathcal{W},Z)-\E \tilde{f}(\mathcal{W},Z)|=o(1)$, so
Lemmas \ref{lemmaGUEwishartmean} and \ref{lemmalipschitzextension} imply, for
all $t>\eps$ and all sufficiently large $n$ (i.e. $n>N_{a,\nu,\gamma,\eps}$
independent of $t$),
\begin{align*}
&\P[\|M_{n,p}\| \geq \|\mu_{a,\nu,\gamma}\|+t \text{ and } Z \in G]\\
&\hspace{0.5in}\leq
\P\left[\|M_{n,p}\|-\E\|M_{n,p}\| \geq t-\tfrac{\eps}{2}+|\E f(\mathcal{W},Z)
-\E\tilde{f}(\mathcal{W},Z)| \text{ and } Z \in G \right]\\
&\hspace{0.5in}\leq e^{-\frac{(t-\eps/2)^2}{2L^2}}\leq e^{-\frac{t^2}{8L^2}}.
\end{align*}
The result follows upon noting that $L=O(n^{-1/2})$.
\end{proof}

\begin{proof}[Proof of Proposition \ref{propMmomentbound}]
Let $c>0$ and $G \subset \R^{p \times n}$ be as in Lemma
\ref{lemmaGUEwishartconcentration}. Then, for any $\eps>0$,
\begin{align*}
\E[\|M\|^l\I\{Z \in G\}] &\leq (\|\mu_{a,\nu,\gamma}\|+\eps)^l
+\E\left[\|M\|^l\I\{\|M\| \geq \|\mu_{a,\nu,\gamma}\|+\eps\}
\I\{Z \in G\}\right]\\
&=(\|\mu_{a,\nu,\gamma}\|+\eps)^l+\int_{(\|\mu_{a,\nu,\gamma}\|+\eps)^l}^\infty
\P\left[\|M\|^l \geq t \text{ and } Z \in G \right]dt\\
&=(\|\mu_{a,\nu,\gamma}\|+\eps)^l+\int_{\|\mu_{a,\nu,\gamma}\|+\eps}^\infty
\P[\|M\| \geq s \text{ and } Z \in G] \cdot ls^{l-1}ds\\
&\leq (\|\mu_{a,\nu,\gamma}\|+\eps)^l+l\int_\eps^\infty e^{-cns^2}
(\|\mu_{a,\nu,\gamma}\|+s)^{l-1}ds
\end{align*}
for all sufficiently large $n$, where we have applied Lemma
\ref{lemmaGUEwishartconcentration}. Note that
\begin{align*}
l\int_\eps^\infty e^{-cns^2}(\|\mu_{a,\nu,\gamma}\|+s)^{l-1}ds
&\leq l\int_\eps^\infty e^{-cns^2+l(\|\mu_{a,\nu,\gamma}\|+s)}ds\\
&=le^{l\|\mu_{a,\nu,\gamma}\|+\frac{l^2}{4cn}}\int_\eps^\infty
e^{-cn\left(s-\frac{l}{2cn}\right)^2}ds\\
&=\frac{le^{l\|\mu_{a,\nu,\gamma}\|+\frac{l^2}{4cn}}}{\sqrt{2cn}}
\int_{\sqrt{2cn}\left(\eps-\frac{l}{2cn}\right)}^\infty e^{-\frac{t^2}{2}}dt\\
&\sim \frac{le^{l\|\mu_{a,\nu,\gamma}\|+\frac{l^2}{4cn}}}
{2cn\left(\eps-\frac{l}{2cn}\right)}
e^{-cn\left(\eps-\frac{l}{2cn}\right)^2} \to 0
\end{align*}
for $l=O(\log n)$, so $\E[\|M\|^l\I\{Z \in G\}] \leq
(\|\mu_{a,\nu,\gamma}\|+\eps)^l+o(1)$. On the other hand, $\P[Z \notin G]
\leq 2e^{-\frac{n}{2}}$ by Lemma \ref{lemmaGUEwishartconcentration}, so Lemma
\ref{lemmaGUEwisharttailbound} implies
$\E[\|M\|^l\I\{Z \notin G\}]=o(1)$ for $l=O(\log n)$. Hence
$\E[\|M\|^l] \leq (\|\mu_{a,\nu,\gamma}\|+\eps)^l+o(1)$, and taking $\eps \to 0$
concludes the proof.
\end{proof}

\bibliography{references}{}
\bibliographystyle{plain}
\end{document}